\documentclass[11pt, twoside]{amsart}

\usepackage[utf8]{inputenc}

\def\ourtitle{The indecomposable objects in the center of Deligne's category $\uRep S_t$}
\title{The indecomposable objects  in the  center of Deligne's category $\uRep S_t$}

\date{\today}

\author{Johannes Flake}
\address{Max Planck Institute for Mathematics, Vivatsgasse 7, 53111 Bonn, Germany}
\email{flake@mpim-bonn.mpg.de}
\urladdr{https://johannesflake.net}

\author{Nate Harman}
\address{School of Mathematics, Institute for Advanced Study, 1 Einstein Dr, Princeton, NJ 08540}
\email{nharman@math.ias.edu}
\urladdr{https://www.math.ias.edu/~nharman/}

\author{Robert Laugwitz}
\address{School of Mathematical Sciences,
University of Nottingham, University Park, Nottingham, NG7 2RD, UK}
\email{robert.laugwitz@nottingham.ac.uk}
\urladdr{https://www.nottingham.ac.uk/mathematics/people/robert.laugwitz}

\usepackage{import}
\usepackage{amsmath,amsfonts,amsthm,amssymb}
\usepackage[alphabetic, initials]{amsrefs}
\usepackage[english]{babel}
\usepackage{url}
\usepackage{fancyhdr}
\usepackage{graphicx}
\usepackage{microtype}
\usepackage{mathtools}
\usepackage[
colorlinks=true,
linkcolor=black, %
anchorcolor=black,%
citecolor=black, %
urlcolor=black, %
]{hyperref}
\usepackage{cleveref} %
\usepackage{geometry}

\usepackage[all]{xy}
\usepackage{tikz}
\usetikzlibrary{positioning}

\tikzset{
	partition/.style={
      scale=0.5,
      yscale=-1,
      baseline={([yshift=-0.5ex]current bounding box.center)}
    }
}

\tikzset{
    bend/.cd,
    0/.style={},
    1/.style={bend right},
    -1/.style={bend left}
}
\newcommand\makePartPt[1]{({Mod(#1,10)},{(#1-Mod(#1,10))*.1})}
\newcommand\makePartLn[2]{%
\pgfmathtruncatemacro\bend{%
(int(#1/10)==int(#2/10)) ?
(#1<10 ? 1 : -1)*(#1>#2 ? 1 : -1)
: 0%
}
\draw ({mod(#1,10)},{int(#1/10}) to[bend/\bend] ({mod(#2,10)},{int(#2/10)});
}
\newcommand\tp[1] {%
\tikz[partition] {
\draw[white,opacity=0] (1,0)--(1,1); 
\def\j{0}
\foreach \i [remember=\i as \j] in {#1} {
  \ifnum \i>0
    \ifnum \j>0
      \makePartLn{\i}{\j};
    \fi
    \draw[fill] \makePartPt\i circle (2.5pt);
  \fi
} 
}}

\makeatletter
\newcommand{\superimpose}[2]{%
  {\ooalign{$#1\@firstoftwo#2$\cr\hfil$#1\@secondoftwo#2$\hfil\cr}}}
\makeatother

\newcommand{\bunderline}[1]{\mkern1mu\underline{\mkern-1mu#1\mkern-7mu}\mkern7mu }

\newcommand{\mat}[1]{\left(\begin{smallmatrix}#1\end{smallmatrix}\right)}

\newcommand{\op}[1]{\operatorname{#1}}

\newcommand{\ov}[1]{\overline{#1}}

\newcommand{\ab}{\mathrm{ab}}

\newcommand{\cha}{\operatorname{char}}

\newcommand{\colim}{\operatorname{colim}}

\newcommand{\Drin}{\operatorname{Drin}}

\newcommand{\End}{\operatorname{End}}

\newcommand{\gr}{\operatorname{gr}}

\newcommand{\Hom}{\operatorname{Hom}}
\newcommand{\ide}{\operatorname{Id}}

\newcommand{\Img}{\operatorname{Im}}
\newcommand{\Ind}{\operatorname{Ind}}
\newcommand{\Irrep}{\operatorname{Irrep}}
\newcommand{\uInd}{\underline{\operatorname{Ind}}}

\newcommand\lift{\operatorname{Lift}}

\newcommand{\one}{\mathbf{1}}

\newcommand{\Res}{\operatorname{Res}}
\newcommand{\uRes}{\underline{\operatorname{Res}}}
\newcommand{\Rep}{\operatorname{Rep}}
\newcommand{\uRep}{\protect\underline{\mathrm{Re}}\!\operatorname{p}}

\newcommand{\rev}{\operatorname{rev}}
\newcommand{\sgn}{\operatorname{sign}}

\newcommand{\triv}{\operatorname{triv}}

\providecommand{\op}[1]{\operatorname{#1}}

\newcommand{\mC}{\mathbb{C}}

\newcommand{\mQ}{\mathbb{Q}}

\newcommand{\mZ}{\mathbb{Z}}
\newcommand{\mN}{\mathbb{N}}

\newcommand{\mF}{\mathbb{F}}

\newcommand{\cA}{\mathcal{A}}
\newcommand{\cC}{\mathcal{C}}
\newcommand{\cD}{\mathcal{D}}

\newcommand{\cF}{\mathcal{F}}

\newcommand{\cO}{\mathcal{O}}

\newcommand{\cP}{\mathcal{P}}

\newcommand{\cU}{\mathcal{U}}

\newcommand{\cZ}{\mathcal{Z}}

\newcommand{\uW}[1]{\bunderline{W}_{\mkern-7mu#1}}
\newcommand{\uWold}[1]{\bunderline{W}_{\mkern-7mu#1}^{\operatorname{old}}}

\newcommand{\Set}[1]{\left\lbrace #1\right\rbrace}
\newcommand{\isomorph}{\stackrel{\sim}{\longrightarrow}}

\newcommand\kk{\Bbbk}

\newcommand\s\sigma
\newcommand\p\pi

\renewcommand\o\otimes

\tikzset{
	partition/.style={
      scale=0.5,
      yscale=-1,
      baseline={([yshift=-0.5ex]current bounding box.center)}
    }
}

\newcommand\makeDot[2]{\draw[fill] (#1,#2) circle (2.5pt);}

\newcommand\tpid[2][1]{%
\foreach \x in {#2} \draw (\x,#1) -- +(0,1);
}

\newcommand\tpart[2]{\tikz[partition]{%
\foreach \i [count=\c] in {#1} \relax
\draw[white] (1,1)--(1,\c); %
\foreach \i [count=\y] in {#1} {
  \ifnum \i>0 \foreach \x in {1,...,\i} \makeDot{\x}{\y};
  \fi
}
#2
}}

\newtheoremstyle{mystyle}%
  {0.5cm}                   %
  {0.5cm}                   %
  {\normalfont}           %
  {}                      %
  {\itfont\bfseries}  %
  {:}                     %
  {0.3cm}              %
  {\thmname{#1}}

\newtheoremstyle{defstyle}%
  {0.5cm}                   %
  {0.5cm}                   %
  {\normalfont}           %
  {}     %
  {\normalfont\bfseries}  %
  {:}                     %
  {0.3cm}              %
  {\thmname{#1}\thmnumber{ #2}\thmnote{ (#3)}}
\numberwithin{equation}{section}
\newtheorem*{rep@theorem}{\rep@title}
\newcommand{\newreptheorem}[2]{%
\newenvironment{rep#1}[1]{%
 \def\rep@title{#2 \ref{##1}}%
 \begin{rep@theorem}}%
 {\end{rep@theorem}}}
\makeatother

\newtheorem{theorem}{Theorem}[section]

\newtheorem{proposition}[theorem]{Proposition}
\newreptheorem{proposition}{Proposition}
\newtheorem{corollary}[theorem]{Corollary}
\newreptheorem{corollary}{Corollary}
\newtheorem{lemma}[theorem]{Lemma}

\newtheorem{theorem*}{Theorem}
\newreptheorem{theorem}{Theorem}

\theoremstyle{definition}
\newtheorem{definition}[theorem]{Definition}

\theoremstyle{remark}
\newtheorem{example}[theorem]{Example}
\newtheorem{remark}[theorem]{Remark}

\renewcommand{\sectionmark}[1]		%
	{
	\markboth{\small\it \thesection{} #1}{}
	}


\begin{document}

\newcommand\jtodo[1]{{\color{cyan}[#1]}}

\begin{abstract}
We classify the indecomposable objects in the monoidal center of Deligne's interpolation category $\uRep S_t$ by viewing $\uRep S_t$ as a model-theoretic limit in rank and characteristic. We further prove that the center of $\uRep S_t$ is semisimple if and only if $t$ is not a non-negative integer. 
In addition, we identify the associated graded Grothendieck ring of this monoidal center with that of the graded sum of the centers of representation categories of finite symmetric groups with an induction product. We prove analogous statements for the abelian envelope.
\end{abstract}
\maketitle

\section{Introduction}

In the seminal paper \cite{Del}, P.~Deligne constructed symmetric tensor categories $\uRep S_t$, where $t$ can be any complex number, which interpolate the categories of representations over the symmetric groups $S_d$, $d\in \mZ_{\geq 0}$. These categories, and their relatives for other series of groups,   have proven interesting to the study of symmetric tensor categories, as well as to the study of stability phenomena in representation theory, for example, see \cites{CO,CO-Rep-St-ab,Et,Et2,BEH,EHS}. The category $\uRep S_t$ can be constructed using the combinatorics of partitions (see \Cref{sec::diagrammatic}) and has a universal property with respect to Frobenius algebra objects of dimension $t$ in symmetric monoidal categories. %

In addition to giving this combinatorial definition and universal property Deligne observed that for $t$ transcendental  $\uRep S_t$ can be realized as an ultraproduct, a sort of model-theoretic limit, of the categories $\Rep  S_d$, as $d\in \mZ_{\geq 0}$ grows to infinity.  This understanding was extended to all values of $t$ by the second listed author in \cite{Har1}*{Theorem~1.1} by viewing $\uRep S_t$ as a limit of the categories $\uRep_{p} S_{d}$, varying the rank $d$ as well as the characteristic~$p$. We recall the main ideas of this approach in Sections~\ref{sec:deligne}--\ref{sect:int-group}.

\smallskip

In this paper, we apply these ultrafilter techniques to prove several results on the monoidal centers $\cZ(\uRep S_t)$ of Deligne's categories. 
The \emph{monoidal} or \emph{Drinfeld center}  $\cZ(\cC)$ \cites{Maj2,JS} of a monoidal category $\cC$ is a universal construction of a braided monoidal category with a forgetful functor $\cZ(\cC)\to\cC$ and instrumental in the construction of quantum groups and solutions to the quantum Yang--Baxter equation, see e.g.~\cites{Maj1, Kas,BK}.

In \cite{FL}, the first and third listed authors started the investigation of $\cZ(\uRep S_t)$, showed that this is a ribbon category, and obtained invariants of framed links as an application. It was shown that the braided categories $\cZ(\uRep S_t)$ interpolate the braided categories $\cZ(\Rep S_d)$ in the sense that $\cZ(\Rep S_d)$ is the semisimplification of $\cZ(\uRep S_d)$ for $d\in \mZ_{\geq 0}$. In the present paper, we answer several open structural questions about the categories $\cZ(\uRep S_t)$ including a classification of their indecomposable objects and computation of the (associated graded) Grothendieck rings.

\smallskip

We start the paper by providing some required background results on $\uRep S_t$ in \Cref{sec:background}. Many of the results stated there are known to experts but are sometimes not available in the literature. Thus, we have included proofs when appropriate.

\smallskip

A key statement we prove in \Cref{sec:centerlimit} displays $\cZ( \uRep S_t)$ as a model-theoretic limit in rank and characteristic, similar to $\uRep S_t$. For this, we define a bi-filtration on the center, with the filtration layer $\cZ(\uRep S_t)^{\leq m,k}$ being the full subcategory on objects which are in the preimage of $(\uRep S_t)^{\leq m,k}$ under the forgetful functor $\cZ(\uRep S_t)\to \uRep S_t$. Here, $(\uRep S_t)^{\leq m,k}$ consists of objects that are direct summands of objects of the form $X^{\otimes m_1}\oplus \ldots\oplus X^{\otimes m_k}$, with $m_i\leq m$, where $X$ is the generating object of $\uRep S_t$. With respect to this filtration, we show that 
$\cZ(\uRep S_t)^{\le k, m}$ is equivalent to the ultraproduct of the categories $\cZ(\Rep_{\ov{\mF}_{p_i}}S_{n_i})^{\le k,m}$ with partially defined monoidal and additive structures, see \Cref{center-limit}. In particular, this enables us to solve the question of semisimplicity of $\cZ(\uRep S_t)$ raised in \cite{FL}*{Question~3.31}. Recall that $\uRep S_t$ is semisimple if and only if $t\not\in\mZ_{\geq0}$.

\begin{reptheorem}{semisimple}
The category $\cZ(\uRep S_t)$ is semisimple if and only if $t\notin \mZ_{\geq 0}$.
\end{reptheorem}

Next, we construct a $\mC$-linear functor
\begin{equation}
\uInd\colon \cZ(\Rep  S_n)\boxtimes \uRep S_{t-n} \longrightarrow \cZ(\uRep S_t).\label{intro:UInd}
\end{equation}
for every $n\geq0$ and $t\in\mC$.
This functor $\uInd$ is a \emph{separable Frobenius monoidal functor} compatible with the braidings, see \Cref{prop:indcenter}. It enables us to classify the indecomposable objects in $\cZ(\uRep S_t)$. Let $n\geq 0$ be an integer, $\mu$ a singleton free partition of $n$, $Z(\mu)$ the centralizer of an element $\s\in S_n$ of cycle type $\mu$, $V$ an $Z(\mu)$-module, and $U$ an object in $\uRep S_{t-n}$. We denote  the image of the object $\Ind_{Z(\mu)}^{S_n}(V)\boxtimes U$ under $\uInd$ by $\uW{\mu,V,U}$. 
Up to isomorphism the object $\uW{\mu,V,U}$ does not depend on the choice of $\sigma$, and only depends on the isomorphism classes of $U$ and $V$.

\begin{reptheorem}{indec-classification}
$\uW{\mu,V,U}$ is indecomposable if and only if $V$ and $U$ are, and the objects $\uW{\mu,V,U}$ for $n,\mu,V,U$ as above with $V$ and $U$ indecomposable form a complete list of all indecomposable objects in $\cZ(\uRep S_t )$ up to isomorphism.
\end{reptheorem}
The question of classifying the indecomposable objects in $\cZ(\uRep S_t)$ naturally emerged from the paper \cite{FL} but was independently raised by P.~Etingof in his research statement. 
Moreover, in \Cref{cor:blocksZRepSt} we describe the blocks of the category $\cZ(\uRep S_t)$ as 
$$B_{\mu,V,B}=\{\uW{\mu,V,U} \mid U\in B\},$$
where $B$ is a block of $\uRep S_{t-n}$ as classified in \cite{CO}, and the pair $(\mu,V)$ is as above (parametrizing blocks of $\cZ(\Rep S_n)$ not induced from $\cZ(\Rep S_m)$ with $m<n$).

We also classify the indecomposable and indecomposable projective objects in  $\cZ(\uRep^{\ab} S_d)$, for the \emph{abelian envelope} $\uRep^{\ab} S_d$ of $\uRep S_d$, $d\in \mZ_{\geq 0}$, as constructed in \cite{CO-Rep-St-ab}, see \Cref{sec:ZRepab}. Abelian envelopes and their general theory have been receiving an increasing amount of attention recently \cites{EHS, BEO, Cou}. We show that $\cZ(\uRep^{\ab} S_d)$ indeed satisfies the universal property of the abelian envelope of $\cZ(\uRep S_d)$ in the sense of \cite{BEO}, see \Cref{cor:abenvelopeZRep} and \Cref{app:ZC}.

We note that the first paper \cite{FL} on $\cZ(\uRep S_t)$ already contained a general construction of objects via explicit idempotents and half-braidings using the combinatorial description of $\uRep S_t$ by partitions. We identify the objects constructed in \cite{FL} in the image of $\uInd$ in \Cref{sec:comparison} and prove that these objects  generate $\cZ(\uRep S_t)$ as a Karoubian tensor category but are, in general, not  indecomposable.

\medskip

In \Cref{sec:K0}, we address the question of describing the associated graded Grothendieck ring  $\gr K_0^{\oplus} (\cZ(\uRep S_t))$ which was suggested by V.~Ostrik. To this end, we introduce an induction tensor product structure on the direct sum of categories $\cZ(\Rep  S_n)$. Namely, we define in \Cref{sec:towers} the abelian monoidal category
$$\cZ\Rep S_{\geq 0}:=\bigoplus_{n\geq 0} \cZ(\Rep  S_n)$$
with the tensor product of an object $V$ in $\cZ(\Rep  S_n)$ and $W$ in $\cZ(\Rep  S_m)$ given by 
$$V\odot W:=\cZ\Ind_{S_n\times S_m}^{S_{n+m}}(V\boxtimes W)\,\, \in\,\, \cZ(\Rep  S_{n+m}).$$
Note that $\cZ\Rep S_{\geq 0}$ is a tower of centers, \emph{not} the center of a tower of representation categories.
Here, $\cZ\Ind_{S_n\times S_m}^{S_{n+m}}(V\boxtimes W)$ is the usual induction of group representations with additional half-braiding defined in \Cref{centerind-groups}. This induction product on the sum (or \emph{tower}) of centers can be applied to other series of groups and may be of independent interest. More generally, in \Cref{app:frob} we show that induction produces separable Frobenius monoidal functors
$$\cZ(\Rep  G)\longrightarrow \cZ(\Rep  H)$$
if $G\subseteq H$ is a subgroup. 

From \eqref{intro:UInd} we obtain an oplax monoidal functor
\begin{equation}
    \uInd\colon \cZ\Rep S_{\geq 0} \longrightarrow \cZ(\uRep S_t),\quad V\longmapsto \uInd(V\boxtimes\one)\label{intro:oplaxInd},
\end{equation}
see \Cref{sec:oplax}, and a description of the associated graded of the  \emph{additive} Grothendieck ring $K_0^{\oplus}$. 
\begin{reptheorem}{thm:K0}
The functor $\uInd$ from \eqref{intro:oplaxInd} induces an isomorphism of graded rings
$$\gr K_0^{\oplus}(\uInd) \colon K_0(\cZ\Rep S_{\geq 0})\isomorph \gr K_0^\oplus(\cZ(\uRep S_t)),$$
where the associated graded of $K_0^\oplus(\cZ(\uRep S_t))$ is taken with respect to the filtration induced by the filtration $\cZ(\uRep S_t)^{\leq k}$ of $\cZ(\uRep S_t)$.
\end{reptheorem}
An analogous statement holds for the abelian envelope $\cZ(\uRep^\ab S_d)$ if $t=d\in \mZ_{\geq 0}$, see \Cref{thm:K0-ab}.
Computations in $\cZ\uRep S_{\geq 0}$ --- and hence in $K^{\oplus}_0(\uRep S_t)$ --- can be carried out by computing induction of modules over centralizer groups of symmetric groups. Some sample computations are included in \Cref{sec:computations}.

\medskip

A particularly important class of tensor categories are \emph{modular (fusion) categories} which find applications in topological field theory, see \cites{TV} and references therein. In particular, the center $\cZ(\Rep  G)$, for $G$ a finite group and its cocycle twists appear in Dijkgraaf--Witten theory \cite{DPR}. Modular  categories and some of their applications have been generalized to non-semisimple finite tensor categories \cites{FS,BD}. In this generality, a \emph{modular category} is a non-degenerate finite ribbon tensor category.

The categories $\cZ(\uRep S_t)$, for $t$ generic, and $\cZ(\uRep^\ab S_d)$, for $d\in \mZ_{\geq 0}$, are \emph{infinite} analogues of modular categories.
This interpretation follows from \cite{FL}*{Theorem~3.27} where $\cZ(\uRep S_t)$ was shown to be a ribbon category and \Cref{sec:non-deg} where we prove that $\cZ(\uRep S_t)$ and $\cZ(\uRep^\ab S_d)$ are non-degenerate braided tensor categories. We note that these categories are also \emph{factorizable} braided tensor categories by \cite{EGNO}*{Proposition~8.6.3}. In the finite case, non-degeneracy and factorizability of braided tensor categories are equivalent \cite{Shi}.

As $\cZ(\uRep S_t)$ and $\cZ(\uRep^\ab S_d)$ are infinite tensor categories, the concept of modular category and applications to topological field theory have not been developed for these categories. However, we note that these categories satisfy all conditions (beside finiteness) imposed on modular categories. For $t$ generic, the category $\cZ(\uRep S_t)$ is, moreover, semisimple.
Hence, the results on $\cZ(\uRep S_t)$ of this paper and \cite{FL} give interesting infinite yet locally finite analogues of modular categories which are not equivalent to (co)modules over (quasi)-Hopf algebras.

\subsection*{Acknowledgements}

The authors thank P.~Etingof and V.~Ostrik for interesting conversations on the topics of this paper. We also thank C.~Schweigert for posing an interesting question that motivated adding \Cref{sec:non-deg}. The authors further thank V.~Miemietz for organizing the workshop `Representations of monoidal categories and 2-categories' (July 2019, in Norwich) discussions at which lead to the start of this project. The research of J.~F. was supported by the Deutsche Forschungsgemeinschaft (DFG, German Research Foundation) -- Project-ID 286237555 -- TRR 195.
The research of R.~L. is supported by a Nottingham Research Fellowship. We further thank the referee for their careful reading of the manuscript and valuable comments which helped to improve the exposition of the paper.

\section{Background}\label{sec:background}

\subsection{Notation and conventions} \label{sec:conventions}

In the following, $\kk$ denotes a field and $\Rep_\kk G$ the category of finite-dimensional $G$-representations over $\kk$, with $\Rep G=\Rep_\mC G$.
Given a prime $p$, we abbreviate $\Rep_p G=\Rep_{\overline{\mF}_p} G$.

The categories considered in this paper are, at the very least, $\Bbbk$-linear Karoubian rigid monoidal categories. The \emph{Karoubian envelope} of a $\Bbbk$-linear category is the idempotent completion of the closure under finite direct sums, and a category is \emph{Karoubian} if the inclusion into its Karoubian envelope gives an equivalence of categories.
In general, the symbol $\cC\boxtimes\cD$ of two such categories $\cC,\cD$ denotes the external product as in \cite{Kel}, \cite{Ost2}*{Section~2.2} which is the Karoubian envelope of the naive $\Bbbk$-linear tensor product which has objects $X\boxtimes Y$, for $X\in \cC$, $Y\in \cD$. Note that in most cases studied, $\cC$ is a finite semisimple (abelian) category, which implies that $\cC\boxtimes \cD$ is a finite direct sum of copies of $\cD$. If, in addition, $\cD$ is abelian, then $\cC\boxtimes \cD$ is abelian. 

For terminology on monoidal, braided and symmetric monoidal categories we follow \cite{EGNO}. In particular, an \email{(abelian) tensor category} is a  $\kk$-linear rigid monoidal category which is locally finite, abelian, with $\End(\one)=\kk$, for the tensor unit $\one$. A \emph{Karoubian tensor category} shall satisfy the same conditions with one exception: instead of being abelian, we require it only to be additive and idempotent-complete.

\smallskip

An important technical tool used in this paper is that of ultraproducts of categories (see e.g.~\cite{Cru}). For simplicity, we assume that the categories considered here are small so that objects and morphisms form sets. We replace representation-theoretic categories by equivalent small ones, noting that up to equivalence the ultraproduct will not depend on the choice of equivalent small categories. We assume throughout that $\cU$ is a fixed (non-principal) \emph{ultrafilter} on $\mathbb{N}$ and refer the reader to, e.g., \cite{Sch} for generalities on this concept. We may think of $\cU$ as a collection of subsets of $\mN$, each of which contains ``almost all'' numbers. 

Given an ultrafilter $\cU$ and a collection of categories $(\cC_i)_{i\in \mN}$ we can defined their \emph{ultraproduct} $\prod_{\cU}\cC_i$. Its objects are sequences $\prod_\cU V_i$ of objects $V_i$ of $\cC_i$ defined for all $i$ in a set belonging to the ultrafilter. Two such sequences are equal if they agree on all indices of some set belonging to $\cU$. Similarly, morphisms are sequences $\prod_{\cU}f_i\colon \prod_\cU V_i\to \prod_\cU W_i$ of morphisms $f_i\colon V_i\to W_i$ which are defined on a set belonging to the ultrafilter and identified when they agree on some set of the ultrafilter $\cU$. We refer the reader to \cite{Cru} or \cite{Har1} for further explanations. If all categories $\cC_i$ are $\Bbbk_i$-linear, then $\prod_\cU \cC_i$ is linear over the ultraproduct of fields $\prod_\cU \Bbbk_i$. If all $\cC_i$ are monoidal categories, then $\prod_\cU \cC_i$ is naturally a monoidal category. 

We will make essential use of \L o\'s' theorem \cite{Sch}*{Theorem 1.3.2} which allows us to transfer any first order logical statement from the categories $\cC_i$ to their ultraproduct $\prod_\cU \cC_i$. See also \cite{HK}*{Section~1} for examples on how this theorem is used. We further need Steinitz' Theorem that states that an uncountable algebraically closed field is determined, up to isomorphism, by its characteristic and uncountable cardinality \cite{Ste}. This theorem implies the existence of isomorphisms of the ultraproducts of fields $\prod_{\cU}\mC$, or $\prod_{\cU}\ov{\mF}_{p_i}$, for a sequence of primes with $\liminf_i p_i=\infty$, and the complex numbers, see e.g. \cite{Sch}*{Chapter~1}.
In the following, we consider equivalences of monoidal categories linear over an ultraproduct $\prod_{\cU}\Bbbk_i$ of algebraically closed fields, which we regard as equivalences of $\mC$-linear monoidal categories under an isomorphism of fields $\prod_{\cU}\Bbbk_i\cong \mC$ obtained from Steinitz' theorem.

\subsection{Deligne's categories as diagrammatic categories}  
\label{sec::diagrammatic}

In \cite{Del}, Deligne constructs a class of Karoubian  $\Bbbk$-linear symmetric monoidal categories $\uRep S_t$ depending on a parameter $t$ in the field $\Bbbk$ of characteristic zero. It is a well-known observation that every simple complex $S_n$-module appears as a direct summand of a tensor power $X_n^{\otimes k}$ of the $n$-dimensional permutation representation $X_n$ of $S_n$ for some $k\geq1$. In other words, the category $\Rep  S_n$ of finite-dimensional ${\mC} S_n$-modules is the  \emph{Karoubian envelope} of the monoidal category generated by the single object $X_n$,  cf. \cite{Del}*{\S 1.7--1.8}. The morphism spaces $\Hom(X_n^{\otimes k},X_n^{\otimes l})$ are given by the $S_n$-invariants of $X_n^{\otimes (k+l)}$ and can be described combinatorially using the partition algebras $P_k(n)$, see e.g.~\cite{CO}*{Section~2} for details. To define $\uRep S_t$, $n$ is now replaced by a general parameter $t\in \Bbbk$.

Deligne's category $\uRep S_t$ can be constructed using a graphical calculus: it has a distinguished object $X$ represented by a point and its tensor powers $X^{\o k}$ for $k\geq 0$ represented by $k$ points. Morphisms between two such tensor powers $X^{\o k}$ and $X^{\o l}$ are represented using diagrams consisting of $k$ upper points labelled $1,\dots,k$ and $l$ lower points labelled $1',\dots,l'$, and an arbitrary number of strings connecting points. Two such diagrams are considered equivalent, if the partitions of the points $1,\dots,k,1',\dots,l'$ given by the connected components of each string diagram coincide. The morphism spaces between $X^{\o k}$ and $X^{\o l}$ is defined as the free $\kk$-vector space spanned by the equivalence classes of such diagrams, so tensor products and compositions can be defined on diagrams and extended linearly. The following is a typical string diagram representing a morphism in $\uRep S_t$:
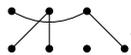
\[
\tp{1,3,14,0,11,2,12,0,13} .
\]

The tensor product of two diagrams is given by stacking the diagrams horizontally. The composition of two diagrams $\pi,\mu$ is achieved by first stacking them vertically, and identifying the lower points of $\pi$ with the upper points of $\mu$. Then these identified points are removed from the string diagram, leaving only the upper points of $\pi$, the lower points of $\mu$, and a number of strings. Each connected component of this string diagram which does not contain upper points of $\pi$ or lower points of $\mu$ (such components may arise when removing the identified points) is removed and the resulting string diagram is multiplied by a factor $t^\ell$, where $\ell\geq 0$ is the number of connected components removed in the process. The category $\uRep S_t$ is now defined as the Karoubian envelope of the category defined on the objects $X^{\otimes k}$, for $k\geq 0$. We refer the reader to \cite{Del} and \cite{CO}*{Section~2} for details on the combinatorial construction of $\uRep S_t$.

In the generic case, i.e.~if $t\notin \mZ_{\geq 0}$, $\uRep S_t$ is a semisimple symmetric tensor category.
If $t=d$ is a non-negative integer, there is an essentially surjective full monoidal functor 
\begin{equation}\label{fiberfunctor}
\cF_d\colon\uRep S_d\longrightarrow \Rep  S_d
\end{equation}
which maps $X$ to the $d$-dimensional standard $S_d$-module $X_d$ and, consequently, $X^{\o k}$ to $(X_d)^{\o k}$ for any $k\geq0$. Choosing a basis $e_1,\dots,e_d$ in $X_d$ provides a basis $(e_{\bf i}=e_{i_1}\o\dots\o e_{i_k})_{i_1,\dots,i_k}$ indexed by tuples $\mathbf i=(i_1,\dots,i_k)$ for $\cF_d(X^{\o k})$, for any $k\geq0$, and the image $\cF_d(\pi)$ of any diagram $\pi$ with $k$ upper and $l$ lower points is the $S_d$-module morphism sending
\begin{equation}\label{Fdeq}
e_{\mathbf i} \, \longmapsto
  \sum_{\mathbf j=(j_1,\dots,j_l)} f(\pi)^{\mathbf i}_{\mathbf j} e_{\mathbf j}
\end{equation}
where the coefficient $f(\pi)^{\mathbf i}_{\mathbf j}$ is $1$ if the indices $i_1,\dots,i_k, j_1,\dots,j_l\in\{1,\dots,t\}$ induce a partition on the upper and lower points of the diagram $\pi$ which refines the one given by the connected components of $\pi$, and $0$ otherwise.
The functors $\cF_d$ allow us to view $\uRep S_t$ as an interpolation category for the classical symmetric tensor categories $\Rep  S_d$.

Further, we recall the (recursive) definition of the morphism $x_{\pi}$ from \cite{CO}*{Equation (2.1)}, 
\begin{align}\label{eqn:xpi}
    x_{\pi}= \pi-\sum_{\tau}x_\tau,
\end{align}
where the sum is taken over all partitions $\tau$ strictly coarser than $\pi$. Note that the set $\{x_\pi\}$, for all partitions of $\{1,\ldots, k,1',\ldots,l'\}$, gives a basis for $\Hom_{\uRep S_t}(X^{\otimes k}, X^{\otimes l})$. If, again, $t=d$, then $\cF_d(x_\pi)$ is the $S_d$-module morphism
\begin{equation}
e_{\mathbf i}\longmapsto\sum_{\mathbf j=(j_1,\dots,j_l)} f'(x_\pi)^{\mathbf i}_{\mathbf j} e_{\mathbf j},
\end{equation}
where now the coefficient $f'(x_\pi)^{\mathbf i}_{\mathbf j}$ is $1$ if the indices $i_1,\dots,i_k, j_1,\dots,j_l\in\{1,\dots,t\}$ induce the same partition on the upper and lower points of the diagram $\pi$ as the one given by the connected components of $\pi$, and $0$ otherwise.

Indecomposable objects in $\uRep S_t$ are classified by partitions, see \cite{CO}*{Section 3.1}:

\begin{theorem} \label{thm::X-lambda}
There is a bijection between partitions $\lambda\vdash n$, for $n\geq 0$, and indecomposable objects $X_\lambda$ of $\uRep S_t$. The object $X_\lambda$ is a direct summand of $X^{\otimes n}$, but not of $X^{\o i}$ for $i<n$.
\end{theorem}

For $t \notin \mathbb{Z}_{\ge 0}$ the objects $X_\lambda$ behave uniformly in $t$. Their dimensions and character values are given by polynomials in $t$, and they can be cut out of $X^{\otimes n}$ by primitive idempotents which are $\mathbb{Q}(t)$-linear combinations of partition diagrams.

At non-negative integral values $t = d \in \mathbb{Z}_{\ge 0}$, some of the rational functions defining these primitive idempotents develop poles and the idempotents no longer exist.  As such, some surviving idempotents which are generically not primitive become primitive at these special values, and the corresponding generically simple objects get ``glued together" in a sense.

As part of their analysis of blocks, Comes and Ostrik described a process of ``lifting'' which takes an indecomposable object $X_\lambda$ at $t=d$, and describes how it splits apart when we deform it to nearby semisimple values of $t$. The combinatorics of this process is completely described by Comes and Ostrik \cite{CO}, but we will just use the following simplified version:

\begin{theorem}[{\cite{CO}*{Proposition~3.10, Proposition~3.12(a), Lemma~5.20}}]\label{lifting}
For every object $X\in\uRep S_d$ there is an object $\lift_d(X)$ in $\uRep S_t$, uniquely defined up to isomorphism, for $t$ in a formal neighborhood of $d$ such that for all $X,Y\in\uRep S_d$,
\begin{enumerate}
    \item $\lift_d(X\oplus Y)\cong\lift_d(X)\oplus\lift_d(Y)$,
    \item $\lift_d(X\o Y)\cong\lift_d(X)\o\lift_d(Y)$,
    \item For any indecomposable object $X_\lambda\in\uRep S_d$, either
    
   \begin{enumerate} \item $\lift_d(X_\lambda)$ remains indecomposable and is still labeled by $X_\lambda$

or

    \item $\lift_d(X_\lambda)$ decomposes as $X_\lambda \oplus X_{\lambda'}$ for a partition $\lambda'$ depending on $d$ and $\lambda$ satisfying $|\lambda| > |\lambda'|$.
    
    \end{enumerate}
       Moreover each $X_{\lambda'}$ arises as a summand of $\lift_d(X_\lambda)$ for at most one partition $\lambda$ with $|\lambda| > |\lambda'|$.
\end{enumerate}

\end{theorem}

\begin{remark}
We believe that over the complex numbers one can check that this lifting operation can be defined for $t$ in an analytic neighborhood of $d$, rather than just a formal one. For our purposes though this will not be necessary, so we will not pursue this direction further.
\end{remark}

\subsection{Deligne's categories as limits in characteristic and rank}\label{sec:deligne}

We now recall \cite{Har1}*{Theorem 1.1} to display $\uRep S_t$, for \emph{any} $t\in \mC$, as an ultraproduct of the representation categories $\Rep_p S_d$, cf.~\Cref{sec:conventions}. The case for $t$ transcendental was already contained in Deligne's seminal paper \cite{Del}. We are interested in three cases of this theorem to describe all cases of Deligne's category $\uRep S_t$ up to equivalence of symmetric monoidal categories (under isomorphisms of fields $\mC\to \mC$ that send one transcendental parameter to the other).

\begin{definition}
For any object $X$ in a Karoubian or abelian tensor category, we denote by $\langle X\rangle$ the full Karoubian tensor subcategory generated by $X$. 
\end{definition}

Note that $\langle X\rangle$ is generated by duals, tensor products, direct sums, and direct summands involving $X$. In the situations we discuss, $X$ will often be a self-dual object.

\begin{theorem}[\cite{Del},\cite{Har1}*{Theorem 1.1}]\label{limitthm}
In each case below, we specify an increasing sequence of positive integers $\mathbf{t}=(t_i)_{i\in \mN}$ and a sequence of fields $(\Bbbk_i)_{i\in \mN}$. We denote by $X$ the object $\prod_{\cU} X_{t_i}$ in $\prod_{\cU}\Rep_{\Bbbk_i}S_{t_i}$, and by $t$ the complex number corresponding to $\prod_{\cU} t_i$ under the respective isomorphism of fields $\prod_{\cU}\Bbbk_i\cong \mC$.
\begin{enumerate}
\item $t$ transcendental: Consider a sequence $\mathbf{t}=(t_i)_{i\in \mN}$ with $\liminf_{i} t_i=\infty$ and the ultraproduct $\prod_{\cU}\Rep  S_{t_i}$.
Then, under an isomorphism of fields $\prod_{\cU}\mC\cong \mC$, $\langle X\rangle\subset\prod_\cU \Rep S_{t_i}$ is equivalent to $\uRep S_t$ as a $\mC$-linear symmetric monoidal category. 
\item $t\in \overline{\mQ}\setminus \mZ_{\geq 0}$, with minimal polynomial $m_t$: We can choose increasing sequences of positive integers $\mathbf{t}=(t_i)_{i\in \mN}$  and primes $\mathbf{p}=(p_i)_{i\in \mN}$ such that $t_i<p_i$ for any $i\in \mN$, satisfying 
\begin{align*}
m_t(t_i)&\equiv 0 \mod{p_i}, &\forall i\in \mN.
\end{align*}
Then, under an isomorphism of fields $\prod_{\cU}\ov{\mF}_{p_i}\cong \mC$, $\langle X \rangle\subset \prod_{\cU}\Rep_{p_i}S_{t_i}$ is equivalent to $\uRep S_t$ as a $\mC$-linear symmetric monoidal category.  
\item $t=d\in \mZ_{\geq 0}$: Set $t_i=p_i+d$ for $p_i$ the $i$-th prime number. Then again, under an isomorphism of fields $\prod_{\cU}\ov{\mF}_{p_i}\cong \mC$, $\langle X \rangle\subset \prod_{\cU}\Rep_{p_i}S_{p_i+d}$ is equivalent to $\uRep S_d$ as a $\mC$-linear symmetric monoidal category.
\end{enumerate}
\end{theorem}
We remark that under the chosen isomorphisms of fields, the respective functors $ \uRep S_t\to \langle X \rangle$ are obtained by the universal property of $\uRep S_t$ by sending the tensor generator of $\uRep S_t$ to the object $X=\prod_{\cU}X_{t_i}$ in the ultraproduct.

Since $\uRep S_t$ is defined in terms of its generating object $X$ there is a natural exhaustive filtration $(\uRep S_t)^{\le k , m}$  encoding the complexity of objects in terms of this generator. Explicitly, $(\uRep S_t)^{\le k , m}$ is the full subcategory of objects which are direct summands of objects $X^{\otimes j_1} \oplus \dots \oplus X^{\otimes j_\ell}$, where each $j_i$ is at most $k$ and the number of terms $\ell$ is at most $m$.  We may similarly filter the categories $\Rep_{p_i} S_{n_i}$ in terms of the defining $n_i$-dimensional representation of $S_{n_i}$, analogously calling the filtered pieces $(\Rep_{p_i} S_{n_i})^{\le k , m}$.

The advantage of these filtrations is that Theorem \ref{limitthm} tells us that $(\uRep S_t)^{\le k , m}$ is equivalent to the ultraproduct of $(\Rep_{p_i} S_{n_i})^{\le k , m}$, with no need to further cut down the ultraproduct. These subcategories $(\uRep S_t)^{\le k , m}$ are not monoidal or even additive, as taking a tensor product or direct sum will possibly land you in a higher term of the filtration, but the ultraproduct does respect those products and sums that are defined.
 
Taking the union over all $m \in \mathbb{N}$ we obtain a coarser filtration $\mathcal{C}^{\le k}$, for either $\mathcal{C}=\uRep S_t$ or $\mathcal{C}= \Rep_{p_i} S_{n_i}$. These filtered pieces are additive and in fact abelian in the case where $t \notin \mathbb{Z}_{\ge 0}$. These are not monoidal subcategories but they satisfy the condition that if $V \in \mathcal{C}^{\le k}$ and $W \in \mathcal{C}^{\le k'}$ then $V\otimes W\in \mathcal{C}^{\le k+k'}$, which makes the Grothendieck ring a filtered ring. 
 
The descriptions of $\uRep S_t$ can be used to transfer any first order statement in the signature of symmetric monoidal categories with a distinguished object $X$ from the classical (modular) representation theory to the limit, i.e.~the interpolation category $\uRep S_t$ using \L o\'s' Theorem. For more details and applications of this philosophy see \cites{Har1, HK}.

\medskip

The combinatorial description of $\uRep S_t$ recalled in \Cref{sec::diagrammatic} can be matched with the above characterization through ultrafilters in \Cref{limitthm}, giving an evaluation of the equivalence $\langle X \rangle \simeq \uRep S_t$ on morphisms. 

Given a partition $\pi$ of ${1,\ldots, k,1',\ldots,l'}$ viewed as a morphism $\pi\colon X^{\otimes k}\to X^{\otimes l}$, $\pi$ corresponds to the ultraproduct $\prod_{\cU} \pi_i$ described in the following. Using the notation $e_{\bf i}$ as in \eqref{Fdeq}, $\pi_i\colon X_{t_i}^{\otimes k}\to X_{t_i}^{\otimes l}$ is defined by
\begin{align}\label{piimage}
\pi_i(e_{\mathbf{i}})&=\sum_{\mathbf{j}}f(\pi)_{\mathbf{j}}^{\mathbf{i}}e_{\mathbf{j}},
\end{align}
for all $k$-tuples $\mathbf{i}$ and all $l$-tuples $\mathbf{j}$ of integers in $\Set{1,\dots,t_i}$, where the coefficient $f(\pi)_{\mathbf{j}}^{\mathbf{i}}$ is the same as in \eqref{Fdeq}, the complex case.

We may extend the assignment $\pi\mapsto \prod_{\cU}\pi_i$ using the fixed isomorphism of fields $\mC\isomorph \prod_{\cU}\ov{\mF}_{p_i}, \alpha_i\mapsto \prod_{\cU}\alpha_i$ by $\alpha\pi\mapsto \prod_{\cU}\alpha_i\pi_i$, for $\alpha\in \mC$. We note that $m\in \mZ$ corresponds to $\prod_\cU m$ under the isomorphism of fields, but for a general complex number $\alpha$, $\alpha_i$ is only defined for almost all $i$.

\subsection{Representations of a fixed group in large characteristic}\label{sect:int-group}
 
In addition to symmetric groups, we will often have an auxiliary finite group $G$ for which we want to compare representations across different large characteristics. 
 
Let $G$ be a fixed finite group, and let $X$ be a faithful representation of $G$ defined over the integers (e.g.~the permutation representation). In a slight abuse of notation we will use $X$ to denote the corresponding base changes to $\Rep G$ as well as to $\Rep_p G$. As $X$, in particular, labels an object in each of the categories $\Rep_{p_i} G$ for any sequence of primes $(p_i)_{i\in\mN}$, we also have an object $\prod_\cU X$ in the ultraproduct $\prod_\cU \Rep_{p_i} G$, which we will also denote by $X$.
The following result of Crumley realizes the characteristic zero representation theory of $G$ as an ultraproduct of the representations of $G$ in large characteristic. 
\begin{theorem}[{\cite{Cru}*{Section 9.5.1}}]\label{rep-limit}
Let $\mathbf{p} = (p_i)_{i\in \mathbb{N}}$ be an arbitrary increasing sequence of primes.   Under an isomorphism of fields $\prod_{\cU}\ov{\mF}_{p_i}\cong \mC$, there exists an equivalence of $\mC$-linear symmetric monoidal categories between $\langle X \rangle \subset \prod_{\cU} \Rep_{p_i}G$ and $\Rep  G$.
\end{theorem}

This result is nice in that it involves an equivalence of categories, and closely parallels the ultraproduct construction of Deligne categories above. However, in this case we can actually be very explicit about what happens at the level of objects, but first let us recall a bit about the representation theory of finite groups in large characteristic.

It was first observed by Dickson in 1902 \cite{Dic} that if $p$ does not divide $|G|$, then the representation theory of $G$ over an algebraically closed field of characteristic $p$ is ``the same'' as over the complex numbers.  Translated into a more modern set-up the following theorem explicitly describes this relationship. We refer to \cite{Ser}*{Part~III} for basic facts about the modular representation theory of finite groups.
 
\begin{theorem}\label{modular-summary}

Suppose $\mathcal{O}$ is the ring of integers in a number field $\Bbbk$ such that every irreducible complex representation of $G$ is defined over $\mathcal{O}$, $p$ is a prime number not dividing $|G|$ or the discriminant of $\mathcal{O}$, and let $\mathfrak{p}$ be a prime ideal of $\mathcal{O}$ lying above $p$. If $V$ is an irreducible complex representation of $G$, choose an integral form $V_\mathcal{O}$ defined over $\mathcal{O}$ and define its reduction modulo $p$ as $V_\mathfrak{p} := V_\mathcal{O} \otimes_{\mathcal{O}} \mathcal{O}/\mathfrak{p}$. 

\begin{enumerate}

\item $V_\mathfrak{p}$ is an (absolutely) irreducible representation of $G$ over $\mathcal{O}/\mathfrak{p}$.

\item The isomorphism class of $V_\mathfrak{p}$ is independent of the choice of the integral form $V_\mathcal{O}$.

\item If $\mathfrak{p}'$ is another prime lying above $p$, then there exists an automorphism $\sigma$ of $\kk$ sending $\mathfrak{p}$ to $\mathfrak{p}'$.  The induced isomorphism $\tilde{\sigma}: \mathcal{O}/\mathfrak{p} \to \mathcal{O}/\mathfrak{p}'$ identifies $V_\mathfrak{p}$ with $V_{\mathfrak{p}'}$.

\item Every irreducible representation of $G$ over $\mathcal{O}/\mathfrak{p}$ arises this way.

\item If $V$ and $W$ are two irreducible complex representations then $V_\mathfrak{p} \cong W_\mathfrak{p}$ if and only if $V \cong W$.

\end{enumerate}

\end{theorem}

We are about ready to give an explicit description of what Crumley's equivalence does, but first let us recall basics about central characters.  If $C \subset G$ is a conjugacy class, then the element $\sum_{g \in C} g \in Z(\mathbb{Z}G)$ defines an endomorphism of the identity functor in $\Rep  G$ and $\Rep_{p}G$ for all primes $p$.  Evaluating the trace of these endomorphisms on a representation $V$ gives the central character of $V$. Over the complex numbers the central character of $V$ is just a rescaling of the ordinary character, and in particular determines $V$ up to isomorphism. 

 Fix $\mathcal{O}$ as in \Cref{modular-summary}, and for each prime $p_i$ as in Theorem \ref{rep-limit} fix a prime ideal $\mathfrak{p}_i$ lying above $p_i$ as well as an algebraic closure $\overline{\mathcal{O}/\mathfrak{p}_i}$ of $\mathcal{O}/\mathfrak{p}_i$.  We may choose our isomorphism $\mathbb{C} \cong \prod_\mathcal{U} \overline{\mathcal{O}/\mathfrak{p}_i}$ such that on $\mathcal{O}$ it is the ultraproduct of the natural quotient maps $\mathcal{O} \to \mathcal{O}/\mathfrak{p}_i$.

\begin{proposition}\label{ultra-reduction}
Under the equivalence of categories defined in Theorem \ref{rep-limit}, the image of an irreducible representation $V \in \Rep  G$ is isomorphic to $\prod_{\cU} V_{\mathfrak{p}_i} \in \prod_{\cU} \Rep_{\mathfrak{p}_i} G$.
\end{proposition}

\begin{proof} A priori we know that under this equivalence of categories $\prod_{\cU} V_{p_i}$ gets identified with some irreducible complex representation $V'$, hence it suffices to check that $V$ and $V'$ have the same central character. $V$ is defined over $\mathcal{O}$, so therefore its central character is as well. Moreover, by construction the central character of $V_\mathfrak{p_i}$ is the reduction modulo $\mathfrak{p_i}$ of the central character of $V$.  Since we chose our identification $\mathbb{C} \cong \prod_\mathcal{U} \overline{\mathcal{O}/\mathfrak{p}_i}$ to identify $\mathcal{O}$ with the product of its reductions modulo $\mathfrak{p}_i$ we see that indeed these characters agree. 
\end{proof}

\subsection{The Grothendieck ring of \texorpdfstring{$\uRep S_t$}{Rep St}} \label{sec:K0RepSt}

Given a $\Bbbk$-linear additive monoidal category $\cC$, let $K_0^{\oplus}(\cC)$ denote its \emph{additive Grothendieck ring}, i.e., the quotient the free ring generated by all isomorphism classes of objects in $\cC$ by the ideal of relations given through direct sums and tensor products. If we denote by $[Y]$ the symbol of an object $Y$ of $\cC$ inside of the ring $K_0^\oplus(\cC)$, these relations are
$$
[Y_1]+[Y_2] = [Y_1\oplus Y_2] ,\qquad
[Y_1][Y_2] = [Y_1\o Y_2]
$$
for all $Y_1,Y_2\in\cC$. 

The Grothendieck ring of $\uRep S_t$ is a filtered ring: Recall the filtration on the category $\uRep S_t$ from \Cref{sec:deligne}. Here, an object $Y\in\uRep S_t$ is in $(\uRep S_t)^{\leq k}$, for $k\geq0$, if $Y$ is isomorphic to a direct summand of a sum of objects $X^{\o l}$ with $l\leq k$, where $X$ is the tensor generator of $\uRep S_t $. This defines a filtration of the ring $K_0^\oplus(\uRep S_t )$.

Let us denote the irreducible complex $S_n$-module corresponding uniquely up to isomorphism to a partition $\lambda\vdash n$ by $S^\lambda$.

\begin{lemma}[{\cite{Del}*{Proposition 5.11},\cite{CO}*{Proposition~3.12},\cite{Har2}*{Theorem~3.3}}] 
\label{lem::K0-RepSt} Sending $[X_\lambda]\mapsto [S^\lambda]$ induces an isomorphism $\gr K_0^\oplus(\uRep S_t)\cong\bigoplus_{n\geq0} K_0(\Rep  S_n)$, where the product on the right-hand side is given by induction, or equivalently, by Littlewood--Richardson coefficients. %
\end{lemma}

As the ring on the right-hand side is a graded ring which does not depend on $t$, this result exhibits the Grothendieck ring of $\uRep S_t$ as a filtered deformation of that ring.

\subsection{Induction and restriction between Deligne categories}\label{ind-res-sect}

The universal property of $\uRep S_{t+k}$ gives a natural exact symmetric monoidal \emph{restriction functor} 
$$
 \uRes^{S_{t+k}}_{S_k \times S_t}\colon \uRep S_{t+k} \longrightarrow \Rep  S_k \boxtimes \uRep S_t 
\ ,
$$
which sends the defining object $X_{t+k}$ to $X_k \boxtimes \mathbf{1} \oplus \mathbf{1} \boxtimes X_t$. If $t = d$ is a non-negative integer this descends to the ordinary restriction functor from  $\Rep S_{d+k}$ to $\Rep  S_k \boxtimes \Rep S_d \cong \Rep (S_k \times S_d)$, under the fiber functors $\cF_{d+k}$, $\cF_d$ to the semisimple categories from \eqref{fiberfunctor}.

Etingof (\cite{Et}, Section 2.3) considered \emph{induction} and \emph{co-induction} functors
$$
\uInd_{S_k \times S_t}^{S_{t+k}}, \widetilde{\uInd}_{S_k \times S_t}^{S_{t+k}} \colon \Rep  S_k \boxtimes \uRep S_t \longrightarrow \uRep S_{t+k}
\ ,
$$
defined a priori as the left and right adjoints to the restriction functor defined above. Etingof also considered restriction to a product of two Deligne categories, and in that case one needs to pass to an ind-completion in order to define these adjoints, but for our purposes we will only need the finite versions. 

It was observed in \cite{HK} that if we think of $\uRep S_t$ and $\Rep  S_k$ as the model-theoretic limits of categories $\Rep_{p_i} S_{n_i}$ and $\Rep_{p_i}S_k$ respectively (in the sense described above), then these induction and co-induction functors are limits of the ordinary induction and co-induction functors
$$
\Ind_{S_k \times S_{t_i}}^{S_{t_i+k}}, \widetilde{\Ind}_{S_k \times S_{t_i}}^{S_{t_i+k}} \colon \Rep_{p_i} S_k\boxtimes \Rep_{p_i} S_{t_i} \longrightarrow \Rep_{p_i}S_{k+t_i}
$$
corresponding to the embedding of $S_k \times S_{t_i}$ into $S_{k+t_i}$ for each $t_i$. Moreover, since induction and co-induction are naturally isomorphic for representations of finite groups (over any field) it follows that these Deligne category induction and co-induction functors are naturally isomorphic as well and we can view them as a single two-sided adjoint to the restriction functor which we will refer to just as induction. 

These induction functors are also well behaved with respect to the filtrations $(\uRep S_t)^{\le m}$ defined earlier.  In particular we have:
\begin{align}\label{ind-fitration-comp}
\uInd_{S_k \times S_t}^{S_{t+k}}\colon \Rep  S_k \boxtimes (\uRep S_t)^{\le m} \longrightarrow \uRep S_{t+k} ^{\le m+k}
\ ,
\end{align}
moreover, this $m \to m+k$ shift in the filtration is optimal in a  strong sense: If $V$ is an object of $\Rep  S_k \times (\uRep S_t)^{\le m}$ that does not lie in $\Rep  S_k \times (\uRep S_t)^{\le m-1}$, then $\uInd_{S_k \times S_t}^{S_{t+k}}(V)$ lies in $(\uRep S_{t+k}) ^{\le m+k}$ but not in $(\uRep S_{t+k}) ^{\le m+k-1}$.

If $G \subset S_k$ is a subgroup then we may further restrict from $\uRep S_{t+k}$ to $\Rep  G \boxtimes \uRep S_{t} $. By the same reasoning as above, this also admits a two-sided adjoint induction functor

$$
\uInd_{G \times S_t}^{S_{t+k}}\colon \Rep  G \boxtimes \uRep S_t \longrightarrow \uRep S_{t+k}
\ .
$$
We note that $\uInd_{G \times S_t}^{S_{t+k}}$ is naturally isomorphic to $\uInd_{S_k \times S_t}^{S_{t+k}} \circ (\Ind_G^{S_k} \boxtimes \text{Id}_{\uRep S_t })$, where we first perform ordinary induction from $G$ to $S_k$ and then perform Deligne category induction. In particular, it shifts up the filtration by the same amount.

\medskip

Given a partition $\lambda$ of size $n$ we recall that $S^\lambda$ denotes the corresponding irreducible representation of $S_n$. We will use $X_\lambda$ to denote the corresponding indecomposable object of $\uRep S_t $, see \Cref{sec:conventions}. Note that by \Cref{thm::X-lambda}, $X_\lambda\in (\uRep S_t)^{\leq |\lambda|}$ but not in $(\uRep S_t)^{\leq |\lambda|-1}$.

If $\lambda = (\lambda_1, \lambda_2, \dots \lambda_\ell)$ is a partition, then for $n$ sufficiently large define the \emph{padded partition} $\lambda[n] = (n- |\lambda|, \lambda_1, \lambda_2, \dots \lambda_\ell)$. In terms of Young diagrams, this padding operation just adds a long first row to $\lambda$ to make it have $n$ total boxes. 
The relevance for our purposes is that Comes and Ostrik showed that
if $n-|\lambda|\geq \lambda_1$ (i.e., if $\lambda[n]$ defines a Young diagram), then $X_\lambda \in \uRep S_n$ gets mapped to $S^{\lambda[n]} \in \Rep  S_n$ under the specialization functor $\mathcal{F}_n\colon\uRep S_n\to\Rep  S_n$ \cite{CO}*{Proposition 3.25}.

For a partition $\lambda$ let $\lambda - h.s.$ denote the set of partitions $\mu$ which can be obtained from $\lambda$ by removing a horizontal strip, that is, by removing at most one box from each column. The following Deligne category Pieri rule gives us an upper triangularity property between the simple objects $X_\lambda$ and certain easy to work with induced objects. 

\begin{lemma} \label{lem:inddec} \  \begin{enumerate}
\item For $t \not\in \mathbb{Z}_{\ge 0}$ or for $t \in \mathbb{Z}$ with $t \gg |\lambda|$
$$\uInd_{S_k \times S_t}^{S_{t+k}}(S^{\lambda} \boxtimes \mathbf{1}) = \bigoplus_{\mu \in \lambda-h.s.} X_\mu
.
$$
In particular, the right hand side is $X_\lambda$ plus terms $X_\mu$ with $|\mu| < |\lambda|$.

\item For $t \in \mathbb{Z}_{\ge 0}$ the object $\uInd_{S_k \times S_t}^{S_{t+k}}(S^{\lambda} \boxtimes \mathbf{1})$ decomposes as a multiplicity-free direct sum of indecomposable objects $X_\mu$, such that:

\begin{enumerate}
    \item $X_\lambda$ occurs with multiplicity one.
    
    \item Each $X_\mu$ that appears has $\mu \in \lambda-\text{h.s.}$ (but not all such $\mu$ need appear).
\end{enumerate}

\end{enumerate}

\end{lemma}

\begin{proof}

The ordinary Pieri rule for symmetric groups tells us that in characteristic $0$ or $p > n+k$
$$\Ind_{S_k \times S_n}^{S_{n+k}}(S^{\lambda} \boxtimes \mathbf{1}) = \bigoplus_{\mu \in \lambda-h.s.} S^{\mu[n]}$$
Part 1 follows since $S^{\mu[n]}$ corresponds to the object $X_\mu$, under both the ultraproduct identification and under the quotient functor from $\uRep S_n$ to $\Rep  S_n$.

Now suppose $t=d \in \mathbb{Z}_{\ge 0}$. A priori one knows that $\Ind_{S_k \times S_d}^{S_{d+k}}(S^{\lambda} \boxtimes \mathbf{1})$ decomposes as a direct sum of indecomposable objects $X_\nu$ with some multiplicities $c_\nu$. If we apply the Comes--Ostrik lifting operator $\lift_d$ to this sum we must obtain the answer from Part (1).

$$\lift_d(\bigoplus_{\nu} c_\nu X_\nu) = \bigoplus_{\mu \in \lambda-h.s.} X_\mu $$

Comes and Ostrik's description of lifting (as summarized in Theorem \ref{lifting}) tells us that $X_\lambda$ appears as a direct summand of $\lift_{d}(X_\nu)$ for $\nu = \lambda$ and in that case, at most one other partition $\lambda'$,  with $|\lambda'| > |\lambda|$ appears as a summand giving $\lift_{d}(X_{\lambda'}) = X_{\lambda'} \oplus X_\lambda$.
Since $\lambda$ is the largest partition appearing in $\lambda-h.s.$ and it appears with multiplicity one we see that $X_\lambda$ must occur at $t=d$ with multiplicity one, proving part (1). 

To show part (2) we then inductively apply the same logic as above to the next largest partitions $\mu \in \lambda-h.s.$.  Note that as we repeat this argument for some $\mu \in \lambda-{h.s.}$ the term $X_\mu$ might come from $\lift_d(X_{\mu'})$ for some larger $\mu' \in \lambda-{h.s.}$ already accounted for, in which case $X_\mu$ does not appear in the induction at $t=d$.
\end{proof}

Note that this lemma gives us an alternative way of characterizing the indecomposable object $X_\lambda$ as the unique direct summand of $\Ind_{S_k \times S_t}^{S_{t+k}}(S^{\lambda} \boxtimes \mathbf{1})$ not occurring as a summand of $\Ind_{S_j \times S_t}^{S_{t+j}}(S^{\mu} \boxtimes \mathbf{1})$ for any partition $\mu$ with $|\mu| < |\lambda|$.

Once one has established the Deligne category Pieri rule for induction, one can iteratively compute more general induced representations combinatorially.  In particular, if we only keep track of the leading order terms one obtains:

\begin{corollary}[See Section 2.1 of \cite{Har2}]\label{cor:inddec}

$$\uInd_{S_k \times S_t}^{S_{t+k}}(S^{\lambda} \boxtimes X_\mu) = \bigoplus_\nu c_{\lambda,\mu}^{\nu} X_\nu \oplus \{\text{terms } X_\tau \text{ with } |\tau| < |\lambda|+|\mu| \}$$ 
where $c_{\lambda,\mu}^{\nu}$ denotes a Littlewood--Richardson coefficient.
\end{corollary}

\subsection{The abelian envelope of Deligne's categories}\label{sec:RepSdab}

An (abelian) multitensor category $\cD$ is called an \emph{abelian envelope} of a Karoubian rigid monoidal tensor category $\cC$ if it contains $\cC$ and for any multitensor category $\cD'$, the category of tensor functors $\cD\to\cD'$ is equivalent to the category of faithful monoidal functors $\cC\to\cD'$ by restriction. If it exists, the abelian envelope is unique up to equivalence. We refer to   \cite{BEO} and \cite{EHS}*{Section~9} for details on the abelian envelope and its universal property. 

For $d\in \mZ_{\geq 0}$, the abelian envelope $\uRep^\ab S_d$ of $\uRep S_d$ (\cite{BEO}*{Example 2.45(1)}) has several explicit constructions.
The first construction uses a symmetric monoidal functor $\uRep S_d\to \uRep S_{-1}$ and displays the abelian envelope as a category of representations over an algebraic group object internal to the semisimple category $\uRep S_{-1}$. The second construction introduces a $t$-structure on complexes of objects in $\uRep S_d$ and displays the abelian envelope as the heart of this $t$-structure --- see \cite{CO-Rep-St-ab} for details on these constructions.  

A third construction, which we will employ, is given using ultrafilters. For this, we consider the ultraproduct of general modules over $S_{p_i+d}$ in finite characteristic $p_i$ (rather than just those in $\langle X_{p_i+d}\rangle$) in \Cref{limitthm}(c) and obtain the following modification of the ultraproduct description for the abelian envelope, which is obtained as the closure of $\langle X\rangle$ under subquotients.

\begin{theorem}[{\cite{Har1}*{Theorem~1.1(b)}}]\label{limitthm2}
Set $t_i=p_i+d$ for $p_i$ the $i$-th prime number. Then there are equivalences of additive categories 
$$(\uRep^\ab S_d)^{\leq k}\cong \prod_{\cU}(\Rep_{p_i}S_{p_i+d})^{\leq k}$$
which provide, under an isomorphism of fields $\prod_{\cU}\ov{\mF}_{p_i}\cong \mC$, an equivalence of $\mC$-linear symmetric tensor categories between
$\colim_{k\geq 0}\prod_{\cU}(\Rep_{p_i}S_{p_i+d})^{\leq k}$ and $\uRep^\ab S_d$.
\end{theorem}

It follows from \Cref{limitthm2} (or can be deduced directly from the universal property of $\uRep^\ab S_d$) that the induction functor from \Cref{ind-res-sect} extends to the abelian envelope such that 
\begin{align}
    \vcenter{\hbox{\xymatrix{
    \Rep  S_k\boxtimes \uRep S_{d-k}\ar[rr]^-{\uInd_{S_k\times S_{d-k}}^{S_d}}\ar@{^{(}->}[d]&&\uRep S_d\ar@{^{(}->}[d]\\
    \Rep  S_k\boxtimes \uRep^\ab S_{d-k} \ar[rr]^-{\uInd_{S_k\times S_{d-k}}^{S_d}}&&\uRep^\ab S_d
    }}}
\end{align}
is a commutative diagram of  $\mC$-linear functors. This functor is again left and right adjoint to restriction and hence exact.

Comes and Ostrik showed (\cite{CO-Rep-St-ab}*{Proposition 2.9,  Corollary 4.6}) that each block of $\uRep^\ab S_d$ is equivalent to a block in the category of representations of quantum $SL(2)$, with the objects in $\uRep S_d$ corresponding to tilting objects. As a consequence, $\uRep^\ab S_d$ itself forms a highest weight category with the indecomposable objects $X_\lambda \in \uRep S_d$ as the indecomposable tilting objects.  The simple objects $D^\lambda$ of $\uRep^\ab S_d$ are again indexed by partitions, with $D^\lambda$ appearing as a composition factor of $X_{\lambda}$ with multiplicity one, and all other composition factors $X_{\lambda}$ are of the form $D^\mu$ with $\mu$ satisfying $|\mu| < |\lambda|$. 

Under the ultrafilter identification the indecomposable module $X_\lambda$ corresponds to the so-called Young modules $Y(\lambda[p_i+d])$, which arise as direct summands of the permutation representation. The simple objects $D^\lambda$ of $\uRep^\ab S_d$ correspond to ultraproducts of simple objects $D^{\lambda[p_i+d]}$.  We note that while in general $\Rep  S_n$ is not a highest weight category, the truncated categories $(\Rep  S_n)^{\le k}$ with $k$ less than the characteristic $p$ are highest weight with Specht modules as standard objects and Young modules as tilting objects. One can check that the ultraproduct of these highest weight structures defines a highest weight structure on $\uRep^\ab S_d$ agreeing with the one defined by Comes and Ostrik.

Induction does not preserve semisimplicity in general so we cannot expect a version of Corollary \ref{cor:inddec} for induction of simple objects to hold as a direct sum decomposition in $\uRep^\ab S_d$. We can however relax it by passing to the Grothendieck ring and instead keeping track of composition multiplicities. 

For an abelian category $\cA$, the \emph{(abelian) Grothendieck ring} $K_0(\cA)$ is the quotient of the additive Grothendieck ring $K_0^{\oplus}(\cA)$ from \Cref{sec:K0RepSt} by relations obtained from short exact sequences. If $\cA$ is semisimple, then $K_0(\cA)=K_0^{\oplus}(\cA)$.

\begin{corollary}\label{cor:Ko(St-ab)}
In $K_0(\uRep^\ab S_d)$ one has 
$$[\uInd_{S_k \times S_d}^{S_{d+k}}(S^{\lambda} \boxtimes D^\mu)] = \sum_\nu c_{\lambda,\mu}^{\nu} [D^\nu] + \{\text{terms } [D^\tau] \text{ with } |\tau| < |\lambda|+|\mu| \}$$ 
where $c_{\lambda,\mu}^{\nu}$ denotes a Littlewood--Richardson coefficient.
\end{corollary}

\begin{proof}
This follows immediately from Corollary \ref{cor:inddec} by passing to the Grothendieck ring and then substituting $[X_\lambda] = [D^\lambda] + \{l.o.t\}$ everywhere and collecting all of the lower order terms to one side.
\end{proof}

We have the following analog of \Cref{lem::K0-RepSt} for the abelian Grothendieck ring of $\uRep S_d$:

\begin{lemma}
\label{lem::K0-RepSt-ab} Sending $[D^\lambda]\mapsto [S^\lambda]$ induces an isomorphism of rings from $\gr K_0(\uRep^\ab S_d)$ to $\bigoplus_{n\geq0} K_0(\Rep  S_n)$, where the product on the right-hand side is given by induction, or equivalently, by Littlewood--Richardson coefficients. %
\end{lemma}

\begin{proof} The proof will be via a chain of isomorphisms.

First note that the upper triangularity property $[X_\lambda] = [ D^\lambda] + \{ \text{l.o.t.}\}$ implies that these $[X_\lambda]$ form a basis for $K_0(\uRep^\ab S_d)$, and therefore the inclusion $K_0^\oplus(\uRep S_d) \hookrightarrow K_0(\uRep^\ab S_d)$ is in fact an isomorphism. Moreover this upper triangularity property implies that the assignment $[D^\lambda] \to [X_\lambda]$ defines an isomorphism of associated graded rings $ \gr K_0(\uRep^\ab S_d) \cong \gr K_0^\oplus(\uRep S_d)$.

Next, the first two properties of lifting in \Cref{lifting} say that the map $[X] \mapsto [\lift_d(X)]$ defines a ring homomorphism $K_0^\oplus(\uRep S_d) \to K_0(\uRep S_t)$ at a nearby transcendental value of $t$.  The third property in Theorem \ref{lifting} is again an upper triangularity property, which implies that this homomorphism is an isomorphism, but moreover that $[X_\lambda] \mapsto [X_\lambda]$ defines an isomorphism of associated graded rings $\gr K_0^\oplus(\uRep S_d) \to \gr K_0(\uRep S_t)$ at generic $t$.

Finally we use the isomorphism $\gr K_0(\uRep S_t) \cong \bigoplus_{n\geq0} K_0(\Rep  S_n)$ of \Cref{lem::K0-RepSt}, which sends $[X_\lambda]$ to $[S^\lambda]$. Composing these three isomorphisms gives the desired result.
\end{proof}

\subsection{The projective objects in the abelian envelope}

We now turn to describing the projective objects in $\uRep^\ab S_d$ for later use. Key to understanding of the projectives will be the projective cover of the tensor unit.

\begin{lemma}\label{lem:proj-ultra}
Let $P\cong \prod_{\cU} P_i$ be an object in $\uRep^\ab S_d$. Then $P$ is projective in $\uRep^\ab S_d$ if and only if there exists an integer $k_0\geq 0$ such that for each integer $k\geq k_0$, almost all $P_i$ are projective objects in $(\Rep_{p_i} S_{t_{i}})^{\leq k}$.
\end{lemma}

\begin{proof}
By virtue of $P$ being an object of $\uRep^\ab S_d$, there exists $k_0\geq 0$ such that $P\in (\uRep^\ab S_d)^{\leq k_0}$. In particular, $P$ is contained in $P\in (\uRep^\ab S_d)^{\leq k}$ for any $k\geq k_0$. Hence, for fixed $k$, almost all $P_i$ are contained in $(\Rep_{p_i} S_{t_{i}})^{\leq k}$.
Assume that $P$ is projective.
The property that $P$ is projective is defined by the functor $\Hom(P,-)$ being right exact, as a functor on $(\uRep^\ab S_d)^{\leq k}$, which can be expressed as a first order property. Thus, $P_i$ is a projective object in $(\Rep_{p_i} S_{t_{i}})^{\leq k}$, for almost all $i$.

Conversely, again using \L o\'s' Theorem, we see that if $P_i$ is a projective object in $(\Rep_{p_i} S_{t_{i}})^{\leq k}$, for almost all $i$, then $P$ is projective in $(\uRep^\ab S_d)^{\leq k}$. Since $\uRep^\ab S_d$ is the filtered colimit of all of these categories, projectivity of $P$ holds in the entire category $\uRep^\ab S_d$. 
\end{proof}

Recall the indecomposable objects $X_{\lambda}$ of $\uRep S_t$, see \Cref{sec::diagrammatic} and consider the case $\lambda=(d+1)$.

\begin{lemma}\label{lem:projcover}
The indecomposable object $X_{(d+1)}\in (\uRep^\ab S_d)^{\leq d+1}$ is the projective cover of the tensor unit $\one$.
\end{lemma}
\begin{proof}
Recall from \Cref{limitthm2} that $(\uRep^\ab S_d)^{\leq k}$ is equivalent to the ultraproduct of the categories $(\Rep_{p_i}S_{p_i+d})^{\leq k}$. Observe that $\uInd$ is left adjoint of a functor preserving epimorphisms and thus preserves projective objects, therefore any object of the form
$\Ind_{S_{d+1}\times S_{p_i-1}}^{S_{p_i+d}}(V\boxtimes W)$
is projective in $\Rep_{p_i}S_{p_i+d}$ as it is induced from a semisimple category where all objects are projective. In particular, observe that $$P_i:=\Ind_{S_{d+1}\times S_{p_i-1}}^{S_{p_i+d}}(\one\boxtimes \one)\in (\uRep^\ab S_d)^{\leq d+1}$$
and that the surjective morphisms $P_i\to \ov{\mF}_{p_i}, g\otimes (1\boxtimes 1)\mapsto 1$ induce a surjective morphism $P\twoheadrightarrow \one$, for $P=\prod_{\cU}P_i$.

By \Cref{lem:inddec}, $P\cong X_{(d+1)}\oplus \text{(l.o.t)}$. The indecomposables appearing as lower order terms are of the form $X_{(k)}$, with $k\leq d+1$. The combinatorial description of objects in the  block of $X_{(d+1)}$ from \cite{CO-Rep-St-ab}*{Theorem~2.6} implies that of the objects $X_{(k)}$, only $X_{(0)}=\one$ belongs to this block and $\Hom(P,Y)=0$ for objects $Y$ from other blocks. However, $\one$ is not projective as $\uRep^\ab S_d$ is not semisimple. Thus, $X_{(d+1)}$ is the projective cover of $\one$.
\end{proof}

\begin{proposition}\label{prop:proj-filt}
Let $X\in (\uRep^\ab S_d)^{\leq k}$. Then $X$ is projective in $(\uRep^\ab S_d)^{\leq (k+d+1)}$ if and only if $X$ is projective in $\uRep^\ab S_d$.
\end{proposition}
\begin{proof}
Given an indecomposable object $X$ in $(\uRep^\ab S_d)^{\leq k}$, we can use \Cref{lem:projcover} to see that the projective cover of $X$ is contained in $(\uRep^\ab S_d)^{\leq (k+d+1)}$. Indeed, $X\otimes X_{(d+1)}$ is a projective object in $(\uRep^\ab S_d)^{\leq(k+d+1)}$ with an epimorphism to $X\otimes\one\cong X$. Now, $X$ is projective in an additive subcategory $\cC$ of $\uRep^\ab S_d$ if and only if it its projective cover is contained in $\cC$ and $X$ is isomorphic to its projective cover. By the above observation and \Cref{lem:proj-ultra}, this condition holds in $\cC=\uRep^\ab S_d$ if and only if it holds in $\cC=(\uRep^\ab S_d)^{\leq (k+d+1)}$.
\end{proof}

The advantage of the above proposition is that projectivity of an object in $\uRep^\ab S_d$ can be checked using the components of an ultrafilter presentation, as well as the following observation.

\begin{corollary}
All projective objects of $\uRep^\ab S_d$ are contained in the subcategory $\uRep S_d$.
\end{corollary}
\begin{proof}
This follows directly from \cite{CO-Rep-St-ab}*{Remark~4.8} since for any object $P$ of $\uRep^\ab S_d$ there exists a surjective map $\bigoplus_{i=1}^m X^{\otimes n_i}\twoheadrightarrow P$. Hence, if $P$ is projective, it is a direct summand of an object in $\uRep S_d$.
\end{proof}

\subsection{Indecomposable Yetter--Drinfeld modules over \texorpdfstring{$S_n$}{Sn} in arbitrary characteristic}\label{YDmods}
In this section, let $\Bbbk$ be an algebraically closed field and $G$ a finite group.
We now turn to background results on the monoidal center $\cZ(\Rep_\kk G)$ of the category $\Rep_\kk G$. We employ the equivalent description of this braided tensor category as \emph{Yetter--Drinfeld modules} \cite{Yet}*{Definition~3.6}, i.e.~$G$-graded $G$-representations $V=\bigoplus_{g\in G}V_g$ such that $g\cdot V_h=V_{ghg^{-1}}$, see e.g.~\cite{Maj1}*{Proposition~7.1.6}. If $v\in V_g$, we write $|v|=g$ for the \emph{degree} of $v$. Equivalently, $\cZ(\Rep_\kk G)$ can be described as modules over the \emph{Drinfeld double} $\Drin(G)$ of $G$ \cite{Dri}.
We recall the following result found in \cite{DPR} for $\Bbbk=\mC$ and in \cite{Wit}*{Corollary 2.3} for general characteristic. %

\begin{theorem}[Dijkgraaf--Pasquier--Roche, Witherspoon]\label{DPR-W-Thm}
Let $G$ be a finite group and $\Bbbk$ be an algebraically closed field. A complete list of indecomposable (respectively, irreducible) objects in $\cZ(\Rep_\Bbbk G)$ is given by modules of the form $W_{\s,V}=\Ind_Z^G(V)$ where $\s$ is a representative of a conjugacy class of elements in $G$ and $V$ is an indecomposable (respectively, irreducible) module over the centralizer $Z=Z(\s)$ of $\s$ in $G$.
\end{theorem}

In fact, for each $\s\in G$, with $Z=Z(\s)$, we have functors
$$\Ind_Z^G\colon \Rep_\Bbbk Z\longrightarrow \cZ(\Rep_\Bbbk G),\qquad  V\longmapsto \Ind_Z^G(V).$$
The half-braiding on $W_{\s,V}:=\Ind_Z^G(V)$ is given by $c_{W}((g\otimes v)\otimes w)= (g\s g^{-1})\cdot w\otimes (g\otimes v)$, for any $G$-module $W$. 
The $G$-grading of the associated Yetter--Drinfeld module $\Ind_Z^G(V)$ here is given by $\delta(g\otimes v)=g\s g^{-1}\otimes v$. This grading does not depend on $V$ so it is clear that any morphism of $\Bbbk Z$-modules $f\colon V\to V'$ induces one in $\cZ(\Rep_\Bbbk G)$, namely, $\Ind_{Z}^{G}(f)\colon W_{\s,V}\to W_{\s, V'}$. The constructions are independent of the choice of a representative $\s$ of a conjugacy class up to isomorphism.

The following lemma is easily seen from the description of $\cZ(\Rep_\Bbbk G)$ in terms of Yetter--Drinfeld modules over $G$.

\begin{lemma}\label{lem:cZHoms}
The functors $\Ind_Z^G$ are full and faithful, i.e.~there are isomorphisms
$$\Hom_{\Rep_\Bbbk Z}(V,W)\isomorph \Hom_{\cZ(\Rep_\Bbbk G)}(\Ind_{Z}^G V,\Ind_Z^G W),$$
for any $\Bbbk Z$-modules $V$ and $W$. 

If $\s$ and $\tau$ are not conjugate, $V$ is a $\Bbbk Z(\s)$-module, and $W$ is a $\Bbbk Z(\tau)$-module, then 
$$\Hom_{\cZ(\Rep_\Bbbk G)}(\Ind_{Z(\s)}^G V,\Ind_{Z(\tau)}^G W)=\{ 0\}.$$
\end{lemma}

Hence, as an abelian category, $\cZ(\Rep_\Bbbk G)$ decomposes as a direct sum  (as defined in \cite{EGNO}*{Section~1.3})
$$\cZ(\Rep_\Bbbk G)\cong \bigoplus_{[g]\in G^G} \Rep_\Bbbk Z(g),$$
where $g$ ranges over a set of representatives of the conjugacy classes of $G$.

A direct consequence of this observation is the following: let $W$ be a Yetter--Drinfeld $G$-module over any field $\Bbbk$, and let $W_{g,V}$ be the simple Yetter--Drinfeld $G$-module obtained from $g$ in $G$ and a simple module $V$ over the centralizer $Z(g)$ of $g$ in $G$. Then the multiplicity $[W:W_{g,V}]$ in $\cZ(\Rep_\Bbbk G)$ equals the multiplicity $[W_g:V]$ in $\Rep_\Bbbk Z(g)$, where $W_g$ is the homogeneous subspace of $W$ of degree $g$.

\begin{lemma}\label{proj-reps}
$W_{\s,P}$ is projective in $\cZ(\Rep_\Bbbk G)$ if and only if $P$ is projective in $\Rep_\Bbbk Z(\s)$.
\end{lemma}
\begin{proof} 
For $\s\in G$, consider the regular $\Bbbk Z$-module and note that $\Ind_Z^G(\Bbbk Z)=\Bbbk G$. As an object in $\cZ(\Rep_\Bbbk G)$, this is the projective $\Drin(G)$-submodule $\Drin(G)\delta_\s=\Bbbk G\delta_\s$ of $\Drin(G)$, where $\{\delta_g\}_{g\in G}$ is the basis of $(\Bbbk G)^*$ dual to the basis $\{g\}_{g\in G}$. By functoriality of $\Ind_Z^G$, $P$ being a direct summand of $(\Bbbk Z)^{\oplus n}$, it readily follows that $W_{\s,P}$ is a direct summand of $\Drin(G)\delta_\s$ and thus projective.

Assume that $W:=W_{\s,P}$ is a direct summand of a direct sum $R:=((\Bbbk G)^*\otimes\Bbbk G)^{\oplus m}$ of the regular module in $\cZ(\Rep_\Bbbk G)$. Note that $R_\sigma=(\Bbbk G\delta_\sigma)^{\oplus m}$ and $W_\sigma=P$. Thus, we obtain that $P$ is a direct summand of $(\Bbbk G\delta_\sigma)^{\oplus m}$ as a $\Bbbk Z$-module. Choosing a decomposition of $G$ into right $Z$-cosets, we observe that as a $\Bbbk Z$-module, $\Bbbk G\delta_\sigma$ is simply a direct sum of copies of the regular module. Thus, $P$ is projective.
\end{proof}

The Grothendieck ring of the monoidal center of representations over algebraically closed fields of arbitrary characteristic was studied in \cite{Wit}. See also \cite{Don} for some concrete examples.

\begin{corollary}[{\cite{Wit}*{Section~3}}]\label{cor:WithK0}
For $G$ a finite group. There is an isomorphism of rings
$$K_0(\cZ(\Rep_\Bbbk  G))\otimes_{\mZ}\mC\cong \bigoplus_{\s} Z(\mC Z(\s)),$$
where $\s$ varies over a set of representatives of the conjugacy classes of $G.$
\end{corollary}
This corollary is proved in \cite{Wit}*{p.~316} and uses a result of G.~Lusztig.

\section{Classification of indecomposable objects in the center of \texorpdfstring{$\uRep S_t$}{Rep St}}\label{sec:classification}

\subsection{The center as a model-theoretic limit}\label{sec:centerlimit}

Viewing $\uRep S_t$ as a model-theoretic limit of categories $\Rep_{p_i} S_{t_i}$ as in \Cref{limitthm} suggests that it may be possible to interpret the center $\cZ(\uRep S_t)$ as a limit of the centers $\cZ(\Rep_{p_i}S_{t_i})$.  %
An object of the center of a monoidal category consists of the data of an object of the category along with a half-braiding. Therefore, the centers $\cZ(\uRep S_t)$ and $\cZ(\Rep_{p_i}S_{t_i})$ inherit filtrations $\cZ(\uRep S_t)^{\le k, m}$ and $\cZ(\Rep_{p_i}S_{t_i})^{\le k,m}$ just by looking at the underlying object. Precisely, $\cZ(\Rep_{p_i}S_{t_i})^{\le k,m}$ consists of those objects $(V,c)$ where $V$ lies in $(\Rep_{p_i}S_{t_i})^{\le k,m}$.

\begin{proposition}\label{center-limit} %
Recall the setup of \Cref{limitthm}(1)--(3). In all cases, 
the category $\cZ(\uRep S_t)^{\le k, m}$ is equivalent to the ultraproduct of the categories $\cZ(\Rep_{p_i}S_{t_i})^{\le k,m}$ (or, $\cZ(\Rep S_{t_i})^{\le k,m}$) as $\mathbb{C}$-linear categories with partially defined monoidal and additive structures.
\end{proposition}

\begin{proof}
One direction is fairly clear: If $(Y_i, c_i)$ is an object of $\cZ(\Rep_{p_i}S_{t_i})^{\le k,m}$ for each $i$, then we can take the ultraproduct of the underlying objects $Y_i$ to obtain an object $Y$ in $(\uRep S_t)^{\le k , m}$. The fact that $(Y_i, c_i)$ is an element of $\cZ(\Rep_{p_i}S_{t_i})$ means that the half-braiding $c_i$ is not just defined on $\cZ(\Rep_{p_i}S_{t_i})^{\le k,m}$, but on all larger filtered pieces as well. Using \L o\'s' theorem, the ultraproduct $(Y,c)$ becomes an object in $(\uRep S_t)^{\le k, m}$ with  half-braiding $c_Z$ globally defined for all $Z\in \uRep S_t$, hence lies in $\cZ(\uRep S_t)^{\le k, m}$.

Going the other way requires a bit more work. If we start with an object $(Y, c)$ in $\cZ(\uRep S_t)^{\le k, m}$ we can always represent $Y$ as an ultraproduct of objects $Y_i$ in $(\Rep_{p_i} S_{t_i})^{\le k , m}$.  However, a priori the half-braiding $c$ need not come from an ultraproduct of half-braidings $c_i$ globally on each $\Rep_{p_i} S_{t_i}$. Instead, initially we can just conclude it gives a sequence of partially defined half-braidings $c_i$, each defined on some $(\Rep_{p_i} S_{t_i})^{\le k_i , m_i}$ with $k_i$ and $m_i$ tending to infinity as $p_i$ does, but possibly not globally. 
This is essentially because the half-braiding is a natural transformation defined globally on the entire category, and the ultraproduct is only well behaved on these finite filtered pieces. To correct for this we  use Lemma \ref{center-data}, which shows that the data of a half-braiding is equivalent to some finite data that only involves objects and morphisms  inside one of these finite filtered pieces, namely $(\uRep S_t)^{\leq 2,2}$. 
Under the ultraproduct identification $Y$ corresponds to a sequence of objects $Y_i$ in $\cZ(\Rep_{p_i} S_{t_i})^{\le k , m}$ and $c_X$ gives us a sequence of maps $c_{Y_i,X}\colon  Y_i \otimes X \to  X  \otimes Y_i$ which satisfy the conditions of Lemma \ref{center-data} for almost all $i$. Namely, as soon as $k_i,m_i\geq 2$. Hence for almost all $i$ this corresponds to an object $(Y_i,c_i)$ inside $\cZ(\uRep S_t)^{\le k, m}$, as desired. This completes the proof of Proposition \ref{center-limit}.
\end{proof}

The following lemma specifies the finite datum required to obtain objects in the center of both interpolation categories and representation categories in finite characteristic.

\begin{lemma}\label{center-data}
Given an object $Y$ in $\uRep S_t$ or $\Rep_p S_d$, the datum of a half-braiding $c=\Set{c_V\colon Y\otimes V\to V \otimes Y\mid V\in \uRep S_t}$ is equivalent to the datum of a single morphism $c_X\colon Y \otimes X \to X\otimes Y$ such that
\begin{gather}
 c_X (\ide_{Y}\o\p_*) = \ide_Y\o \p_* \ ,
\quad
 (\p^*\o\ide_Y)  c_X = \p^*\o\ide_Y \ ,\label{centercond1}
\\
 c_{X\otimes X} (\ide_Y\o\p_X) = (\p_X\o\ide_Y) c_{X\otimes X}\,
\quad
 c_{X\otimes X} (\ide_Y\o\p_H) = (\p_H\o\ide_Y) c_{X\otimes X}\ .\label{centercond2}
\end{gather}
Here we denote
\begin{gather}
\p^* = \tpart{1,0}{}
  \in\Hom(X,\one)\ ,
\quad
\p_* = \tpart{0,1}{}
  \in\Hom(\one,X)\ ,
\quad
\p_H = \tpart{2,2}{\tpid{1,2} \draw (1,1.5)--(2,1.5);}
  \ ,
\quad
\p_X = \tpart{2,2}{\draw (1,1)--(2,2) (1,2)--(2,1);}
  \in\End(X^{\otimes 2})\, ,
  \end{gather}
      and we note that
  \begin{gather}
c_{X\otimes X}=(\ide_X\otimes c_X)(c_X\otimes \ide_X).
\end{gather}
In the case of $\Rep_p S_d$, we use the morphisms $f_{\pi^*},f_{\pi_*}, f_{\pi_H},f_{\pi_X}$ instead of $\pi^*,\pi_*,\pi_H,\pi_X$.

Equivalently, $(Y,c)$ being an object in the center is equivalent to $c_X$ commuting with the structural maps of the Frobenius algebra structure of $X$.

A morphism $f\colon (Y,c)\to (Y',c')$ in $\cZ(\uRep S_t)$ is equivalent to the datum of a morphism $f\colon Y\to Y'$ such that 
\begin{gather}
    c'_X(f\otimes \ide_X)=(\ide_X\otimes f)c_X.
\end{gather}
\end{lemma}
\begin{proof}
See \cite{FL}*{Proposition 3.1} for details on the argument in the case of $\uRep S_t$. 

If $\cha \Bbbk=p>d$, then $\Rep_p S_d$ is semisimple and the argument detailed in \cite{FL}*{Appendix A} proves the lemma since the Karoubian tensor subcategory $\langle X\rangle$ generated by $X$ inside of $\Rep_p S_d$ is the entire category $\Rep_p S_d$.

In remains to consider the case where $p\leq d$ and $\Rep_p S_d$ is non-semisimple. In this case, the $\langle X_d\rangle$ contains all projective objects. Further, morphisms $X_d^{\otimes n}\to X_d^{\otimes m}$ are generated by $\p_*, \p^*,\p_H, \p_X$.  
Hence, \Cref{center-extend} and \Cref{cor:ZReppSd-extend} shows that $c$ is uniquely determined by its restriction to $\cA=\langle X_d\rangle$. The half-braiding on $\langle X_d\rangle$, in turn, is determined by $c_X$ using an argument as in \cite{FL}*{Proposition A.1}.
\end{proof}

\subsection{Semisimplicity of the center} 
The techniques of viewing $\uRep S_t$ as a model-theoretic limit via ultrafilters enable us to resolve a question raised in \cite{FL}*{Question 3.32} about semisimplicity of $\cZ(\uRep S_t)$.

\begin{theorem}\label{semisimple}
The category $\cZ(\uRep S_t)$ is semisimple if and only if $t\notin \mZ_{\geq 0}$.
\end{theorem}
\begin{proof}
It is clear that for $t\in \mZ_{\geq 0}$, $\cZ(\uRep S_t)$ is not semisimple since it contains the non-semisimple category $\uRep S_t$ (with the standard symmetric braiding of $\uRep S_t$ for all objects) as a full subcategory. So let us consider the remaining case where $t$ is either transcendental or in $\overline{\mQ}\setminus \mZ_{\geq 0}$.

Since $\uRep S_t$ is semisimple for these values of $t$, we know  that  $\cZ(\uRep S_t)$ is an abelian category with objects of finite length. Therefore, in order to prove that $\cZ(\uRep S_t)$ is semisimple it suffices to check that every epimorphism in $\cZ(\uRep S_t)$ splits. 

Suppose $f\colon V \twoheadrightarrow W$ is an epimorphism in $\cZ(\uRep S_t)$.  There exist $k$ and $m$ such that $V$ and $W$ both lie in $\cZ(\uRep S_t)^{\le k, m}$, so by Proposition \ref{center-limit} we can identify $f$ with a sequence of morphisms $f_i\colon V_i \twoheadrightarrow W_i$ in $\cZ(\Rep_{p_i}S_{n_i})^{\le k,m}$ which are epimorphisms for almost every $i$. However, in these cases we have that $\cZ(\Rep_{p_i}S_{n_i})$ is a semisimple category, so for each value of $i$ where this is an epimorphism we can find a monomorphism $g_i\colon W_i \hookrightarrow V_i$ in $\cZ(\Rep_{p_i}S_{n_i})^{\le k,m}$ such that $f_i \circ g_i$ is the identity, meaning $g_i$ defines a splitting of $f_i$.

Taking this sequence of monomorphisms $g_i$ and applying Proposition \ref{center-limit} again we obtain a map $g: W \hookrightarrow V$ in $\cZ(\uRep S_t)^{\le k, m}$.  Since $f_i \circ g_i$ was the identity map for $W_i$ almost always, $f \circ g$ must be the identity map for $W$.  Hence $g$ defines a splitting for $f$, and we see that every epimorphism in $\cZ(\uRep S_t)$ splits as desired.
\end{proof}

It was shown in \cite{FL}*{Corollary~3.40} that if $d\in \mZ_{\geq 0}$, then $\cZ(\Rep  S_d)$ is the semisimplification of $\cZ(\uRep S_d)$ in the sense of \cite{EO}*{Section~2.3}.

\subsection{Induction functors to the center}

In this section, we give a construction of objects in $\cZ(\uRep S_t)$ using induction functors defined on $\cZ(\Rep  S_n)\boxtimes \uRep S_{t-n}$. For this, we fix $t\in\mC$, $(t_i)_i$, and $(p_i)_i$ such that
$\uRep S_t \subset \prod_\cU \Rep_{p_i} S_{t_i}$.
Then it follows that $\uRep S_{t-n} \subset \prod_\cU \Rep_{p_i} S_{t_i-n}$ using that $t_i$ grows to infinity, cf. Theorem \ref{limitthm},  so that for almost all $i$, $t_i\geq n$.
Note that  in the transcendental case, we use $\Rep S_{t_i}$ for all $i$ without explicitly mentioning this in the following statements, for which the proofs are easier in this case.

 \begin{definition}\label{def:Wcenter}
For any $n\geq 0$, $\mu$ a cycle type in $S_n$ (i.e.~a partition of $n$), $V$ a $\mC$-representation of $Z=Z(\mu)$, the centralizer in $S_n$ of an element of cycle type $\mu$, $U$ an object in $\uRep S_{t-n}$, we define
$$
W := \uInd_{Z\times S_{t-n}}^{S_t} (V \boxtimes U)
\qquad\text{in } \uRep S_t
.
$$
\end{definition}

If we identify $U\cong \prod_{\cU} U_i$ in the ultraproduct $\prod_{\cU}\Rep_{p_i}S_{t_i}$ and, using Theorem \ref{rep-limit}, $V\cong \prod_{\cU}V_i$ with  $V_i\in \Rep_{p_i} Z$, we find that setting 
$$
W_i := \Ind_{Z_i}^{S_{t_i}} \circ \Ind_{Z\times S_{t_i-n}}^{Z_i}(V_i \boxtimes U_i)
\qquad\text{in } \Rep_{p_i}S_{t_i}
,
$$
and $Z_i:=Z_{S_{t_i}}(\s)$ whenever $t_i\geq n$ gives $W\cong \prod_\cU W_i$ in the ultraproduct.

Fix an element $\s\in S_n$ of cycle type $\mu$. Then, for all $i$ such that $t_i\geq n$, we define $c^i$ to be the Yetter--Drinfeld braiding of $W_i$ in $\cZ(\Rep_{p_i} S_{t_i-n})$ for the grading determined by assigning degree $\s$ to every element in $V_i\boxtimes U_i\subseteq W_i$. In other words, the structure of $W_i$ as an object in the center is obtained using Theorem \ref{DPR-W-Thm} from the $Z_i$-module $\Ind_{Z\times S_{t_i-n}}^{Z_i}(V_i \boxtimes U_i)$.

\begin{lemma} The morphism $c:=\prod_\cU c^i$ defines a half-braiding for $W$ in $\uRep S_t$. Up to isomorphism in $\cZ(\uRep S_t)$, $c$ is independent of choice of the element $\s\in S_n$ of cycle type $\mu$. 
\end{lemma}
\begin{proof}
By construction, $c^i_{X}\colon W_i\otimes X\to X\otimes W_i$ defines a morphism in $\Rep_{p_i}S_{t_i}$ for almost all $i$. Thus, we obtain a morphism $c_{X}=\prod_\cU c_X^i \colon  \prod_{\cU}W_i\otimes X\to X\otimes \prod_{\cU}W_i$, using that $\prod_{\cU} X=X$. For almost all $i$, the morphism $c_X$ satisfies the conditions of Lemma \ref{center-data} and thus uniquely extends to give an object $(W,c)$ in $\cZ(\uRep S_t)$.

Now let $\tau\in S_n$ be conjugate to $\s$, i.e.~$\s=g\tau g^{-1}$ for some $g\in S_n$. Then $h\mapsto ghg^{-1}$ defines an isomorphism $\phi\colon Z(\tau)\to Z(\s)$. We regard $V_i$ as a $\ov{\mF}_{p_i}$-representation over $Z(\tau)$, denoted by  $V'_i$, via restriction along $\phi$. This way, we obtain isomorphisms
$$W_i':=\Ind_{Z(\tau)\times S_{t_i-n}}^{S_{t_i}}(V'_i\boxtimes U_i)\isomorph W_i=\Ind_{Z(\s)\times S_{t_i-n}}^{S_{t_i}}(V_i\boxtimes U_i).$$
Now mapping $h\otimes (v\otimes u)$ to $hg^{-1}\otimes v\otimes u$ defines an isomorphism $(W_i',(c')^i)\to (W_i,c^i)$ in $\cZ(\Rep_{p_i} S_{t_i})$. These induce isomorphisms in $\cZ(\uRep S_t)$ of the corresponding ultraproducts of objects. 
\end{proof}

\begin{definition}\label{Wdef} With $\mu\vdash n$, $V$ a representation of $Z(\mu)\subseteq S_n$, $U$ an object in $\uRep S_{t-n}$ as above we denote
$$
\uW{\mu,V,U} := (W, c),
\qquad\text{in }\cZ(\uRep S_t).
$$
The definition depends, by choice of $\mu\vdash n$, on $n$ and we sometimes write $\uW{\mu\vdash n,V,U}$ to highlight this dependence.
\end{definition}

We can make the constructions of \Cref{Wdef} functorial in the following way.

\begin{proposition}\label{prop:indcenter}
Let $n\geq 0$. There exists a $\mC$-linear functor
$$\uInd\colon \cZ(\Rep  S_n)\boxtimes \uRep S_{t-n} \longrightarrow \cZ(\uRep S_t),$$
sending the object $\Ind_{Z(\mu)}^{S_n}(V)\boxtimes U$ to $\uW{\mu,V,U}$. This functor $\uInd$ is a separable Frobenius monoidal functor (see \Cref{app:frob}) compatible with braidings in the sense that the diagram
\begin{gather}\label{braidcomp2}
    \vcenter{\hbox{
    \xymatrix{
    \uInd((W\boxtimes U)\otimes (W'\boxtimes U'))\ar[d]_{\delta_{W\boxtimes U,W'\boxtimes U'}}\ar[rrr]^{\uInd(\Psi_{W,W'}\boxtimes \Psi_{U,U'})}&&&  \uInd( (W'\boxtimes U')\otimes (W\boxtimes U))\\
    \uInd(W\boxtimes U)\otimes \Ind(W'\boxtimes U')\ar[rrr]^{\Psi_{\uInd(W\boxtimes U),\uInd(W'\boxtimes U')})}&&&\uInd(W'\boxtimes U')\otimes \Ind(W\boxtimes U)\ar[u]_{\mu_{W\boxtimes U,W'\boxtimes U'}}
    }
    }}
\end{gather}
commutes for any objects $W,W'$ in $\cZ(\Rep  S_n)$, $U,U'$ in $\uRep S_{t-n}$.
\end{proposition}
\begin{proof}
The composition $F\circ \uInd$ with the forgetful functor is the functor $\uInd_{S_n\times S_{t-n}}^{S_t}$ from \Cref{ind-res-sect}. Thus, given an object $(V,c)$ in $\cZ(\Rep  S_n)$ and $U\in \uRep S_{t-n}$, 
$$\uInd((V,c)\boxtimes U)=\uInd_{S_n\times S_{t-n}}^{S_t}(V\boxtimes U).$$
To define the half-braiding, we describe this functor using  ultraproducts. 
Using Theorem \ref{rep-limit}, we fix an equivalence of symmetric monoidal categories between $\Rep  S_n$ and the subcategory $\langle X \rangle$ of $\prod_{\cU}\Rep_{p_i}S_n$, for the ultraproduct $X$ of the standard representations. Under this equivalence, an object $(V,c)$ of $\cZ(\Rep  S_n)$ corresponds to the pair $(\prod_{\cU}V_i,\prod_{\cU}c^i)$. By a similar argument as the one used in  \Cref{center-limit}, the pairs $(V_i,c^i)$ define objects in $\cZ(\Rep_{p_i} S_{n})$ for almost all $i$.

Further, $U\cong \prod_{\cU} U_i$ for some $U_i\in \Rep_{p_i} S_{t_i-n}$, which can be defined for almost all $i$. Since $ \Rep_{p_i} S_{t_i-n}$ is symmetric monoidal, with braiding $\Psi$, there is a braided monoidal functor
$$ \Rep_{p_i} S_{t_i-n}\longrightarrow \cZ( \Rep_{p_i} S_{t_i-n}), \qquad U_i\mapsto (U_i,\Psi_{U_i,-}).$$
We now form the Deligne tensor product of this functor with the identity on $\cZ(\Rep_{p_i} S_n)$, and compose with the braided equivalences of finite tensor categories (see e.g.~\cite{EGNO})
$$\cZ(\Rep_{p_i} S_n)\boxtimes \cZ(\Rep_{p_i} S_{t_i-n})\isomorph \cZ(\Rep_{p_i} S_n\boxtimes \Rep_{p_i} S_{t_i-n})\simeq \cZ(\Rep_{p_i} (S_n\times S_{t_i-n})).$$
The resulting functor is composed with the functor $\cZ(\Ind_{S_n\times S_{t_i-n}}^{S_{t_i}})$ from \Cref{centerind-groups} to yields a functor
$$\Ind_i\colon \cZ(\Rep_{p_i} S_n)\boxtimes \Rep_{p_i} S_{t_i-n}\longrightarrow \cZ(\Rep_{p_i} S_{t_i}).$$
Passing to ultraproducts, \Cref{center-limit} yields a functor
$$\uInd\colon \cZ(\Rep  S_n)\boxtimes \uRep S_{t-n}\longrightarrow \cZ(\uRep S_{t}).$$
By \Cref{centerind-groups}, the functors $\Ind_i$ are separable Frobenius monoidal functors. Indeed, they are a composition of the Frobenius monoidal functor $\Ind_{S_n\times S_{t_i-n}}^{S_{t_i}}$ with strong monoidal functors (which are, in particular, separable Frobenius monoidal functors). Passing to ultraproducts induces natural transformations $\mu,\delta$, and the unit and counit maps $\eta,\epsilon$ for the functor $\uInd$. The axioms of a lax and oplax monoidal structure, the Frobenius compatibilities Equation \eqref{frobmon1}--\eqref{frobmon2}, as well as separability Equation \ref{eq:sep}, and the braiding compatibility Equation \ref{braidcomp} are first order expressions, since they are functional equalities (after fixing the needed tuples of objects to state these properties). Hence, by \L o\'s' theorem, $\uInd$ inherits all of these properties. 
\end{proof}

In particular, functoriality of $\uInd$ implies that isomorphic $Z(\mu)$-representations $V,V'$ and isomorphic objects $U,U'$ in $\uRep S_{t-n}$ yield isomorphic objects $\uW{\mu,V,U}\cong \uW{\mu,V',U'}$ in $\cZ(\uRep S_t)$. From \Cref{cor:frob}, we obtain the following consequences of \Cref{prop:indcenter}.

\begin{corollary} For any $n\geq 0$, we derive the following properties of the functor $\uInd$.
\begin{enumerate}
    \item $\uInd$ is exact and preserves duals.
    \item $\uInd$ preserves Frobenius algebra objects.
    \item For any $W,W'$ in $\cZ(\Rep  S_n)$, and $U,U'$ in $\uRep S_{t-n}$, the object $\uInd((W\boxtimes U)\otimes(W'\boxtimes U'))$ is a direct summand of $\uInd(W\boxtimes U)\otimes\uInd(W'\boxtimes U')$.
\end{enumerate}
\end{corollary}
\begin{proof}
This follows from the consequences collected in \Cref{cor:frob} which are general properties of a separable Frobenius monoidal functor, see \cite{DP}.
\end{proof}

\subsection{Indecomposable objects of the center}\label{sec:indec}

We are now ready to classify all indecomposable and simple objects of $\cZ(\uRep S_t)$. The following main theorem of this section settles \cite{FL}*{Question 3.32} by proving a classification of indecomposable objects in $\cZ(\uRep S_t)$.

\begin{theorem}\label{indec-classification}
The objects  $\uW{\mu,V,U}$ for $\mu,V,U$ as in \Cref{Wdef}, for $V$ a $Z(\mu)$-module, $U$  in $\uRep S_{t-n}$, and $\mu\vdash n$ singleton free is indecomposable if and only if $V$ is irreducible and $U$ is indecomposable. These objects provide a complete list of indecomposable objects in $\cZ(\uRep S_t )$ up to isomorphism.
\end{theorem}
The condition that $\mu\vdash n$  is singleton free is equivalent to the condition that any $\s$ of cycle type $\mu$ is a fixed-point free permutation of $n$.

\begin{proof} Consider a triple $\mu, V,U$ as above.
Write $Z=Z(\mu)=Z(\s)$ for $\s\in S_n$ of cycle type $\mu$ and assume that $\s$ has no fixed points. 
Using Theorem \ref{rep-limit}, we identify $\Rep  Z\simeq \prod_{\cU}\Rep_{p_i}Z$, and fix an isomorphism $V\cong \prod_\cU V_i$ for $Z$-modules over $\ov{\mF}_{p_i}$.
We also find an isomorphism $U\cong \prod_{\cU}U_i$, where $U_i$ is an object in $\Rep_{p_i}S_{t_i}$.
As the sequence $t_i$ is increasing, we may view $\s\in S_{t_i}$ for almost all $i$, and, using that $\s\in S_n$ is fixed-point free, $Z_{S_{t_i}}(\s)=Z\times S_{t_i-n}$. Using \L o\'s' theorem, $\Ind_{Z\times S_{t_i-n}}^{S_{t_i}}$ preserves indecomposable (respectively, irreducible) objects. Indeed, if $V,U$ are indecomposable, then $U_i$ is indecomposable for almost all $i$, and thus the $Z\times S_{t_i-n}$-module $V_i\boxtimes U_i$ is indecomposable for almost all $i$. Here, we use that the category $\Rep_{p_i}Z$ is semisimple for sufficiently large $i$.  Thus, any object in $\Rep_{p_i}(Z\times S_{t_i-n})$ is a direct sum of objects $V_i\boxtimes W_i$, where $V_i$ is a simple $Z$-module and $W_i$ is an object in $\Rep S_{t_i-n}$. %
Thus, for almost all $i$, $\Ind_{Z\times S_{t_i-n}}^{S_{t_i}}(V_i\boxtimes U_i)$ defines an indecomposable object in the center, and thus $\uInd_{Z\times S_{t-n}}^{S_t}(V\boxtimes U)$ is an indecomposable object in $\cZ(\uRep S_t)$. Here, we use that an object being indecomposable (respectively, irreducible) can be expressed using first order logic as in the proof of \Cref{semisimple}. 

Conversely, let $(W,c)$ be an indecomposable (or irreducible) object in $\cZ(\uRep S_t)^{\leq k,m}$. By \Cref{center-limit}, we know that $(W,c)$ corresponds to an ultraproduct of objects $(W_i, c^i)$ in $\cZ(\Rep_{p_i} S_{t_i})^{\leq k,m}$. For almost all $i$, $(W_i, c^i)$ is indecomposable (respectively, irreducible). Hence, we find elements $\s_i\in S_{t_i}$ and indecomposable (respectively, irreducible) $Z(\s_i)$-modules $Y_i$ such that $W_i\cong \Ind_{Z(\s_i)}^{S_{t_i}}(Y_i)$. We may find $n_i\leq t_i$ such that (possibly after changing $\s_i$ through conjugation), $\s_i\in S_{n_i}$ is fixed-point free. Then $Z(\s_i)\cong Z_i\times S_{t_i-n_i}$ for $Z_i$ the centralizer of $\s_i$ in $S_{n_i}$.

We observe that, since, for almost all $i$, $(W_i,c^i)$ lies in $(\Rep_{p_i} S_{t_i})^{\leq k}$ of the coarser filtration  --- which is based on the maximum tensor power of $X$, $\Ind_{Z_i\times S_{t_i-n_i}}^{S_{n_i}\times S_{t_i-n_i}}(Y_i)$ lies in $\Rep_{p_i} (S_{n_i})\boxtimes \Rep_{p_i} (S_{t_i})^{\leq k-n_i}$ (see \Cref{ind-res-sect}) and in particular for almost all $i$, $n_i\leq k$. This provides an upper bound on $n_i$. Hence, by \L o\'s' theorem, there exists $n\in \mN$ such that $n=n_i$ for almost all $i$. Hence, there are also only finitely many choices for the cycle type $\mu_i$ and one choice $\mu\vdash n$ is assumed for almost all $i$.

We have seen above that $(W_i,c^i)$ is indecomposable for almost all $i$ and $W_i=\Ind_{Z\times S_{t_i-n}}^{S_{t_i}}(Y_i)$ for a fixed choice of $n$, $\s\in S_n$ and $Z=Z(\s)$. Hence, by \Cref{DPR-W-Thm} and \L o\'s' theorem, $Y_i$ is indecomposable as an object in $$\Rep_{p_i}(Z\times S_{t_i-n})\simeq \Rep_{p_i}Z\boxtimes \Rep_{p_i}S_{t_i-n}$$ for almost all $i$. As the size of $Z$ is fixed and the sequence $p_i$ increases to infinity, for large enough $i$, $\Rep_{p_i}Z$ is semisimple. Thus, we find that the indecomposable object $Y_i$ is isomorphic to an object $Y_i\cong V_i\boxtimes U_i$, where $V_i$ is an indecomposable (and hence simple) $Z$-module and $U_i$ is an indecomposable  $S_{t_i-n}$-module.

Now, under the isomorphism $\prod_{\cU} \ov{\mF}_{p_i}\cong \mC$, the ultraproduct $\prod_{\cU}V_i$ corresponds to a $Z$-representation over $
\mC$ using \Cref{rep-limit}. Further, by construction, $U:=\prod_{\cU}U_i$ defines an object in $\uRep S_{t-n}$. The objects $V$, $U$ are indecomposable (respectively, irreducible) again using \L o\'s' theorem. By construction, it follows that $W\cong W_{\s,V,U}$ as an object in $\cZ(\uRep S_t)$. Thus, the objects $W_{\mu,V,U}$, with $\mu$ singleton free, $V$ irreducible and $U$ indecomposable, provide a full list of indecomposable (respectively, irreducible) objects of $\cZ(\uRep S_{t})$ as claimed.
\end{proof}

\begin{remark}In the notation of \Cref{indec-classification}, we may take $\mu=\varnothing\vdash 0$ and obtain the indecomposable objects of $\uRep S_{t}$ embedded in $\cZ(\uRep S_t)$ using the symmetric half-braiding.
\end{remark}

By \Cref{semisimple}, we have in particular classified the irreducible objects in the semisimple categories $\cZ(\uRep S_t)$ for $t\notin \mZ_{\geq 0}$ in \Cref{indec-classification}.

\begin{proposition}\label{uWdistinct}
Assume given two objects $\uW{\mu,V,U}$ and $\uW{\mu',V',U'}$, with singleton free $\mu\vdash n$, $\mu'\vdash n'$,  $V$, $V'$ irreducible as a $\mC Z(\mu)$-module, respectively, $\mC Z(\mu')$-module, $U$, $U'$ indecomposable (respectively, irreducible) in $\uRep S_{t-n}$, respectively, $\uRep S_{t-n'}$. Then  $\uW{\mu,V,U}$ and $\uW{\mu',V',U'}$ are isomorphic in $\cZ(\uRep S_t)$ if and only if $\mu=\mu'$, $V\cong V'$, and $U\cong U'$.
\end{proposition}
\begin{proof}
We may use the ultraproduct description $(W,c)\cong \prod_{\cU}(W_i,c^i)$ as in the proof of \Cref{indec-classification} which gives that  $W=\uW{\mu,V,U}$ and $W'=\uW{\mu',V',U'}$ and almost all $i$, we find that $W_i\cong \Ind_{Z(\mu)\times S_{t_i-n}}^{S_{t_i}}(V_i\boxtimes U_i)$, for some indecomposable (respectively, irreducible) $V_i,U_i$; and similarly, $W'_i\cong \Ind_{Z(\mu')\times S_{t_i-n'}}^{S_{t_i}}(V'_i\boxtimes U'_i)$. Now, if $W\cong W'$ as objects in $\cZ(\uRep S_t)$, then, for almost all $i$, $W_i\cong W_i'$ in $\cZ(\Rep_{p_i} S_{t_i})$. 
This yields, by \Cref{DPR-W-Thm}, that $n=n'$, $\mu=\mu'$ and 
$$ \Ind_{Z(\mu)\times S_{t_i-n}}^{Z_i}(V_i\boxtimes U_i)\cong  \Ind_{Z(\mu)\times S_{t_i-n}}^{Z_i}(V'_i\boxtimes U'_i),$$
for $Z_i$ the centralizer of $\s$, of cycle type $\mu$, in $S_{t_i}$. Using that $\mu$ is singleton free, we see that $Z_i=Z\times S_{t_i-n}$ and conclude that $V_i\cong V_i'$, $U_i\cong U_i'$ for almost all $i$. Thus,  $\prod_{\cU}V\cong \prod_{\cU}V'$ which implies by  \Cref{rep-limit} that $V\cong V'$. Further, $\prod_{\cU}U_i\cong \prod_{\cU}U'$ whence $U\cong U'$. 

Conversely, if $U\cong U'$, $V\cong V'$, it was already established as a consequence of \Cref{prop:indcenter} that $\uW{\mu,V,U}\cong \uW{\mu',V',U'}$.
\end{proof}

\subsection{The blocks of \texorpdfstring{$\cZ(\uRep S_t)$}{Z(Rep St)}}

We can now describe the blocks of the additive $\Bbbk$-linear category $\cZ(\uRep S_t)$ based on the description of blocks of $\uRep S_t$ from \cite{CO}.

\begin{lemma}\label{lem:ZuRepHoms}
If $\mu\vdash n$ and $\nu\vdash m$ are singleton free partitions, $V\in\cZ(\Rep  S_n)$, $V'\in \Rep  S_m$, $U\in \uRep S_{t-n}$, and $U'\in \uRep S_{t-m}$. Then 
\begin{equation}
\Hom_{\cZ(\uRep S_t)}(\uW{\mu, V,U}, \uW{\nu, V',U'}) \cong \begin{cases}
\Hom_{\Rep  Z(\mu)}(V,V')\otimes \Hom_{\uRep S_{t-n}}(U,U'), & \text{if }\mu=\nu,\\
\{0\}, & \text{if }\mu\neq \nu.
\end{cases}
\end{equation}
\end{lemma}
\begin{proof}
Assume that $\mu\neq \nu$. Using an ultraproduct description and \Cref{lem:cZHoms} we see that there are no non-zero morphisms between any $\uW{\mu, V,U}$ and $\uW{\nu, V',U'}$.
If $\mu=\nu$, we write $U\cong \prod_\cU U_i$ and $U'\cong \prod_\cU U'_i$, to see that
\begin{align*}
\Hom_{\cZ(\uRep S_t)}(\uW{\mu, V,U}, \uW{\nu, V',U'})
&\cong \prod_{\cU} \Hom_{\cZ(\Rep_{p_i} S_{t_i})}(
\Ind^{S_{t_i}}_{Z\times S_{t_i-n}}(V_i\boxtimes U_i),
\Ind^{S_{t_i}}_{Z\times S_{t_i-n}}(V'_i\boxtimes U_i')),
\\
&\cong \prod_\cU \Hom_{\Rep_{p_i} Z\times S_{t_i-n}} (V_i\boxtimes U_i, V'_i\boxtimes U'_i) 
\\
&\cong \prod_\cU \Hom_{\Rep_{p_i} Z}(V_i,V'_i)\otimes \Hom_{\Rep S_{t_i-n}} (U_i, U'_i)
\\
&\cong \Hom_{\Rep  Z}(V,V')\otimes \Hom_{\uRep S_{t-n}} (U, U'),
\end{align*}
as desired. Here, we use \Cref{lem:cZHoms} and the fact that, given that $\mu$ is singleton free, $Z(\mu)=Z\times S_{t_i-n}\leq S_{t_i}$, to obtain the third isomorphism.
\end{proof}

The above lemma enables us to decompose $\cZ(\uRep S_t)$ as a direct sum of additive $\Bbbk$-linear categories.

\begin{corollary}
Let $\mu \vdash n$ be singleton free. Then the $\Bbbk$-linear additive functor 
$$\uInd\colon \Rep  Z(\mu) \boxtimes \uRep S_{t-n}\longrightarrow \cZ(\uRep S_t)$$
is fully faithful. As an additive $\Bbbk$-linear category, $\cZ(\uRep S_t)$ is equivalent to the direct sum of categories
$$
\bigoplus_{n\geq0,\mu\vdash n} \Rep  Z(\mu)\boxtimes \uRep S_{t-n}.
$$
\end{corollary}
\begin{proof}
\Cref{lem:ZuRepHoms} implies that $\uInd$ is fully faithful on objects of the form $U\boxtimes V$, for $U,V$ indecomposable. Since the category $ \Rep  Z(\mu)$ is semisimple, this provides a full list of indecomposables of $\Rep  Z(\mu) \boxtimes \uRep S_{t-n}$.
\end{proof}

Note that the above decomposition is \emph{not} a decomposition of tensor categories. We can now describe the blocks of $\cZ(\uRep S_t)$. For this, denote by $\Irrep(Z)$ the set of isomorphism classes of irreducible representations of $\mC Z$.

\begin{corollary}\label{cor:blocksZRepSt}
The blocks in $\cZ(\uRep S_t )$ are parametrized by triples $(\mu,V,B)$, where $\mu\vdash n$ is a singleton free partition of some integer $n\geq 0$ and $V\in \Irrep(Z(\mu))$ and $B$ is a block of $\uRep S_{t-n}$.
\end{corollary}
\begin{proof} Note that the categories $\Rep  Z(\mu)$ are semisimple. Thus, each block contains a unique simple module $V$. \Cref{lem:ZuRepHoms} implies that $\uW{\mu,V,U}$ and $\uW{\mu,V',U'}$ are in the same block if and only if $\mu=\mu'$, $V\cong V'$, and $U$ and $U'$ are contained in the same block of $\uRep S_{t-n}$.
\end{proof}

The blocks of $\uRep S_t$ have been described combinatorially in \cite{CO}*{Section~5.1}. In particular they showed that for each $d\in \mathbb{Z}_{\ge 0}$ there are finitely many non-semisimple blocks in $\uRep S_d$, all of which are equivalent to the unique non-semisimple block in $\uRep S_0$.  It follows immediately from the description above that the same holds for $\cZ(\uRep S_t )$. \Cref{cor:blocksZRepSt} implies that blocks of $\cZ(\uRep S_t)$ are parametrized by pairs of blocks $B$ of $\uRep S_{t-n}$ and $(\mu,V)$. The latter parametrize those blocks of $\cZ(\Rep S_n)$ not induced from $\cZ(\Rep S_m)$ with $m<n$.

\subsection{Comparison with the previous construction}\label{sec:comparison}

In Section \ref{sec:indec}, we have constructed all objects in the center of $\uRep S_t$. Hence, our construction should recover those central objects constructed in the previous paper \cite{FL}. In this section, we show how the objects constructed there are indeed special cases of the construction explained here. We also note that the old objects from \cite{FL}*{Definition~3.12} generate $\cZ(\uRep S_t)$ as a Karoubian tensor category.

We recall the definition from \cite{FL} below. For this, we recall the embedding
\begin{equation}\label{x}
    x\colon \mC S_n\hookrightarrow \End_{\uRep S_t}(X^{\otimes n}), \quad g\mapsto x_g,
\end{equation}
see, e.g.,  \cite{CO}*{Equation (2.1)} for the definition of $x_g$.  We write $1_n$ for the identity of $S_n$. %
Hence, we can embed $\mC S_n\otimes M_k(\mC)$ into $\End_{\uRep S_t}((X^{\otimes n})^{\oplus k})$.

\begin{definition}[The central objects $\uWold{\mu,V}$] \label{Wold} Consider $n,k\geq 0$, $\sigma\in S_n$ an element of cycle type $\mu\vdash n$, $Z$ the centralizer of $\sigma$ in $S_n$, $V$ a $k$-dimensional representation of $Z$, $e_V\in\mC Z\otimes_\mC M_k(\mC)$ an idempotent with image isomorphic to $V$. We denote by $e$ the image of $e_V$ under the above embedding \eqref{x},
define
$\Img e = (X^{\otimes n}, e)$ to be the subobject of $X^{\otimes n}$ defined by the idempotent $e$ in $\uRep S_t$
and set
$$
d^{\s,V}_1 := (e\otimes \ide_{X})\Big(1_{n+1} + \sum_{1\leq i\leq n} E^i_{\sigma(i)} - E^i_i\Big)
\quad \in \End_{\uRep S_t}(\Img e\otimes X)
\ ,
$$
where $E^i_j$ is defined as a partition of $n+1$ upper and lower points as follows:
$$
E^i_j := \{ \{k,k'\} \}_{1\leq k\leq n}
 \setminus \{\{i,i'\},\{j,j'\}\}
 \cup \{\{i,n+1\},\{j',(n+1)'\}\}
\ .
$$
Let $\tau_n$ be the permutation $(1\ 2\ \dots\ n+1)$ viewed as an endomorphism of $X^{\otimes(n+1)}$ in $\uRep S_t$, i.e., $\tau_n$ is the symmetric braiding $\Psi_{X^{\otimes n},X}$ of $\uRep S_t$. It was shown in \cite{FL}*{Theorem~3.11}, that $c_X:= \tau_n d_1^{\s,V}$ determines an object $(\Img e,c)$ in the center of $\uRep S_t$ via Lemma \ref{center-data}. %
It was also shown in \cite{FL}*{Proposition 3.22} that the isomorphism class of this central object depends only on the cycle type $\mu$ of $\s$ and on the isomorphism class of $V$. Therefore, we will denote this central object considered in \cite{FL}*{Definition 3.12} by $\uWold{\mu,V}$.
\end{definition}

Let us assume $\mu$ has $r\geq 0$ singletons, so any permutation $\sigma$ of cycle type $\mu$ has $r$ fixed points. Let $\mu_0$ be the partition of $n_0:=n-r$ we obtain from $\mu$ by removing all singletons. Then the centralizer $Z$ of $\sigma$ is isomorphic to $Z_0\times S_r$, where $Z_0$ is the centralizer of an element of cycle type $\mu_0$ in $S_{n_0}$, and the $\mC Z$-module $V$ decomposes as $V_0\boxtimes U_0$ for a $\mC Z_0$-module $V_0$ and a $\mC S_r$-module $U_0$. We set $U_{\mu,V}:=\uInd_{S_r\times S_{t-n}}^{S_{t-n_0}} U_0\boxtimes\one$. Note that, in particular, if $r=0$, then $n_0=n$, $Z_0=Z$, $V_0=V$, and $U_{\mu,V}=\one$.

\begin{lemma}\label{lem::induced-objects} If we interpret $e_V$ as an idempotent $e$ in $\uRep S_t$ as in Definition \ref{Wold}, then 
\begin{align}\label{ind-eq-1}
\Img e \cong \uInd_{Z_0\times S_{t-n_0}}^{S_t} (V_0\boxtimes U_{\mu,V}) .
\end{align}
In particular, if $\mu$ has no singletons, that is, $\sigma$ has no fixed points, then
\begin{align}\label{ind-eq-2}
\Img e \cong \Ind_{Z\times S_{t-n}}^{S_t} (V\boxtimes \one) .
\end{align}
\end{lemma}

\begin{proof}
Again we use the ultraproduct concepts explained in \Cref{sect:int-group} and \Cref{ind-res-sect}: We choose an ultraproduct representation $\uRep S_t\cong \langle X\rangle\subset \prod_\cU \Rep_{p_i} S_{t_i}$ for $\uRep S_t$, and similarly, we fix  $\Rep  S_n \cong \langle X\rangle \subset \prod_\cU \Rep_{p_i} S_n$, for $n\geq 0$, cf. \Cref{rep-limit}. For simplicity, write $\mF_i:=\ov{\mF}_{p_i}$.

We prove the general case, Equation \eqref{ind-eq-1}, first. Almost always, $t_i\geq n$, so let us assume this holds.
Choose a finite field extension $\Bbbk$ of $\mQ$ such that $V_0,U_0$ are defined over $\Bbbk$. Denote by $\cO$ the ring of integers of $\Bbbk$ and for all $i$, choose a prime ideal $\mathfrak{p}_i\triangleleft \cO$ containing $p_i$. Assume that $i$ is sufficiently large $i$ such that $p_i$ does not divide the discriminant of $\Bbbk$ and $\cO/\mathfrak{p}_i$ is $\mF_{p_i^d}$ for some $d$, so in particular, $\cO/\mathfrak{p}_i$ embeds into $\mF_i$.

For the $\mC S_r$ module $U_0$,  by \Cref{ultra-reduction}, we fix an ultraproduct presentation $U_0\cong \prod_{\cU} U_{0,i}$ where $U_{0,i}$ is the reduction of a chosen integral form of $U_0$ modulo $\mathfrak{p}_i$ viewed as an $\mF_i S_{t_i}$-module.
From Section \ref{ind-res-sect} we know that the induction functor can be computed as an ultraproduct of the corresponding induction functors in finite characteristic. Thus, we find an ultraproduct presentation $U_{\mu,V}=\prod_\cU U_i$, where
$$U_i=\Ind_{S_r\times S_{t_i-n}}^{S_{t_i}-n_0}(U_{0,i}\boxtimes \mF_i).$$ 
Similarly, we choose an integral form of $V_0$ and write $V_{0,i}$ for the reduction of the integral form of $V_0$ modulo $\mathfrak{p}_i$ and observe that $V_0\cong \prod_{\cU} V_{0,i}$ by \Cref{ultra-reduction}. Thus, the right-hand side of \eqref{ind-eq-1} is isomorphic to the ultraproduct of the $\mF_i S_{t_i}$-modules 
\begin{align*}
\Ind_{Z_0\times S_{t_i-n_0}}^{S_{t_i}}(V_{0,i}\boxtimes U_i)
&=\Ind_{Z_0\times S_{t_i-n_0}}^{S_{t_i}}\big(V_{0,i}\boxtimes \Ind_{S_r\times S_{t_i-n}}^{S_{t_i-n_0}}(U_{0,i}\boxtimes \mF_i)\big)\\
&\cong \Ind_{Z_0\times S_r\times S_{t_i-n}}^{S_{t_i}}(V_{0,i}\boxtimes U_{0,i}\boxtimes \mF_i).
\end{align*}

As we have chosen integral forms of $V_0$, $U_0$ over the ring of integers $\cO$ of $\Bbbk$, we note that $(V_0\boxtimes U_0)\otimes_{\cO}\mC\cong V$. Thus, we denote $V_{i}:=(V_0\boxtimes U_0)\otimes_\cO \mF_i$ and $V\cong \prod_{\cU} V_i$. We may write, possibly replacing $V$ with an isomorphic module,
$$
e_V = \sum_{g\in Z} g\otimes_{\mC} m_g,
$$
with $k$-by-$k$-matrices $(m_g)_{g\in Z}$ whose entries lie in $\cO$. It follows that 
$$
e=\sum_{g\in Z} x_g\otimes_{\mC} m_g=\sum_{g\in Z} gx_{1_n}\otimes_{\mC} m_g\;\in \;\End_{\uRep S_t}((X^{\otimes n})^{\oplus k}).
$$
We first observe that the image of the component $(x_{1_n})_i$  of the idempotent $x_{1_n}$, in the notation of \eqref{eqn:xpi}, acting on the $\mF_iS_{t_i}$-module $X_{t_i}^{\otimes n}$ is isomorphic to $\mF_i S_{t_i}/S_{t_i-n}$ \cite{FL}*{Remark 2.5}. Denote by $m_{g,i}$ the $\mF_i$-matrix obtained from $m_g$ by reducing all entries modulo $\mathfrak{p}_i$. 
Now, the object $\Img e$ is isomorphic to the ultraproduct of the $\mF_i S_{t_i}$-modules $P_i$, where
$P_i$ is the image of the idempotent 
$$\sum_{g\in Z}g_i\otimes_{\mF_i} m_{g,i}\colon \mF_iS_{t_i}/S_{t_i-n}\otimes_{\mF_i} \mF_i^{\oplus k}\longrightarrow \mF_iS_{t_i}/S_{t_i-n}\otimes_{\mF_i} \mF_i^{\oplus k}.$$
Here, $g_i$ denotes the evaluation of the partition corresponding to $g$ from \eqref{piimage}. We note that $g_i$ corresponds to the right multiplication action of the subgroup $\{1_{t_i-n}\}\times S_n$ on $S_{t_i}/S_{t_i-n}$ \cite{FL}*{Remark 2.5}.
Hence, it follows from a computation similar to \cite{FL}*{Proposition~3.8} that $P_i$ is isomorphic to 
\begin{align*}
\mF_iS_{t_i}/S_{t_i-n}\otimes_{Z} V_i\cong \Ind_{Z\times S_{t_i-n}}^{S_{t_i}} (V_i \boxtimes \mF_i)\cong \Ind_{Z_0\times S_{r}\times S_{t_i-n}}^{S_{t_i}} (V_{0,i}\boxtimes U_{0,i} \boxtimes \mF_i).
\end{align*}
Thus, for almost all $i$, the component in $\Rep_{p_i} S_{t_i}$ of the objects on the left and right side in \eqref{ind-eq-1} are isomorphic, which proves the assertion.

It is clear from $V=V_0$ and $U_{\mu,V}=\one$ that \eqref{ind-eq-2} is a special case of \eqref{ind-eq-1}.
\end{proof}

\begin{proposition} $\uWold{\mu,V}$ is isomorphic to $\uW{\mu_0,V_0,U}$ for $U=U_{\mu,V}$.
In particular, if $\mu$ has no singletons, then $\uWold{\mu,V}$ is isomorphic to $\uW{\mu,V,\one}$.
\end{proposition} 

\begin{proof} By \Cref{lem::induced-objects}, we know the underlying objects are isomorphic. To see that the half-braidings coincide, it is enough to verify this for the half-braidings evaluated at the tensor generating object.

With the same choices and symbols as in the proof of \Cref{lem::induced-objects}, let us pick a basis $e_1,\dots,e_{t_i}$ for $X_{t_i}$, the tensor generator in $\Rep_{p_i} S_{t_i}$. Then a basis of $X_{t_i}^{\o n}$ is given by the tensor products $e_{i_1}\otimes\dots\otimes e_{i_n}$. Now the $\mF_i S_{t_i}$-module map  $\big(d_{1}^{\sigma,V}\big)_i$ defined in \eqref{piimage}, corresponding to $d_{1}^{\sigma,V}$, sends $(e_{i_1}\otimes\dots\otimes e_{i_n}\otimes v)\o e_j$ to
$$(e_{i_1}\otimes\dots\otimes e_{i_n}\otimes v)\o \begin{cases}
e_{\phi\sigma\phi^{-1} (j)}, 
& \text{if } (i_1,\dots,i_n)
  \text{ are pairwise distinct and }
  j\in\{i_1,\dots,i_n\},
\\
e_{j},
& \text{if }(i_1,\dots,i_n)
  \text{ are pairwise distinct and }
  j\not\in\{i_1,\dots,i_n\}.
\\
0, & \text{else},
\end{cases}
$$
for any $v\in V_i$ (the $\mF_i$-reduction of the integral form of $V$), where $\phi\colon\{1,\dots,n\}\to\{1,\dots,t\}$ is the injective map defined by $\phi(k)=i_k$, for $k=1,\ldots, n$.

Under the isomorphism \eqref{ind-eq-1}, evaluated on the $i$-th part of the ultrafilter presentation used in the proof of \Cref{lem::induced-objects}, the map $\big(d_{1}^{\s,V}\big)$ corresponds to the map
$$
([g]\otimes v)\otimes e_j
\longmapsto ([g]\otimes v)\otimes (g(1,\sigma)g^{-1}\cdot e_j)
$$
for any $[g]\in S_{t_i}/S_{t_i-n}$, which is indeed the desired Yetter--Drinfeld braiding we expect from $\uW{\mu_0,V_0,U}$. This computation follows (similarly to \cite{FL}*{Proposition~3.36}) using that $[g]$ is identified with $e_{i_1}\otimes\dots\otimes e_{i_n}$ if $g(k)=i_k$ under the isomorphisms of $(\Img x_{1_n})_i$ and $\mF_i S_{t_i}/S_{t_i-n}$.
\end{proof}

We note that $U_{\mu,V}$ is generally (in particular, for $r\geq 1$) a decomposable object in $\uRep S_{t-r}$ whose decomposition into indecomposable objects implies a decomposition for $\uWold{\mu,V}$. Hence, by \Cref{indec-classification}, $\uWold{\mu,V}$ is typically not indecomposable.

We conclude our comparison with the results in \cite{FL} with the following observation, which answers \cite{FL}*{Question 3.31} in the affirmative:

\begin{proposition}\label{prop:oldgenall}
The objects $\uWold{\mu,V}$ generate the entire center $\cZ(\uRep S_t)$ as a Karoubian tensor category.
\end{proposition}

\begin{proof} By \Cref{indec-classification}, it is enough to show that any object $\uW{\mu',V',U'}$ for $\mu'$ without singletons, $V'$ irreducible, and $U'$ indecomposable, is isomorphic to a subobject of $\uWold{\mu,V}$ for some $\mu, V$, with $r, n_0, \mu_0, V_0, U_0, U_{\mu,_V}$ all defined as above. Note that by construction, $\uWold{\mu,V}$ contains $\uW{\mu_0,V_0,Y}$ for any subobject $Y$ of $U_{\mu,V}$. Hence it is enough to show that, for any $n_0$, any indecomposable $U'$ in $\uRep S_{t-n_0}$ is isomorphic to a subobject of $\uInd_{S_r\times S_{t-n_0-r}}^{S_{t-n_0}} (U_0\boxtimes\one)$ for some $r, U_0$,  or equivalently, that any indecomposable $U'$ in $\uRep S_t$ is isomorphic to a subobject of $\uInd_{S_r\times S_{t-r}}^{S_t} (U_0\boxtimes\one)$ for some $r, U_0$, for any $t$.  
But this is true by \Cref{lem:inddec}.
\end{proof}

\subsection{The center of \texorpdfstring{$\uRep^\ab S_d$}{Repab St} } \label{sec:ZRepab}

Let $d\in \mZ_{\geq 0}$ and consider the abelian envelope $\uRep^\ab S_d$ of the category $\uRep S_d$ of  \cites{Del,CO-Rep-St-ab}. There is a full and faithful functor of symmetric monoidal categories $\uRep S_d\to \uRep^\ab S_d$.

Before determining the center of $\uRep^\ab S_d$ we need some preliminary observations. 

\begin{proposition}\label{Z-ab-firstorder}
Let $(Y,c)$ be an object in $\cZ(\uRep^\ab S_d)$. Then the half-braiding $c$ is uniquely determined by the morphism $c_X\colon Y\otimes X\to X\otimes Y$.

Conversely, any pair $(Y,c_X)$, where $Y$ is an object in $\uRep^\ab S_d$ and $c_X$ is a morphism satisfying the finite list of conditions from Lemma \ref{center-data} uniquely extends to give an object in $\cZ(\uRep^\ab S_d)$. 
\end{proposition}

\begin{proof}
The abelian category $\uRep^\ab S_d$ has enough projectives. The projectives are direct summands of sums of tensor powers of the distinguished generating object $X$ \cite{CO-Rep-St-ab}*{Remark 4.8}. Hence, the projectives of $\cC=\uRep^\ab S_d$ are
contained in $\cA=\uRep S_d$. Thus, we can use \Cref{center-extend} to conclude that $c$ is uniquely determined by its restriction to $\uRep S_d$.
By a slight generalization of \Cref{center-data}, allowing the object $Y$ be be in $\uRep^\ab S_d$ rather than $\uRep S_d$, it follows that $c$ is uniquely determined by a morphism $c_X$ satisfying the conditions \eqref{centercond1}--\eqref{centercond2}.
\end{proof}

The next corollary is a direct consequence of \Cref{Z-ab-firstorder} and \Cref{center-extend}.

\begin{corollary}\label{center-subcat}
The inclusion of $\uRep S_d$ into its abelian envelope induces a full and faithful functor of braided monoidal categories
$$\cZ(\uRep S_d)\hookrightarrow \cZ(\uRep^\ab S_d).$$
\end{corollary}

To classify simple objects in $\cZ(\uRep^\ab S_d)$ we first note that $\cZ(\uRep^\ab S_d)$ has the \emph{Jordan--H\"older property} by virtue of being a locally finite abelian category \cite{EGNO}*{Theorem 1.5.4}. That is, any object in $\cZ(\uRep^\ab S_d)$ has a finite filtration by simple objects in this category. Moreover, the simple composition factors are unique up to order.

The induction functor from \Cref{prop:indcenter} extends to the abelian envelope such that the following diagram commutes
\begin{align}\label{uInd-abelianized}
    \vcenter{\hbox{
    \xymatrix{ 
        \cZ(\Rep  S_n)\boxtimes \uRep S_{d-n} \ar@{^{(}->}[d]\ar[rr]^-{\uInd}&&\cZ(\uRep S_d )\ar@{^{(}->}[d]\\
        \cZ(\Rep  S_n)\boxtimes \uRep^\ab S_{d-n}\ar[rr]^-{\uInd^{\ab}}&&\cZ(\uRep^\ab S_d),
    }
    }}
\end{align}
where taking the abelian envelope of $\uRep S_{d-n}$ is only required if $d-n\geq 0$. 
This follows, using the above \Cref{center-subcat}, from right exactness of $\uInd$ in the second tensor factor. Indeed, we may resolve a given object $M$ in $\uRep^\ab S_{d-n}$ by $P_0\xrightarrow{p} P_1\to M\to 0$, where $P_0,P_1$ are projective objects of $\uRep S_{d-n}$, and define $\uInd(V\boxtimes M)$ to be the cokernel of the map $\uInd(\ide_V\boxtimes p)$. Alternatively, we may also introduce the functor $\uInd^{\ab}$ using ultraproducts, similarly to how $\uInd$ was introduced in \Cref{prop:indcenter}, working with objects from $\uRep^\ab S_{d-n}$ in the second tensor factor, cf. \Cref{limitthm2}.

We can now extend \Cref{indec-classification} to the abelian envelope.

\begin{corollary}\label{indec-classification2}
The objects  $\uW{\mu,V,U}$, for $\mu,V,U$ as in \Cref{Wdef}, where $V$ is irreducible as a $Z(\mu)$-module, $U$ is indecomposable (or irreducible, or indecomposable projective) in $\uRep^\ab S_{d-n}$ and $\mu\vdash n$ is singleton free, provide a full list of indecomposable (respectively, irreducible, or indecomposable projective) objects in $\cZ(\uRep^\ab S_d)$ up to isomorphism.
\end{corollary}
\begin{proof}
The classification of irreducible or indecomposable objects in $\cZ(\uRep^\ab S_d)$ is completely analogous to the proof of \Cref{indec-classification}. The only difference is that we consider ultraproducts of general (indecomposable) modules in the categories $\cZ(\Rep_{p_i} S_{t_i})$ rather than just objects for which the underlying $S_{t_i}$-modules are in $\langle X_{t_i}\rangle$. The projective objects among the indecomposables are identified in \Cref{prop::X-d-plus-1-in-center}(2) below.
\end{proof}

\begin{proposition} \label{prop::X-d-plus-1-in-center} Let $d\in \mZ_{\geq 0}$.
\begin{enumerate}
    \item The object $X_{(d+1)}$ with symmetric half-braiding is the projective cover of $\one$ in $\cZ(\uRep^\ab S_d)$.
    \item An object $\uW{\mu,V,U}$ is projective if and only if $U$ is projective. 
    \item All projective objects of $\cZ(\uRep^\ab S_d)$ are contained in the full subcategory $\cZ(\uRep S_d)$.
\end{enumerate}
\end{proposition}
\begin{proof}
Note that the functor 
$$I\colon \uRep^\ab S_d\to \cZ(\uRep^\ab S_d), \qquad V \mapsto (V, \Psi_{V,-})$$
admits an ultrafilter description. It is given by the ultrafilter of the corresponding functors in characteristic $p_i$. This functor preserves projective objects (cf. \Cref{proj-reps}) as an induction functor (for $\s=1\in S_0$). Thus $I$ preserves projective objects and we can apply $I$ to the projective cover $X_{(d+1)}\twoheadrightarrow\one$ from \Cref{lem:projcover}. This proves Part (1).

To prove Part (2), using Part (1), an object $P$ in $\cZ(\uRep^\ab S_d)^{\leq k}$, $P$ is projective in $\cZ(\uRep^\ab S_d)$ if and only if it is projective as an object of $\cZ(\uRep^\ab S_d)^{\leq k+d+1}$ (cf. the proof of \Cref{prop:proj-filt}). Thus, for  $P\cong \prod_\cU P_i$ is projective if and only if for almost all $i$, $P_i$ is projective. Since $P\cong \uW{\mu,V,U}$, we have that $P_i\cong \Ind_{Z\times S_{t_i-n}}^{S_{t_i}}(V_i\boxtimes U_i)$ and for almost all $i$. By \Cref{proj-reps}, $P_i$ is projective if and only if $U_i$ is projective. Thus, almost all $P_i$ are projective if and only if almost all $U_i$ are projective, proving Part (2). Further, observe that for almost all $i$, $P_i$ is in $\langle X_{t_i}\rangle$ and hence $P$ is contained in the subcategory $\cZ(\uRep S_d)$. This proves Part~(3). 
\end{proof}

We can now collect a few consequences of the constructions of this section.
\begin{corollary} \label{enough-projectives}
Any indecomposable object in $\cZ(\uRep^\ab S_d)$ is a quotient of an object in $\cZ(\uRep S_d)$. In particular, the abelian category $\cZ(\uRep^\ab S_d)$ has enough projectives.
\end{corollary}
\begin{proof}
By the above \Cref{indec-classification2}, an indecomposable object in  $\cZ(\uRep^\ab S_d)$ is isomorphic to one of the form $\uW{\mu,V,U}$. We may choose a projective cover $P\twoheadrightarrow U$. Then $V\boxtimes P$ is projective in $\cZ(\Rep  S_n)\boxtimes \uRep^\ab S_{d-n}$ as the first tensorand is a semisimple category. By right exactness of $\uInd$ in the second tensor factor, $\uW{\mu,V,U}$ arises as a quotient of the object $\uW{\mu,V,P}$, which lies in $\cZ(\uRep S_d)$.
This argument also shows that $\cZ(\uRep^\ab S_d)$ has enough projectives.
\end{proof}

Recall the concept of the abelian envelope from \cite{BEO}, discussed in \Cref{sec:RepSdab}, to obtain the following result.

\begin{corollary}\label{cor:abenvelopeZRep}
The category $\cZ(\uRep^\ab S_d)$ is the abelian envelope of $\cZ(\uRep S_d)$.
\end{corollary}

\begin{proof} 
The corollary follows from the more general results of \Cref{cor::center-abelian-envelope-general} since the functor $\uRep^\ab S_d\to\cZ(\uRep^\ab S_d)$ preserves projectives (as noted in the proof of \Cref{prop::X-d-plus-1-in-center}).
\end{proof}

\subsection{Non-degeneracy of the centers}\label{sec:non-deg}

We conclude this section with a discussion on non-degeneracy of the center of Deligne's interpolation categories and their abelian envelopes in order to highlight the analogy with modular tensor categories.

We note that the categories $\cZ(\uRep S_t)$, for $t$ generic, and $\cZ(\uRep^\ab S_d)$, for $d\in \mZ_{\geq 0}$, are tensor categories in the sense of \cite{EGNO}. That is, they are rigid $\Bbbk$-linear monoidal abelian categories, with bilinear tensor product, that are locally finite such that $\End(\one)=\Bbbk$. Recall that a braided tensor category $\cC$ is \emph{factorizable} if the canonoical functor  $\cC\boxtimes \cC^{\rev}\to \cZ(\cC)$
gives an equivalence of braided monoidal categories. The centers $\cZ(\uRep S_t)$ and $\cZ(\uRep^\ab S_d)$ are factorizable using \cite{EGNO}*{Proposition~8.6.3}. 

Let $\cC$ be a braided $\Bbbk$-linear monoidal category with braiding $\Psi$. The object $V$ in $\cC$ \emph{centralizes} the object $W$ of $\cC$ if 
$$\Psi_{W,V}\circ \Psi_{V,W}=\ide_{V\otimes W}.$$
The \emph{M\"uger center} $\cC'$ of $\cC$ is the full subcategory of $\cC$ on those objects which centralize \emph{all} objects of $\cC$. 
Recall that $\cC$ is called \emph{non-degenerate} if the M\"uger center $\cC'$ is generated by the tensor unit (i.e.~equivalent to $\op{Vect}_\Bbbk$). 

In the case of \emph{finite} braided tensor categories, the concepts of factorizability and non-degeneracy are equivalent \cite{Shi}. However, we saw that the categories $\cZ(\uRep S_t)$ are only locally finite and have infinitely many simple objects, see \Cref{indec-classification} and  \Cref{indec-classification2}. Nevertheless, we have the following result.

\begin{proposition}
The braided tensor categories $\cZ(\uRep S_t)$, for $t$ generic, and $\cZ(\uRep^\ab S_d)$, for $d\in \mZ_{\geq 0}$, are non-degenerate. 
\end{proposition}
\begin{proof}
By \Cref{indec-classification} and \Cref{indec-classification2}, we have a classification of indecomposable objects $W=\uW{\mu,U,V}$ in $\uRep S_t$ and $\uRep^\ab S_d$, respectively. Assume that $\s\in S_n$ is an element of cycle type $\mu$ which is fixed-point free. Recall the ultrafilter description $W\cong \prod_{\cU}W_i$, where $W_i=\Ind_{Z(\s)\times S_{t_i-n}}^{S_{t_i}}(U_i\boxtimes V_i)$. 

First consider the case that $\s\neq 1$. Then for $i$ large enough, $\s\notin Z(S_{t_i})$. Choose $\tau\in S_{t_i}$ which does not commute with $\s$ and define $W':=\uW{\tau,\one,\one}$. Then choose an ultraproduct representation
$W'=\prod_{\cU}W_i'$ and for any $j\geq i$ one computes that 
$\Psi_{W'_j,W_j}\circ\Psi_{W_j,W_j'}\neq \ide_{W_j\otimes W_j'}.$
Indeed,
\begin{align*}
    \Psi_{W'_j,W_j}\circ\Psi_{W_j,W_j'}\left((1\otimes (u\boxtimes v))\otimes (1\otimes 1)\right )&=
    \Psi_{W,W'}\left((\s\otimes 1)\otimes  (1\otimes (u\boxtimes v))\right)\\
    &=(\s\tau \s^{-1} \otimes (u\boxtimes v))\otimes (\s\otimes 1)
\end{align*}
does not equal the identity as $\s\tau \s^{-1}\in Z(\s)$ if and only if $\tau \in Z(\s)$ (note that the tensor products in the induced modules are taken over the group algebras of subgroups, not over the base field).

If $\s=1\in S_0$, then without loss of generality $U=\one$ but assume that $V\neq \one\in \uRep S_{t}$, as otherwise $W\cong\one$ as objects of the center. Let us fix an ultrafilter representation $V\cong\prod_\cU V_i$. Then there exists an element $\tau\in S_{t_i}$ such that $\tau$, viewed as an element of $S_{t_j}$, for $j\geq i$, acts non-trivially on $V_j$. Then, computing as above, the braiding again does not square to the identity. Indeed, for almost all $j$,
\begin{align*}
    \Psi_{W'_j,W_j}\circ\Psi_{W_j,W_j'}\left((1\otimes v)\otimes (1\otimes 1)\right )&=
    (\s\tau\s^{-1}\otimes v)\otimes (\s\otimes 1)\\
    &=(1\otimes \tau v)\otimes (1 \otimes 1),
\end{align*}
for all $v\in \Bbbk \boxtimes V_j\cong V_j$. Passing to ultraproducts, as for almost all $i$, the braiding does not square to the identity, we conclude $\Psi_{W',W}\Psi_{W,W'}\neq \ide_{W\otimes W'}$ in both cases. Thus, $W$ is not in the M\"uger center, which is hence trivial.
\end{proof}

To summarize, $\cZ(\uRep S_t)$ and $\cZ(\uRep^\ab S_d)$ are non-degenerate and factorizable braided tensor categories that possess a ribbon structure by \cite{FL}*{Theorem~3.28}. Thus, these categories can be seen as infinite (possibly non-semisimple) analogues of modular tensor categories (cf. \cite{Shi} in the case of a finite tensor category).

\section{The Grothendieck ring of the center of \texorpdfstring{$\uRep S_t$}{Rep St}}\label{sec:K0}

In this section, we show that the Grothendieck ring of $\cZ(\uRep S_t)$ is a filtered ring such that its associated graded ring can be identified with the Grothendieck ring of a braided monoidal category $\cZ\Rep  S_{\geq 0}$ defined as the sum of the categories $\cZ(\Rep  S_n)$, for all $n\geq 0$, with the induction product on centers from \Cref{app:frob}. Thus, the Grothendieck ring can be described using a tower of centers of the symmetric groups $S_n$, varying $n$.

\subsection{The tower of centers of a sequence of groups}\label{sec:towers}

In this subsection, $\{G_n\}_{n\in \mZ_{\geq 0}}$ can be any sequence of groups such that $G_0=\{1\}\subseteq G_1\subseteq G_{2}\subseteq\ldots$, together with embeddings $\iota_{m,n}\colon G_m\times G_n\subseteq G_{m+n}$ such that $\iota_{m+n,r}(\iota_{m,n},\ide)=\iota_{m,n+r}(\ide,\iota_{n,r})$ under the natural identifications of $(G_m\times G_n)\times G_r$ and $G_m\times (G_n\times G_r)$. In addition, we require that $G_n\times G_0 \to G_n$ and $G_0\times G_n\to G_n$ are simply given by $(g,1)\mapsto g$, respectively, $(1,g)\mapsto g$. We refer to such a sequence as a \emph{tower of groups}. It induces a tower of algebras $\{\Bbbk G_n\}_{n\geq 0}$ as in \cite{SY}. Our main example of interest is the case where $G_n=S_n$, the tower of symmetric groups. 

Note that, unlike the group algebras $\Bbbk G_n$, the Drinfeld doubles $\Drin(G_n)$, varying $n$, do \emph{not} constitute a tower of algebras in the sense of \cite{SY}. Nevertheless, we can use Proposition \ref{centerind-groups} to construct an external tensor product
$$\cZ\Ind_{G_n\times G_m}^{G_{n+m}}\colon \cZ(\Rep_\Bbbk  G_n)\boxtimes \cZ(\Rep_\Bbbk  G_m)\cong \cZ(\Rep_\Bbbk  (G_{n}\times G_m))\longrightarrow \cZ(\Rep_\Bbbk  G_{n+m}).$$
We denote 
\begin{align}
    V\odot W:&=\Ind_{G_n\times G_m}^{G_{n+m}}(V\boxtimes W), &V\in \cZ(\Rep_\Bbbk  G_n), W\in \cZ(\Rep_\Bbbk  G_m),
\end{align}
and define the additive $\Bbbk$-linear category
$$\cZ(\Rep_\Bbbk G_{\geq 0}):=\bigoplus_{n\geq 0} \cZ(\Rep_\Bbbk  G_n).$$
Its objects are formal direct sums $\bigoplus_n V_n$, where $V_n$ is an object of $\cZ(\Rep_\Bbbk  G_n)$ which is zero for all but finitely values of $n$. The morphism spaces in $\cZ(\Rep_\Bbbk G_{\geq 0})$ are given by 
$$\Hom_{\cZ(\Rep_\Bbbk G_{\geq 0})}\Big(\bigoplus_n V_n,\bigoplus_n W_n\Big)=\bigoplus_{n}\Hom_{\cZ(\Rep_\Bbbk  G_n)}(V_n,W_n).$$

\begin{lemma}\label{lem:Ztower}
Given objects $V\in \cZ(\Rep_\Bbbk  G_n), W\in \cZ(\Rep_\Bbbk  G_m), U\in \cZ(\Rep_\Bbbk  G_k)$, there are natural isomorphisms 
\begin{align}
    \alpha_{V,W,U}\colon (V\odot W)\odot U\to V\odot (W\odot U)
\end{align}
in $\cZ(\Rep_\Bbbk G_{n+m+k})$ which satisfy the pentagon axiom of the associativity isomorphism of a monoidal category. In addition, there are coherent natural isomorphisms $V\odot \one\cong V\cong \one \odot V$, for $\one\in \cZ(\Rep_\Bbbk  G_0)$. 

This way, $\cZ(\Rep_\Bbbk G_{\geq 0})$ obtains the structure of an abelian $\Bbbk$-linear symmetric monoidal category with biexact tensor product such that $\End(\one)=\Bbbk$.
\end{lemma}
\begin{proof}
Given a fixed choice of adjunctions of $\Ind_{G_n\times G_m}^{G_{n+m}}$ and $\Res_{G_n\times G_m}^{G_{n+m}}$ we obtain that the two left adjoints $\Ind_{G_{n+m}\times G_k}^{G_{n+m+k}}\circ\big(\Ind_{G_n\times G_m}^{G_{n+m}}\boxtimes \ide\big)$ and $\Ind_{G_{n}\times G_{m+k}}^{G_{n+m+k}}\circ\big(\ide\boxtimes \Ind_{G_m\times G_k}^{G_{m+k}}\big)$ of the corresponding restriction $\Res^{G_{n+m+k}}_{G_n\times G_m\times G_k}$ are canonically naturally isomorphic by the Yoneda Lemma and hence necessarily coherent. Similarly, one proves the coherence of the natural isomorphisms
$$V\odot \one=\Ind_{G_n\times G_0}^{G_n}(V\boxtimes \one)\cong V\cong\Ind_{G_0\times G_n}^{G_n}(\one\boxtimes V)\cong  \one \odot V.$$
The category $\cZ(\Rep_\Bbbk G_{\geq 0})$ is abelian and $\Bbbk$-linear as a direct sum of such categories. Note that $V\odot W$ is exact in both arguments as it is given by the composition of the exact functors of external tensor product \cite{EGNO}*{Section~1.11} and induction. Further, $$\End_{\cZ(\Rep_\Bbbk G_{\geq 0})}(\one)\cong\End_{\cZ(\Rep_\Bbbk  G_{0})}(\one)\cong \End_{\Rep_\Bbbk  G_0}(\one)=\Bbbk.$$
The natural isomorphisms
$$V\odot W=\Ind_{G_n\times G_m}^{G_{n+m}}(V\boxtimes W)\isomorph\Ind_{G_m\times G_n}^{G_{n+m}}(W\boxtimes V) =W\odot V,\quad 
g\otimes (v\boxtimes w) \mapsto g\otimes (w\boxtimes v),
$$
square to the identity and give $\cZ(\Rep_\Bbbk G_\geq 0)$ the structure of a symmetric monoidal category. 
\end{proof}
In the terminology of \cite{EGNO}*{Def.~4.2.3}, $\cZ(\Rep_\Bbbk G_{\geq 0})$ is a \emph{ring category}. Note that it does not have duals since there are no morphisms from $V\odot W$ to $\one\in \cZ(\Rep_\Bbbk  G_0)$ as soon as either $V$ or $W$ are of degree at least $1$.
Moreover, \Cref{lem:Ztower} shows that the category $\cZ(\Rep_\Bbbk G_{\geq 0})$ is an $\mN$-graded monoidal category with respect to the grading given by
$$\cZ(\Rep_\Bbbk G_{\geq 0})^k:=\cZ(\Rep_\Bbbk  G_k).$$

\begin{remark}\label{rem:computation}
In practice, the tensor product $\odot$ and the resulting product on the Grothendieck ring is computed as follows: Set $W=\Ind_{Z(\s)}^{G_n}(V)$ and $W'=\Ind_{Z(\tau)}^{G_m}(V')$ for $\s\in G_n$, $\tau\in G_m$ and $V$ a $Z(\s)$-module, $V'$ a $Z(\tau)$-module. Then, as an $G_{n+m}$-module, 
$$W\odot W'\cong \Ind_{Z(\s)\times Z(\tau)}^{G_{n+m}}(V\boxtimes V').$$
The $G_{n+m}$-coaction for the element $g\otimes (v\otimes v')\in W\odot W'$ is given by
$$\delta(g\otimes (v\otimes v'))= g(\iota_{n,m}(\s,\tau))g^{-1}\otimes (v\otimes v').$$
Thus, one needs to decompose $\Ind_{Z(\s)\times Z(\tau)}^Z(V\boxtimes V')$ into indecomposable $Z$-modules, for $Z=Z(\iota_{n,m}(\s ,\tau))\subseteq G_{n+m}$, in order to compute $W\odot W'$.
\end{remark}

\subsection{An oplax tensor functor from \texorpdfstring{$\cZ(\Rep S_{\geq 0})$}{ZRep S>=0} to \texorpdfstring{$\cZ(\uRep S_t)$}{Z(Rep St)}}\label{sec:oplax}

For the rest of this section, we restrict to the case of the symmetric groups $G_n=S_n$, $\kk=\mC$, and recall the functors
$$\uInd\colon \cZ(\Rep  S_n)\boxtimes \uRep S_{t-n}\longrightarrow \cZ(\uRep S_t)$$
from \Cref{prop:indcenter}. We restrict to a functor 
\begin{align*}
    \uInd\colon \cZ(\Rep  S_n)&\longrightarrow \cZ(\uRep S_t),\qquad  V\longmapsto \uInd(W\boxtimes \one),
\end{align*}
for the tensor unit $\one$ of $\uRep S_{t-n}$ and $V\in \cZ(\Rep  S_n)$. These functors, varying  $n\geq 0$, produce a functor
\begin{align}\label{eq:uInd}
    \uInd\colon \cZ(\Rep S_{\geq 0})&\longrightarrow \cZ(\uRep S_t).
\end{align}

We further recall the filtration $\cZ(\uRep S_t)^{\leq k}$ induced by the filtration on $\uRep S_t$ determined by the number of tensor powers of the generating object $X$, see \Cref{sec:centerlimit}.

\begin{proposition}\label{prop:Indtensorfunctor}
There is a split injective natural transformation  
$$\tau_{V,W}\colon \uInd(V\odot W)\to \uInd (V)\otimes \uInd(W)$$
making $\uInd\colon \cZ(\Rep S_{\geq 0})\to\cZ(\uRep S_t)$ an oplax monoidal functor of braided $\Bbbk$-linear monoidal categories which is compatible with the respective filtrations. 

In particular, the oplax monoidal structure $\tau$ is compatible with the braiding in the sense that the diagram 
\begin{align*}
    \xymatrix{
    \uInd(V\odot W) \ar[rr]^{\uInd(\Psi_{V,W})}\ar@{^{(}->}[d]^{\tau_{V,W}}&& \uInd(W\odot V)\ar@{^{(}->}[d]^{\tau_{W,V}}\\
    \uInd(V)\otimes \uInd(W)\ar[rr]^{\Psi_{\uInd(V),\uInd(W)}}&& \uInd(W)\otimes \uInd(V)
    }
\end{align*}
commutes. Here $\Psi_{V,W}$ is the symmetric braiding on $\cZ\Rep S_{\geq 0}$.
\end{proposition}
\begin{proof}
Let $V\in \cZ(\Rep  S_n)$ and $W\in \cZ(\Rep  S_m)$ be objects. We know by \eqref{ind-fitration-comp} in Section \ref{ind-res-sect} that 
$\uInd(V)\in \cZ(\uRep S_t)^{\leq n}$, $\uInd(W)\in \cZ(\uRep S_t)^{\leq m}$. Further, $V\odot W$ is an object in  $\cZ(\Rep  S_{n+m})$ whence $\uInd (V\odot W)\in \cZ(\uRep S_t)^{\leq m+n}$ and hence $\uInd$ is compatible with the filtrations (using  the induced filtration from the grading on $\cZ\Rep S_{\geq 0}$).

In order to  construct $\tau$ and prove split injectivity, 
we use an ultrafilter description. Without loss of generality, consider $i$ such that $t_i\geq n+m$ and write $\mF_i:=\ov{\mF}_{p_i}$. We note that the trivial module $\mF_i$ is a submodule and quotient of $\Ind_{S_{t_i-n-m}\times S_m}^{S_{t_i-n}}(\mF_i\boxtimes \mF_i)$ and $\Ind_{S_{t_i-n-m}\times S_n}^{S_{t_i-m}}(\mF_i\boxtimes \mF_i)$, respectively, in a canonical way (using the unit of the adjunction $(\Res, \Ind)$, respectively, counit of the adjunction $(\Ind,\Res)$ evaluated on the one-dimensional module).  %
As induction is an exact functor, having a left and right adjoint, this implies that we have a surjective map in $\cZ(\Rep_{p_i} S_{t_i})$
\begin{align*}
&\Ind_{S_n\times S_m\times S_{t_i-n-m}}^{S_{t_i}}(V_i\boxtimes \mF_i\boxtimes \mF_i)
\otimes\Ind_{S_n\times S_m\times S_{t_i-n-m}}^{S_{t_i}}(\mF_i\boxtimes W_i\boxtimes \mF_i)\\
&\cong \Ind_{S_n\times S_{t_i-n}}^{S_{t_i}}(V_i\boxtimes \Ind_{S_m\times S_{t_i-n-m}}^{S_{t_i-n}}(\mF_i\boxtimes\mF_i))
\otimes\Ind_{S_m\times S_{t_i-m}}^{S_{t_i}}(W_i\boxtimes \Ind_{S_n\times S_{t_i-n-m}}^{S_{t_i-m}}(\mF_i\boxtimes \mF_i))\\
&\twoheadrightarrow \Ind_{S_n\times S_{t_i-n}}^{S_{t_i}}(V_i\boxtimes \mF_i)
\otimes\Ind_{S_m\times S_{t_i-m}}^{S_{t_i}}(W_i\boxtimes \mF_i)
\end{align*}
which displays the former as a quotient of the latter. These morphisms are natural in $V_i, W_i$ by construction. We pre-compose with the natural transformation $\delta_{V_i,W_i}$ from \Cref{centerind-groups} to yield the natural transformation 
\begin{align*}
\tau_{V_i,W_i}\colon \Ind_{S_n\times S_m\times S_{t_i-n-m}}^{S_{t_i}}(V_i\boxtimes W_i\boxtimes \mF_i)%
&\longrightarrow \Ind_{S_n\times S_{t_i-n}}^{S_{t_i}}(V_i\boxtimes \mF_i)
\otimes\Ind_{S_m\times S_{t_i-m}}^{S_{t_i}}(W_i\boxtimes \mF_i).
\end{align*}

We claim that $\tau_{V_i,W_i}$ is split injective. To see this, we note the  equality
$$(S_{n}\times S_{t_i-n})\cap (S_{m}\times S_{t_i-m})=S_n\times S_m\times S_{t_i-n-m}$$
of subgroups of $S_{t_i}$, where $S_n$ embeds as $S_{\{1,\ldots,n\}}$ and $S_m$ embeds as $S_{\{n+1,\ldots, n+m\}}$. To show this, take an element $g=(g_n,h_n)=(g_m,h_m)$ in the intersection, with $g_n\in S_n$, $g_m\in S_m$, and $h_n\in S_{t_i-n}$, $h_m\in S_{t_i-m}$. We note that $g_n^{-1}g=(g_m,g_n^{-1}h_m)$ preserves the set $\{1,\ldots, n\}$, which is fixed by $g_m$ by assumption, and hence also preserved by $g_n^{-1}h_m\in S_{t_i-m}$. Thus, $g_n^{-1}h_m\in S_n\times S_m\times S_{t_i-n-m}$, and hence $g\in S_n\times S_m\times S_{t_i-n-m}$.

Next, we recall the elementary observation that for subgroups $K,H$ of $G$, $g,h\in G$, the intersection of cosets $gH\cap hK$ is either empty or a coset $z(K\cap H)$ of $K\cap H$ in $G$, for any $z\in gH\cap hK$. We use this observation as follows. Fix coset decompositions %
$$G=\bigsqcup_\alpha g_\alpha H, \qquad G=\bigsqcup_\beta h_\beta K,$$
to yield a coset decomposition
$$G=\bigsqcup_{(\alpha,\beta)} g_\alpha H\cap h_\beta K= \bigsqcup_{(\alpha,\beta)}\s_{\alpha,\beta} H\cap K,$$
over the pairs $(\alpha,\beta)$ for which the intersections of cosets is non-empty, and
where
$$\s_{\alpha,\beta}=g_\alpha h_{\alpha,\beta}=h_{\beta}k_{\alpha,\beta}\in g_\alpha H\cap h_\beta K,$$
for some elements $h_{\alpha,\beta}\in H$, $k_{\alpha,\beta}\in K$. We set $H=S_n\times S_{t_i-n}$ and $K=S_m\times S_{t_i-m}$ viewed as subgroups of $G=S_{t_i}$.

Observe that a basis for $\Ind_{S_n\times S_m\times S_{t_i-n-m}}^{S_{t_i}}(V_i\boxtimes W_i\boxtimes \mF_i)$ is given by the set 
$$\{\s_{\alpha,\beta}\otimes (v_\gamma\boxtimes w_\delta\boxtimes 1)\}_{(\alpha, \beta, \gamma, \delta)},$$
where $\{v_\gamma\}_\gamma$ and $\{w_\delta\}_\delta$ are bases of $V_i$ and $W_i$, respectively, and $(\alpha,\beta)$ with non-empty coset intersection as above. We consider the images of these basis elements under $\tau_{V_i,W_i}$. Using the formula for $\delta_{V_i,W_i}$ in \Cref{prop:indcenter}, we compute that
\begin{align*}
    \tau_{V_i,W_i}\big(\s_{\alpha,\beta}\otimes (v_\gamma\boxtimes w_\delta\boxtimes 1)\big)&=\big(\s_{\alpha,\beta}\otimes (v_\gamma\boxtimes 1)\big)\otimes\big( \s_{\alpha,\beta}\otimes ( w_\delta\boxtimes 1)\big)\\
    &=\big(g_\alpha h_{\alpha,\beta}\otimes (v_\gamma\boxtimes 1)\big)\otimes\big( h_{\beta}k_{\alpha,\beta}\otimes ( w_\delta\boxtimes 1)\big)\\
    &=\big(g_\alpha \otimes (v_\gamma\boxtimes 1)\big)\otimes\big( h_{\beta}\otimes ( w_\delta\boxtimes 1)\big).
\end{align*}
Thus, the map $\tau_{V_i,W_i}$ sends the given basis to a subset of a basis for the target space $\Ind_{S_n\times S_{t_i-n}}^{S_{t_i}}(V_i\boxtimes \mF_i)
\otimes\Ind_{S_m\times S_{t_i-m}}^{S_{t_i}}(W_i\boxtimes \mF_i)$. Hence, $\tau_{V_i,W_i}$ is injective.

Next, we claim that a complement of the image of $\tau_{V_i,W_i}$ is closed under the $S_{t_i}$-action. We have seen that the image of $\tau_{V_i,W_i}$ is spanned over $\mF_i$ by those vectors
$$u_{\alpha,\beta,\gamma,\delta}:=\big(g_\alpha \otimes (v_\gamma\boxtimes 1)\big)\otimes\big( h_{\beta}\otimes ( w_\delta\boxtimes 1)\big)$$
for which $g_\alpha H\cap h_\beta K\neq \varnothing$, or equivalently, $g_\alpha^
{-1}h_\beta\in H\cap K$. Those vectors $u_{\alpha,\beta,\gamma,\delta}$ for which $g_\alpha^
{-1}h_\beta\not\in H\cap K$ clearly span a $G$-invariant complement. 

If $\{v_\gamma\}_\gamma$ and $\{w_\delta\}_\delta$ are chosen to be $G$-homogeneous bases, this complement is also a $G$-graded subspace, hence, a Yetter--Drinfeld submodule, showing that $\tau_{V_i,W_i}$ splits as a morphism in $\cZ(\Rep_{p_i} S_{t_i})$.

Next, we observe that $\tau_{V_i,W_i}$ is oplax monoidal. This is a consequence of $\delta_{V_i,W_i}$ being oplax monoidal mapping to a larger object in $\cZ(\Rep_{p_i} S_{t_i})$ by \Cref{prop:indcenter} and naturality of $\uInd(X\boxtimes Y)$ in $Y\in \uRep S_{t-n}$. 

Using \eqref{braidcomp} and naturality of the braiding of $\cZ(\uRep S_t)$, we see that this oplax monoidal structure is compatible with the braidings, using the symmetric braiding on $\cZ\Rep S_{\geq 0}$.

By \L o\'s' theorem, we obtain an induced injective natural transformation $\tau_{V,W}$ as claimed. This structure makes $\uInd$ an oplax monoidal functor compatible with braiding as claimed.
\end{proof}

\subsection{The Grothendieck ring of \texorpdfstring{$\cZ(\uRep S_t)$}{Z(Rep St)}}

Recall that 
$$\gr K_0^\oplus (\uRep S_t)\cong \bigoplus_{k\geq 0}K_0(\Rep  S_n)$$
from \Cref{sec:K0RepSt}, where the right-hand side is the ring of symmetric functions.
To study the Grothendieck ring of the center $\cZ(\uRep S_t)$, we first note that the direct sum of Grothendieck rings
$$K_0(\cZ (\Rep S_{\geq 0}))=\bigoplus_{n\geq 0} K_0\big(\cZ(\Rep  S_n)\big)$$
obtains the structure of a graded commutative algebra with the product
\begin{align}
    [V]\cdot [W]:=[V\odot W].
\end{align}
The unit is given by $[\one]$, for $\one\in \cZ(\Rep  S_0)$.

\begin{theorem}\label{thm:K0}
The functor $\uInd$ from \eqref{eq:uInd} induces an isomorphism of graded rings
$$\gr K_0^{\oplus}(\uInd) \colon K_0(\cZ\Rep S_{\geq 0})\isomorph \gr K_0^\oplus\cZ(\uRep S_t)),$$
where the associated graded of $K_0^\oplus\cZ(\uRep S_t))$ is taken with respect to the filtration induced by the filtration $\cZ(\uRep S_t)^{\leq k}$ of categories.
\end{theorem}
\begin{proof}
The functor $\uInd$ induces a morphism of abelian groups
$$\gr K_0^{\oplus}(\uInd)\colon K_0(\cZ\Rep S_{\geq 0})\longrightarrow \gr K_0 (\cZ(\uRep S_t)).$$
We first show $\gr K_0^{\oplus}(\uInd)$ is an algebra map. Recall the split injective natural transformations 
$$\tau_{V,W}\colon \uInd(V\odot W)\to \uInd (V)\otimes \uInd(W)$$
for $V\in\cZ(\Rep  S_n)$ and $W\in\cZ(\Rep  S_m)$ from \Cref{prop:Indtensorfunctor} which gives an oplax monoidal structure. 
We use $\tau_{V,W}$  to show that 
$$[\uInd(V\odot W)]=[\uInd (V)\otimes \uInd(W)]$$
in $\gr K_0^\oplus(\cZ(\uRep S_t))$. 
For this, we will first look at how these objects decompose inside of $\uRep S_t$ which will help us identify the highest degree parts.
To this end, we decompose 
$$V=\bigoplus_i S^{\lambda_i},\qquad  W=\bigoplus_j S^{\mu_j},$$
as direct sums of simple modules in $\Rep  S_n$ and $\Rep  S_m$, respectively. By the Littlewood--Richardson rule, we have the decomposition
$$V\odot W=\Ind_{S_n\times S_m}^{S_{n+m}}(\bigoplus_i S^{\lambda_i}\boxtimes \bigoplus_j S^{\mu_j})=\bigoplus_{i,j,\nu}(S^\nu)^{\oplus c_{\lambda_i,\mu_j}^{\nu}}$$
in $\Rep  S_{n+m}$. By \Cref{lem:inddec} we have, decomposing as elements in $K_0^{\oplus}(\uRep S_t)$,
\begin{gather*}
[\uInd (V)]=\sum_i [X_{\lambda_i}]+ \{\text{l.o.t.}\},\qquad [\uInd(W)]=\sum_j [X_{\mu_j}]+ \{\text{l.o.t.}\},\\
    [\uInd(V\odot W)]=\sum_{i,j,\nu}c_{\lambda_i,\mu_j}^{\nu}[X_\nu]+ \{\text{l.o.t.}\},
\end{gather*}
where the lower order terms $\{\text{l.o.t.}\}$ are in strictly lower filtration pieces $K_0^{\oplus}(\uRep S_t)^{\leq k}$.
Further, we have by \Cref{lem::K0-RepSt} that in $K_0^\oplus(\uRep S_t)$,
$$[X_{\lambda_i}][X_{\mu_j}]=[X_{\lambda_i}\otimes X_{\mu_j}]=\sum_{\nu} c_{\lambda_i,\mu_j}^{\nu}[X_{\nu}]+ \{\text{l.o.t.}\}.$$
Thus, all direct summands, as objects of $\uRep S_t$, that have maximal degree $n+m$ in $\uInd(V)\otimes \uInd(W)$, are in fact already contained in $\Ind(V\odot W)$ and hence are contained in the image of $\tau_{V,W}$ using split injectivity.

Now, we may decompose $$\uInd(V)\otimes \uInd(W)=\uInd(V\odot W)\oplus Y,$$
 in $\cZ(\uRep S_t)$,
where $Y$ is a complement of the image of $\tau_{V,W}$ (see \Cref{prop:Indtensorfunctor}).
We want to identify all \emph{leading} terms, i.e., those that are contained in the leading filtration piece $K_0^\oplus(\cZ(\uRep S_t))^{\leq n+m}$ but not contained in $K_0^\oplus(\cZ(\uRep S_t))^{\leq n+m-1}$. By the Krull--Schmidt property of the category  $\cZ(\uRep S_t)$, the decomposition in  $\cZ(\uRep S_t)$ refines into a decomposition into indecomposables in $\uRep S_t$. In such a refinement, any leading direct summand in $\cZ(\uRep S_t)$ needs to contain at least a direct summand from $(\uRep S_t)^{n+m}$ that is not contained in $(\uRep S_t)^{n+m-1}$. However, all such summands occurring in $\uInd(V)\o\uInd(W)$ are already contained in $\uInd(V\odot W)$ by the above observations. This implies that 
$$[\uInd(V)\otimes \uInd(W)]=[\uInd(V\odot W)]+\{\text{l.o.t.}\}$$
in $K_0^\oplus (\cZ(\uRep S_t)).$
Thus, the functor $\uInd$ induces an algebra map to the associated graded Grothendieck ring as claimed. 

\smallskip

It remains to show that $\gr K_0^{\oplus}(\uInd)$ is an isomorphism.
Indeed, using \Cref{prop:oldgenall} we see that $\gr K_0^{\oplus}(\uInd)$ is surjective, since by combining \Cref{indec-classification} and \Cref{uWdistinct}, we know that a $\mZ$-basis for $K_0^\oplus(\cZ(\uRep S_t))$ is given by the objects $\uW{\mu,V,U}$ for $V,U$ indecomposable and $\mu$ singleton free.

To prove injectivity of $\gr K_0^{\oplus}(\uInd)$ we recall the notation from the proof of \Cref{prop:oldgenall}. Given a simple $Z(\mu)$-module $V$, with $\mu$ not necessarily singleton free, decompose $Z(\mu)\cong Z_0\times S_r$ and $V\cong V_0\boxtimes U_0$. Since we are working over $\mC$ and $V$ is simple, $U_0$ is simple as a $\mC S_r$-module. Thus, by \Cref{lem:inddec}, $U_{\mu,V}=\uInd_{S_r\times S_{t-n}}^{S_{t-n_0}}(U_0\boxtimes \one)$ contains a unique simple summand $X$ of maximal filtration degree $r$. This implies that $\uW{\mu_0,V_0,X}$ is a subobject of $\uInd(\Ind_{Z(\mu)}^{S_n}V)=\uWold{\mu,V}$, which has maximal filtration degree $n=n_0+r$ by \Cref{cor:inddec}. By the same reasoning, decomposing $[\uWold{\mu,V}]$ in $K_0^\oplus(\cZ(\uRep S_t))$, all other summands are of strictly smaller filtration degree. Thus, $[\uInd(\Ind_{Z(\mu)}^{S_n}V)]=[\uW{\mu_0,V_0,X}]$ in $\gr K_0^\oplus(\cZ(\uRep S_t))$. The datum $(\mu_0,V_0,X)$ uniquely determines $(\mu, V)$ since $X$ determines $r,U_0$ and $V\cong V_0\boxtimes U_0$.
Thus, since the $[\uW{\mu_0,V_0,X}]$ are linearly independent in $\gr K_0^\oplus\cZ(\uRep S_t))$, the map $\gr K_0^{\oplus}(\uInd)$ is injective. 
\end{proof}

The description of the additive Grothendieck ring of $\cZ(\uRep S_t)$ can be extended to the Grothendieck ring of the abelian envelope as follows:

\begin{theorem}\label{thm:K0-ab}
For any $d\in\mZ_{\ge0}$, the functor $\uInd$ from \eqref{eq:uInd} induces an isomorphism of graded rings
$$\gr K_0^{\oplus}(\uInd)\colon K_0(\cZ\Rep S_{\geq 0})\isomorph \gr K_0(\cZ(\uRep^\ab S_d)),$$
where the associated graded of $K_0(\cZ(\uRep^\ab S_d))$ is taken with respect to the filtration induced by the filtration $\cZ(\uRep^\ab S_d)^{\leq k}$ of categories.
\end{theorem}

\begin{proof}
First recall that the functor $\uInd$, together with the split injective oplax monoidal structure $\tau$ from \Cref{prop:Indtensorfunctor}, extends to the abelian envelope as via composition
$$
\uInd\colon \cZ\Rep S_{\geq 0}\longrightarrow \cZ(\uRep S_t)\hookrightarrow \cZ(\uRep^\ab S_d),
$$
cf. the diagram \eqref{uInd-abelianized}.

The proof of \Cref{thm:K0} can be adapted to working with the abelian envelope using \Cref{cor:Ko(St-ab)} instead of \Cref{lem::K0-RepSt}. Instead of the Krull--Schmidt property, we use the Jordan--H\"older property of the category $\cZ(\uRep^\ab S_d)$. 

For the proof in the case of the abelian envelope we also used that by \Cref{prop::X-d-plus-1-in-center}(3), all projective objects are contained in the subcategory $\cZ(\uRep S_d)$. Thus, all simple objects occur as subquotients of the objects $\uInd(V)$ by \Cref{prop:oldgenall}. This shows that $\gr K_0(\uInd)$ is an isomorphism in the case of the abelian envelope.
\end{proof}

\subsection{Some sample computation in the associated graded Grothendieck ring}\label{sec:computations}

As described in \Cref{rem:computation}, the induction product $W\odot W'$, for $W=\Ind_{Z(\s)}^{S_n}(V)\in \cZ(\Rep S_n)$ and $W'=\Ind_{Z(\tau)}^{S_m}(V')\in \cZ(\Rep S_m)$ can be computed by decomposing
$$\Ind_{Z(\s)\times Z(\tau)}^{Z(\s\times\tau)}(V\boxtimes V')$$
as a module over the centralizer $Z(\s\times\tau)\subset S_{n+m}$, where $\s\times \tau$ permutes $\{1,\ldots, n\}$ using $\s$ and $\{n+1,\ldots, n+m\}$ using $\tau$. With \Cref{thm:K0}, this computes products in the associated graded Grothendieck ring $\gr K_0^\oplus (\cZ(\uRep S_t))$ or its abelian envelope. In this section, we consider a few examples. We start with the smallest non-trivial example.

\begin{example}
Let $\s=\tau=(12)\in S_2$. Then $\s\times\tau=(12)(34)\in S_4$ and $Z((12)(34))$ is the wreath product 
$$\mZ_2\wr S_2\cong \langle (12),(34), (13)(24)\rangle\subset S_4.$$ 
Using \cite{CM}*{Proposition~3.7}, there are five simple modules of $\mZ_2\wr S_2$, one two-dimensional module $V_2$, and four one-dimensional modules $V^{\epsilon_1,\epsilon_2}$, where $\epsilon_i\in \{\pm 1\}$, and $(12),(34)$ act via multiplication by $\epsilon_1$, and $(13)(24)$ acts via multiplication by $\epsilon_2$.

We can choose $V, V'$ to be either the trivial module $\Bbbk^{\triv}$ or the sign module $\Bbbk^{\sgn}$ of $S_2=Z(\s)=Z(\tau)$. The resulting module $V\odot V'=\Ind_{S_2\times S_2}^{\mZ_2\wr S_2}(V\boxtimes V')$ is two-dimensional, with a basis given by 
$$v_1=1\otimes (v\boxtimes v'),\quad  v_2=(13)(24)\otimes (v\boxtimes v'),$$
where $v$ generates $V$ and $v'$ generates $V'$.

If $V=V'=\kk^{\triv}$, then $V\odot V'\cong V^{+1,+1}\oplus V^{+1,-1}$ is a direct sum of the trivial and sign module over $\mZ_2\wr S_2$, splitting as $\Bbbk(v_1+v_2)\oplus \Bbbk(v_1-v_2)$.
Similarly, if $V=V'=\kk^{\sgn}$, then  $V\odot V'\cong V^{-1,+1}\oplus V^{-1,-1}$.
If $V\neq V'$, then $V\odot V'\cong V'\odot V$ is the simple two-dimensional module $V_2$ over $\mZ_2\wr S_2$.
\end{example}

We include the product computations in a slightly more general situation involving cyclic permutations. 
\begin{example}
Assume that $\s$ and $\tau$ are $k$-cycles in $S_k$. Then their centralizers are cyclic groups isomorphic to $\mZ_k$. Thus, 
$$Z(\s\times\tau)=\mZ_k\wr S_2=\langle \s,\tau, \omega \rangle \subset S_{2k},$$
where conjugation by $\omega$ swaps $\s$ and $\tau$.
Let $\Bbbk^{\zeta}$  denote the one-dimensional simple $\mZ_k$-module where the generator acts through a root of unity $\zeta\in \Bbbk$.
Again using, e.g., \cite{CM}*{Proposition 3.7}, the simple modules of $\mZ_k\wr S_2$ fall into two classes. First, two-dimensional simples
\begin{align*}
V^{\zeta_1,\zeta_2}_2=\Ind_{\mZ_k\times\mZ_k}^{\mZ_k\wr S_2}\left(\Bbbk^{\zeta_1}\boxtimes \Bbbk^{\zeta_2}\right),
\end{align*}
where $\zeta_1,\zeta_2$ are distinct $k$-th roots of unity; second, one-dimensional simples 
$V^{\zeta, \epsilon}$, where both copies of $\mZ_k$ act via multiplication by $\zeta$ and $S_2$ acts via $\kk^{\triv}$ if $\epsilon=1$ and via $\kk^{\sgn}$ if $\epsilon=-1$. By \cite{CM}*{Lemma~3.1}, $V_2^{\zeta_1,\zeta_2}\cong V_2^{\zeta_2,\zeta_1}$.

We see that for all $k$-th roots of unity $\zeta$ and $\zeta_1\neq \zeta_2$,
\begin{align*}
\Bbbk^{\zeta}\odot \Bbbk^{\zeta} &=\Ind_{\mZ_k\times \mZ_k}^{\mZ_k\wr S_n}(\Bbbk^{\zeta}\boxtimes \Bbbk^{\zeta})\cong V^{\zeta,+1}\oplus V^{\zeta,-1},&
\Bbbk^{\zeta_1}\odot \Bbbk^{\zeta_2}&\cong V_2^{\zeta_1,\zeta_2}\cong V_2^{\zeta_2,\zeta_1}\cong \Bbbk^{\zeta_2}\odot \Bbbk^{\zeta_1}.
\end{align*}
\end{example}

If the cycle types of $\s$ and $\tau$ have no common cycle length (including one-cycles) then $Z(\s\times \tau)=Z(\s)\times Z(\tau)$ and the induction product $\odot$ is simply given by the exterior tensor product $\boxtimes$. 

In general, the representation theory of the products of wreath product groups that appear as centralizers in the symmetric groups is well understood and induction products may be computed using similar but more involved analogues of Littlewood--Richardson coefficients for wreath products.

\appendix

\section{The abelian envelope of the monoidal center}\label{app:ZC}

In this section, we collect general results about the monoidal center required in the core of the paper. Given a monoidal category $\cC$, the center $\cZ(\cC)$ is a braided monoidal category \cites{Maj2, JS}. Objects in $\cZ(\cC)$ are pairs $(Y,c)$ where $Y$ is an object of $\cC$ and $c=\Set{c_V\colon Y\otimes V\to V\otimes Y}_{V\in \cC}$ is a natural isomorphism (in $V$), called a \emph{half-braiding}, which satisfies the tensor compatibility 
\begin{align*}c_{V\otimes W}&=(\ide_V\otimes c_W)(c_V\otimes \ide_W), &\forall V,W\in \cC,
\end{align*}
where the associativity isomorphism of $\cC$ is omitted. A morphism $f\colon (Y,c)\to (Y',c')$ is a morphism $f\colon Y\to Y'$ in $\cC$ such that 
\begin{align*}
c'_V(f\otimes \ide_V)&=(\ide_V\otimes f)c_V,& \forall V\in \cC.
\end{align*}
For basic properties of the monoidal center see e.g.~\cite{EGNO}*{Section 7.13}. We employ the following extension property of the monoidal center.

\begin{lemma}\label{center-extend}
Let $\cC$ be a locally finite abelian rigid monoidal category with enough projectives and $\cA$ be an additive tensor subcategory of $\cC$ containing all projective objects.

Assume given a pair $(Y,\left.c\right|_{\cA})$ where $Y$ is an object in $\cC$ and 
$$\left.c\right|_{\cA}\colon Y\otimes \ide_{\cA}\xrightarrow{~\sim~}\ide_{\cA} \otimes Y$$
is a natural isomorphism which is tensor compatible in the sense that 
\begin{align*}
    \big(\left.c\right|_{\cA}\big)_{A\otimes A'}&=\Big(\ide_{A}\otimes \big(\left.c\right|_{\cA}\big)_{A'}\Big)\Big(\big(\left.c\right|_{\cA}\big)_A\otimes \ide_{A'}\Big),& \forall A,A'\in \cA.
    \end{align*}
Then $\left.c\right|_{\cA}$ admits a unique extension to a half-braiding defining an object in $\cZ(\cC)$.
\end{lemma}

In particular, the inclusion functor $\cA\hookrightarrow \cC$ extends to an inclusion functor $\cZ(\cA)\hookrightarrow \cZ(\cC).$
\begin{proof}
Assume given an object $Y$ and
$$\left.c\right|_{\cA}\colon Y\otimes \ide_{\cA}\xrightarrow{\sim} \ide_{\cA}\otimes Y $$
as in the statement of the lemma.
We note that the functors of left and right tensoring with $Y$ are exact (see e.g.~\cite{BK}*{Proposition 2.1.8}). We can  restrict $c$ to a natural isomorphism of these functors on the full monoidal subcategory $\cP$ on projective objects of $\cC$, which is contained in $\cA$ by assumption.

Now let $M$ be an object of $\cC$. Then we can find projective objects $P_1,P_0$ and an exact sequence 
$$P_1\to P_0\to M\to 0.$$
We note that $P_1,P_0$ are both finite sums of indecomposable projectives by assumption of local finiteness of $\cC$ and therefore the argument of \cite{Mit}*{Theorem 5.4} can be adapted to extend the  $\left.c\right|_{\cA}$ to the full subcategory of projective objects to all of $\cC$ in a unique way to a natural isomorphism
$$c\colon Y\otimes \ide_{\cC}\xrightarrow{\sim} \ide_{\cC}\otimes Y.$$
It remains to check that $c$ satisfies tensor compatibility and thus gives a half-braiding on all of $\cC$. But this follows from  exactness of the tensor product. In fact, for objects $M,N$ of $\cC$, given exact sequences 
$$P_1\to P_0\to M\to 0, \qquad Q_1\to Q_0\to N\to 0,$$
with $P_i,Q_i$ projective, exactness of $\otimes$ gives an exact sequence 
$$P_1\otimes Q_0\oplus P_0\otimes Q_1\to P_0\otimes Q_0\to M\otimes N\to 0.$$
Now both $c=c_{M\otimes N}$ and $c=(c_{M}\otimes \ide)(\ide\otimes c_N)$ make the following diagram commute:
\begin{align*}
    \xymatrix{
Y\otimes P_1\otimes Y\otimes Q_0\oplus Y\otimes P_0\otimes Q_1\ar[rrrr]\ar[d]_{\mat{c_{P_1\otimes Q_0}&0\\0&c_{P_0\otimes Q_1}}=}^{\mat{c_{P_1}\otimes \ide&0\\0&c_{P_0\otimes \ide}}\mat{\ide\otimes c_{Q_0}&0\\0&\ide\otimes c_{Q_1}}}&&&& Y\otimes P_0\otimes Q_0\ar[r]\ar[d]_{c_{P_0\otimes Q_0}=}^{(c_{P_0}\otimes \ide)(\ide\otimes c_{Q_0})}& Y\otimes M\otimes N\ar[d]^{c}\\
P_1\otimes Y\otimes Q_0\otimes Y\oplus P_0\otimes Q_1\otimes Y\ar[rrrr]&&&& P_0\otimes Q_0\otimes Y\ar[r]& M\otimes N\otimes Y
}
\end{align*}
Thus, these morphisms have to be equal by uniqueness of the morphism $c$. 
\end{proof}

\begin{corollary}\label{cor:ZReppSd-extend}
Let $X$ be a faithful representation in $\Rep_p G$. Assume given two objects $(Y,c)$, $(Y,c')$ in $\cZ(\Rep_p G)$. Then $c_X=c_X'$ implies $c=c'$.
\end{corollary}
\begin{proof}
Note that $\cA=\langle X \rangle$ contains all projective objects \cite{BK1}*{Theorem 1}. The half-braidings are determined by $c_{X^{\otimes n}}$ which in turn is determined by $c_X$.
\end{proof}

In the following, let $\cA$ be a Karoubian tensor category. We recall that an (abelian) multitensor category $\cA^\ab$ with a fully faithful tensor functor $\iota\colon\cA\to\cA^\ab$ is called an \emph{abelian envelope} (\cite{BEO}*{Section~2.10}) of $\cA$ if for any multitensor category $\cD$, the category of tensor functors $\cA^\ab\to\cD$ is equivalent to the category of faithful monoidal functors $\cA\to\cD$ by restriction along $\iota$. If it exists, the abelian envelope is unique up to equivalence. If the abelian envelope exists and has enough projectives, then a construction is given in \cite{BEO}. A necessary and sufficient condition on $\cA$ (``separated and complete'') for $\cA^\ab$ to have enough projectives is proven in \emph{loc.cit.}
In the following, we will identify $\cA$ with a full tensor subcategory of $\cA^\ab$ using $\iota$.

\begin{lemma} \label{lem::ab-enough-projectives} Assume $\cA^\ab$ is the abelian envelope of $\cA$ and  has enough projectives. Assume also that $\cA$ admits a braiding $c$ such that the functor $\cA^\ab\to\cZ(\cA^\ab)$, $X\mapsto (X,c_{X,-})$, preserves projectives. Then $\cZ(\cA^\ab)$ has enough projectives and all its projectives lie in $\cZ(\cA)$.
\end{lemma}

\begin{proof} We recall that the embedding of $\cA$ into $\cA^\ab$ is full on projectives by \cite{BEO}*{Theorem~2.41, Theorem~2.42}, so all projectives of $\cA^\ab$ lie in $\cA$.

As $\cZ(\cA^\ab)$ is a multitensor category, having enough projectives is equivalent to the existence of a projective object in $\cZ(\cA^\ab)$ with an epimorphism to the tensor unit in $\cZ(\cA^\ab)$.

The tensor unit of $\cZ(\cA^\ab)$ is given by the tensor unit $\one$ of $\cA^\ab$ together with the natural isomorphisms between the functors $\ide\o\one \cong \ide \cong \one\o \ide$. Since $\cA^\ab$ has enough projectives, all of which lie in $\cA$, there is a projective object $X\in\cA^\ab$ which lies in $\cA$ with an epimorphism to $\one\in\cA^\ab$. Now $\cA$ admits a braiding which yields a half-braiding $c$ for $X$. Then $(X,c)$ defines an object in the center of $\cA$, hence by \Cref{center-extend}, this is also an object in the center of $\cA^\ab$. 
By our assumptions, $(X,c)$ is a projective object in $\cZ(\cA^\ab)$ with an epimorphisms to the tensor unit in $\cZ(\cA^\ab)$, so $\cZ(\cA^\ab)$ has enough projectives.

Now consider any projective object $(P,d)$ in $\cZ(\cA^\ab)$. Then $(P,d)\o (X,c)$ is a projective object in $\cZ(\cA^\ab)$ with $(P,d)$ as a quotient, hence, as a direct summand. However, $P\o X$ is a projective object in $\cA^\ab$, hence it lies in $\cA$. This means $(P,d)$ appears as a direct summand in an object of $\cZ(\cA)$, hence, it lies in $\cZ(\cA)$, and we have shown that all projective objects of $\cZ(\cA^\ab)$ lie in $\cZ(\cA)$.
\end{proof}

\begin{corollary} \label{cor::center-abelian-envelope-general} In the situation of \Cref{lem::ab-enough-projectives}, $\cZ(\cA^\ab)$ is the abelian envelope of $\cZ(\cA)$.
\end{corollary}

\begin{proof} By \cite{BEO}*{Theorem~2.42}, the abelian envelope of a Karoubian tensor category $\cC$, if it exists and has enough projectives, is given by any fully faithful monoidal functor $E\colon\cC\to\cD$, where $\cD$ is a multitensor category with enough projectives; in this case, the abelian envelope is the abelian tensor subcategory generated by the image of $E$.

Set $\cD:=\cZ(\cA^\ab)$ and let $E\colon\cZ(\cA)\to\cZ(\cA^\ab)=\cD$ be the functor induced by the inclusion $\cA\to\cA^\ab$ according to \Cref{center-extend}. By \Cref{lem::ab-enough-projectives}, $\cD$ is a multitensor category with enough projectives, the projectives lying in $\cZ(\cA)$. In particular, every object is a quotient of an object in $\cZ(\cA)$ and the subcategory generated by the image of $E$ is all of $\cD$.
\end{proof}

\section{Separable Frobenius monoidal functors and the monoidal center}\label{app:frob}

In this section, we recall the definition of a (separable) Frobenius monoidal functor and show that induction functors of finite group representation display such a structure which extends to their monoidal centers. 

A \emph{Frobenius monoidal functor} $F\colon \cC\to \cD$ between two monoidal categories $\cC$, $\cD$ is a bilax monoidal functor, i.e., comes with a lax monoidal structure $(\mu,\eta)$, and an oplax monoidal structure $(\delta, \epsilon)$, where
\begin{gather}
    \mu_{V,W}\colon F(V)\otimes F(W)\longrightarrow F(V\otimes W), \qquad \delta_{V,W}\colon F(V\otimes W)\longrightarrow F(V)\otimes F(W),\\
    \eta\colon \one \longrightarrow F(\one), \qquad \epsilon\colon F(\one)\longrightarrow \one,
\end{gather}
for any objects $V,W$ of $\cC$,
satisfying the additional compatibility conditions
\begin{gather}\label{frobmon1}
\vcenter{\hbox{\xymatrix{
&&F(V)\otimes F(W)\otimes F(U)\ar[rrd]^{\mu_{V,W}\otimes \ide_{F(U)}}&&\\
F(V)\otimes F(W\otimes U)\ar[rru]^{\ide_{F(V)}\otimes \delta_{W,U}}\ar[rrd]_{\mu_{V,W\otimes U}} &&&& F(V\otimes W)\otimes F(U),\\
&&F(V\otimes W\otimes U)\ar[rru]_{\delta_{V\otimes W,U}}&&
}}}
\end{gather}
\begin{gather}
\vcenter{\hbox{\xymatrix{
&&F(V)\otimes F(W)\otimes F(U)\ar[rrd]^{\ide_{F(V)}\otimes \mu_{W,U}}&&\\
F(V\otimes W)\otimes F(U)\ar[rru]^{\delta_{V,W}\otimes \ide_{F(U)}} \ar[rrd]_{\mu_{V\otimes W,U}} &&&& F(V)\otimes F(W\otimes U).\\
&&F(V\otimes W\otimes U)\ar[rru]_{\delta_{V,W\otimes U}}&&
}}}\label{frobmon2}
\end{gather}
For details on this definition see e.g.~\cite{AA}*{Section~3.5}, \cite{DP}. A Frobenius monoidal functor is \emph{separable} if, in addition, for any objects $V,W$ of $\cC$,
\begin{gather}\label{eq:sep}
    \mu_{V,W}\circ\delta_{V,W}= \ide_{F(V\otimes W)}.
\end{gather}

Examples of Frobenius monoidal functors are obtained from the induction functor of representation categories of finite groups. We prove here that these Frobenius monoidal functors extend to the centers of the representation categories. This functor, without its Frobenius monoidal structure, already appeared in \cite{Dav}*{Theorem 3.3.2}.
In the following, $\Bbbk$ is any field.
\begin{proposition}\label{centerind-groups}
Let $\Bbbk$ be a field and $H\subseteq G$ be finite groups. Then the induction functor $\Ind_{H}^G$ induces a separable Frobenius monoidal functor 
$$\cZ(\Ind_H^G)\colon \cZ(\Rep_\Bbbk H)\longrightarrow \cZ(\Rep_\Bbbk  G).$$
The lax and oplax monoidal structures $\mu$, $\delta$ are compatible with the braiding in the sense that the diagrams
\begin{gather}\label{braidcomp}
    \vcenter{\hbox{
    \xymatrix{
    \Ind_H^G(V\otimes W)\ar[d]_{\delta_{V,W}}\ar[rrr]^{\Ind_H^G(\Psi_{V,W})}&&&  \Ind_H^G(W\otimes V)\ar[d]_{\delta_{W,V}}\\
    \Ind_H^G(V)\otimes \Ind_H^G(W)\ar[rrr]^{\Psi_{\Ind_H^G(V),\Ind_H^G(W)}}&&&\Ind_H^G(W)\otimes \Ind_H^G(V)
    }
    }}
    \end{gather}
    \begin{gather}
\label{braidcompmu}
    \vcenter{\hbox{
    \xymatrix{
    \Ind_H^G(V\otimes W)\ar[rrr]^{\Ind_H^G(\Psi_{V,W})}&&&  \Ind_H^G(W\otimes V)\\
    \Ind_H^G(V)\otimes \Ind_H^G(W)\ar[rrr]^{\Psi_{\Ind_H^G(V),\Ind_H^G(W)}}\ar[u]_{\mu_{V,W}}&&&\Ind_H^G(W)\otimes \Ind_H^G(V)\ar[u]_{\mu_{W,V}}
    }
    }}
\end{gather}
commute for any objects $V,W$ in $\cC$.
\end{proposition}
\begin{proof}
Given $(V,c)$ in $\cZ(\Rep_\Bbbk H)$, $\Ind_H^G(V)=\Bbbk G\otimes_{\Bbbk H}V$ can be equipped with the morphism
$$c'_X ((g\otimes v)\otimes x)= (g|v|g^{-1}\cdot x)\otimes v,$$
for $X$ any $\Bbbk G$-module and regarding the degree $|v|\in H\subseteq G$ as an element in $G$. One checks that $c'_X$ defines a half-braiding. In fact, under the braided equivalence of $\cZ(\Rep_\Bbbk G)$ and the category of Yetter--Drinfeld modules over $G$, this half-braiding corresponds to the Yetter--Drinfeld module with coaction given by 
$$\delta'(g\otimes v)=g|v|g^{-1}\otimes (g\otimes v).$$
This construction is clearly functorial with respect to morphisms in $\cZ(\Rep_\Bbbk H)$. 

The functor $\Ind_H^G\colon \Rep_\Bbbk  H\to \Rep_\Bbbk  G$ is both a left and right adjoint to the monoidal functor $\Res_H^G$. Thus, $\Ind_H^G$ is both lax and oplax monoidal. Explicitly, the lax and oplax structures are given by 
\begin{gather*}
    \mu_{V,W} \colon \Ind_H^G(V)\otimes \Ind_H^G(W)\longrightarrow \Ind_H^G(V\otimes W), \qquad \eta \colon \Bbbk \longrightarrow \Ind_H^G(\Bbbk),\\
    (g\otimes v)\otimes (k\otimes w)\xmapsto{\,\mu_{V,W}\,} \begin{cases} h\otimes (v\otimes g^{-1}kw), & \text{if }g^{-1}k\in H\\ 0, & \text{otherwise}\end{cases}, \qquad 1\xmapsto{\,\eta\,} \sum_{j} g_j\otimes 1,
    \\
    \delta_{V,W} \colon \Ind_H^G(V\otimes W)\longrightarrow \Ind_H^G(V)\otimes \Ind_H^G(W),\qquad \epsilon \colon  \Ind_H^G(\Bbbk)\longrightarrow \Bbbk,\\
    g\otimes (v\otimes w)\xmapsto{\,\delta_{V,W}\,} (g\otimes v)\otimes (g\otimes w), \qquad g\otimes 1 \xmapsto{\,\epsilon\,} 1.
\end{gather*}
where $\{g_j\}_{j\in J}$ is a set of representatives for the left cosets of $H$ in $G$, i.e.~$H=\coprod_j g_jG$.
We have to check that the natural transformations which determine the lax and oplax monoidal structure, are compatible with the half-braidings defined above. It is easiest to check this using the formulation of objects in $\cZ(\Rep_\Bbbk G)$ as Yetter--Drinfeld modules and amounts to a straightforward computation. 

Next, we  check that the lax and oplax monoidal structures displayed above indeed satisfy Equations \eqref{frobmon1}--\eqref{frobmon2}. For instance, Equation  \eqref{frobmon1} follows from the terminal expressions in the following lines being equal:
\begin{align*}
    (g\otimes v)\otimes (k\otimes (w\otimes u))&\mapsto (g\otimes v)\otimes (k\otimes w)\otimes (k\otimes u)\\
    &\mapsto \begin{cases}(g\otimes (v\otimes  g^{-1}kw))\otimes (k\otimes u),& \text{if $g^{-1}k\in H$,}\\ 0,& \text{otherwise}\end{cases}\\
    (g\otimes v)\otimes (k\otimes (w\otimes u))&\mapsto \begin{cases} g\otimes (v\otimes  g^{-1}kw\otimes g^{-1}ku), &\text{if $g^{-1}k\in H$,}\\ 0,& \text{otherwise}\end{cases}\\
    &\mapsto \begin{cases} (g\otimes (v\otimes  g^{-1}kw))\otimes (g\otimes g^{-1}ku),&\text{if $g^{-1}k\in H$,}\\ 0,& \text{otherwise.}\end{cases}
\end{align*}

Finally, we check compatibility with the braiding. The composition $\Psi_{\Ind_H^G(V),\Ind_H^G(W)}\circ\delta_{V,W}$ maps $g\otimes v\otimes w$ to 
$$g|v|g^{-1}g\otimes w\otimes g\otimes v=g|v|\otimes w\otimes g\otimes v=g\otimes |v|w\otimes g\otimes v,$$
using that $|v|\in H$, where $\delta(v)=|v|\otimes v$ is the $H$-coaction on $V$. Hence, this composition equals $\delta_{W,V}\circ\Ind_H^G(\Psi_{V,W})$. This proves Equation \eqref{braidcomp}. Equation \eqref{braidcompmu} follows by noting that 
$\Psi_{V,W}\circ\mu_{V,W}(g\otimes v\otimes k\otimes w)$ is zero unless $g^{-1}k\in H$, in which case it evaluates to 
$$g\otimes (|v|g^{-1}k w\otimes v)=g|v|g^{-1}k\otimes (w\otimes k^{-1}g|v|^{-1}v),$$
which is the image of $g\otimes v\otimes k\otimes w$ under the composition $\mu_{W,V}\circ \Ind_H^G(\Psi_{V,W}).$
\end{proof}

\begin{corollary}\label{cor:frob}
Let $H\subset G$ be finite groups.
\begin{enumerate}
\item
The functor $\cZ(\Ind_H^G)$ is exact and preserves duals. 
\item
The functor $\cZ(\Ind_H^G)$ preserves Frobenius algebra objects. 
\item The object $\Ind_H^G(V\otimes W)$ is naturally isomorphic to a direct summand of $\Ind_H^G(V)\otimes\Ind_H^G(W)$ in $\cZ(\Rep_\Bbbk G)$.
\end{enumerate}
\end{corollary}
\begin{proof}
The functor $\cZ(\Ind_H^G)$ is exact because it equals, on morphisms, the underlying functor $\Ind_H^G$, which is both left and right adjoint to restriction.  Further, a Frobenius monoidal functor preserves left and right duals and Frobenius algebras by \cite{DP}. Separability shows that the morphism $e_{V,W}=\delta_{V,W}\circ\mu_{V,W}$ is an idempotent, natural in $V,W$, that cuts out $\Ind_H^G(V\otimes W)$ as a direct summand of $\Ind_H^G(V)\otimes \Ind_H^G(W)$.
\end{proof}

\bibliography{biblio}
\bibliographystyle{amsrefs}%

\end{document}